\documentclass[11pt,reqno,a4paper]{amsart}
\usepackage{amsmath, amssymb, amsthm,amsfonts}
\usepackage{enumitem}
\usepackage{graphicx,color}
\usepackage{graphics}
\usepackage[OT2,T1]{fontenc}
\usepackage{tikz}
\usepackage{microtype}
\makeatletter
\@namedef{subjclassname@2020}{\textup{2020} Mathematics Subject Classification}
\makeatother
\allowdisplaybreaks[2]
\DeclareSymbolFont{cyrletters}{OT2}{wncyr}{m}{n}
\DeclareMathSymbol{\Sha}{\mathalpha}{cyrletters}{"58}

\DeclareMathOperator{\Hom}{Hom}

\DeclareMathOperator{\Span}{Span}
\DeclareMathOperator{\SL}{SL}
\DeclareMathOperator{\DSL}{\widetilde{SL}}
\DeclareMathOperator{\GL}{GL}

\def\ov{\overline}

\newcommand{\ord}{{\rm ord}}

\newcommand{\p}{{\mathfrak P}}

\newcommand{\Q}{{\mathbb Q}}
\newcommand{\Z}{{\mathbb Z}}

\newcommand{\C}{{\mathbb C}}

\newcommand{\kro}[2]{\left( \frac{#1}{#2} \right) }
\newcommand{\legendre}[2]{\ensuremath{\left( \frac{#1}{#2} \right) }}

\def\ld{\lambda}

\def\*{\times}

\def\b{\beta}

\def\d{\delta}

\newcommand{\hs}[1]{\left( #1 \right)_p}

\newcommand{\mat}[4]{\left(\begin{matrix} #1 & #2 \\ #3 & #4 \end{matrix}\right)}

            \DeclareFontFamily{U}{wncy}{}
            \DeclareFontShape{U}{wncy}{m}{n}{%
               <5>wncyr5%
               <6>wncyr6%
               <7>wncyr7%
               <8>wncyr8%
               <9>wncyr9%
               <10>wncyr10%
               <11>wncyr10%
               <12>wncyr6%
               <14>wncyr7%
               <17>wncyr8%
               <20>wncyr10%
               <25>wncyr10}{}
\DeclareMathAlphabet{\cyr}{U}{wncy}{m}{n}

\begin{document}

\newtheorem{thm}{Theorem}
\newtheorem*{thms}{Theorem}
\newtheorem{lem}{Lemma}[section]
\newtheorem{prop}[lem]{Proposition}
\newtheorem{alg}[lem]{Algorithm}
\newtheorem{cor}[lem]{Corollary}
\newtheorem{corlem}[lem]{Corollary}
\newtheorem{conj}[lem]{Conjecture}
\newtheorem{remark}{Remark}

\newtheorem{ex}{Example}

\theoremstyle{remark}

\title[Hecke Subalgebras and Local Newforms]{Hecke Subalgebras and Local Newforms for the Metaplectic Double Cover of $\SL_2(\mathbb Q_p)$}
\author{Ehud Moshe Baruch}
\author{Markos Karameris}
\author{Soma Purkait}

\begin{abstract}
Let $p$ be an odd prime, and let $\ov{\SL_2(\mathbb Z_p)}$ and $\ov{K_0(p^n)}$ denote the inverse images of $\SL_2(\mathbb Z_p)$ and the congruence subgroup $K_0(p^n)$ in the metaplectic double cover $\widetilde{\mathrm{SL}}_2(\mathbb Q_p)$. For $n\geq2$, we study the subalgebra of the genuine Hecke algebra $H(\widetilde{\mathrm{SL}}_2(\mathbb Q_p)//\ov{K_0(p^n)},\eta)$ consisting of functions supported in $\ov{\SL_2(\mathbb Z_p)}$, where $\eta$ is the genuine extension of either the trivial or the nontrivial quadratic character modulo $p$. We give an explicit basis and a presentation by generators and relations, and prove that this subalgebra is commutative of dimension $2n$. We give the multiplicity-free decomposition of $\operatorname{Ind}_{\ov{K_0(p^n)}}^{\ov{\SL_2(\mathbb Z_p)}}\eta$, determining the dimensions of its irreducible constituents and their associated Hecke eigenvalues.

This Hecke subalgebra acts on the $(\ov{K_0(p^n)},\eta)$-fixed spaces in irreducible admissible genuine representations of $\widetilde{\mathrm{SL}}_2(\mathbb Q_p)$. We compute its action explicitly on newvectors of prescribed $\eta$-type in principal series, Steinberg, even Weil, and supercuspidal representations. For supercuspidal representations, we construct the vectors by compact induction and use Ishimoto's conductor and dimension formulas.

Finally, we relate our operator $\mathcal W_{n-1}$ to Ishimoto's local realization of Ueda's twisting operator.
\end{abstract}

\subjclass[2020]{Primary 11F70; Secondary 11F37, 20C08, 22E50.}
\keywords{Metaplectic groups, Hecke algebras, local newforms,
$\eta$-conductors, half-integral-weight modular forms, supercuspidal
representations.}

\maketitle

\section{Introduction}

The theory of modular forms of half-integral weight is closely tied to the representation theory of the metaplectic double cover of $\SL_2$. Shimura's correspondence provides the basic bridge from half-integral to integral weight forms \cite{Shimura}, while Kohnen's plus space and newform theory isolate distinguished Hecke-stable subspaces on the half-integral-weight side \cite{KohnenPlus,Kohnen}. Ueda extended this theory to nonsquarefree odd level by introducing twisting operators and a corresponding decomposition of the Kohnen space \cite{Ueda1993,Ueda1998}. From the local point of view, these spaces should be reflected in genuine representations and genuine Hecke algebras of the metaplectic group. This makes it natural to seek operator-theoretic descriptions of newforms, rather than relying only on orthogonal complements.

In this paper, we develop the local Hecke algebra approach at
level $p^n$, where $p$ is an odd prime and $n\geq2$.
For the characters $\eta$ described in the abstract, we determine
generators and relations for the Hecke subalgebras supported in
$\overline{\SL_2(\mathbb Z_p)}$. These Hecke subalgebras are commutative of dimension $2n$. We classify the irreducible
finite-dimensional representations of this group containing
$(\overline{K_0(p^n)},\eta)$-fixed vectors, and compute the
Hecke action on newvectors of prescribed $\eta$-type in
irreducible admissible genuine representations of
$\widetilde{\mathrm{SL}}_2(\mathbb Q_p)$.

These local results are a first step towards our main objective:
to characterize newform subspaces of half-integral weight
modular forms of levels $4p^n$ and $8p^n$, for every positive
integer $n$, in terms of Hecke eigenspaces. More generally,
we aim to treat levels $4M$ and $8M$ with $M$ an arbitrary
odd positive integer. We plan to pursue this in subsequent work,
using the present results together with our earlier Hecke operator constructions and Ueda's classical newform theory \cite{Ueda1993,Ueda1998}.

A useful model for this approach is the integral-weight work of Baruch and Purkait \cite{B-P}. There, generators and relations for local Hecke algebras on $\GL_2(\mathbb Z_p)$ are used to characterize local newvectors and, after passage to classical operators, global newform spaces. In half-integral weight, the same philosophy was developed in \cite{B-PI,B-PII}. The paper \cite{B-PI} studies genuine Iwahori-type Hecke algebras for the double cover of $\SL_2(\mathbb Q_p)$ and uses their operators to characterize the minus-space counterpart to Kohnen's plus space at level $4M$. The theory at level $8M$ in \cite{B-PII} is built from the corresponding $2$-adic Hecke algebras and gives a further minus-space decomposition compatible with the relevant Shimura--Ueda correspondence. The representation-theoretic interpretation of the $2$-adic theory is also closely related to the work of Loke and Savin \cite{L-S}; more generally, unramified genuine representations of covering groups were studied by Savin \cite{Savin}.

Local newform theory for metaplectic representations has developed in parallel. Roberts and Schmidt introduced local oldforms and newforms for genuine representations of the metaplectic group and related the total number of newforms to the available Whittaker models \cite{R-S}. For $\SL_2$, the conductor and newform theory of Lansky and Raghuram \cite{L-R} is particularly useful for the linear representations underlying compactly induced metaplectic representations. Manderscheid's classification of genuine supercuspidal representations \cite{ManderscheidI,ManderscheidII}, together with the Kutzko--Sally description on the linear side \cite{KutzkoSally}, supplies the compact types used in the supercuspidal case. Ishimoto has given explicit conductor formulas and formulas for the dimensions of the corresponding spaces of fixed vectors in irreducible genuine representations of the rank-one metaplectic group, and studied their relation with the local theta correspondence \cite{Ishimoto}. We use his conductor and dimension formulas in our study of the Hecke action on newvectors of prescribed $\eta$-type.

We work throughout with a fixed quadratic character $\eta$ as in the abstract, and hence with the $\eta$-conductor
$$
c_\eta(\pi)=\min\{n:V_\eta^{K_0(p^n)}(\pi)\ne0\},
$$
rather than the minimum of $c_\eta(\pi)$ over all admissible characters. This is the conductor naturally seen by the Hecke subalgebra associated with the fixed character $\eta$. It is also the appropriate notion for the intended global application: under adelization, the character $\eta$ is induced by the local $p$-component of the nebentypus.

The Hecke subalgebra acts on the $(\ov{K_0(p^n)},\eta)$-fixed spaces in irreducible admissible genuine representations of $\widetilde{\mathrm{SL}}_2(\mathbb Q_p)$. We compute the action on explicit vectors in these spaces. For induced representations we write down the relevant isotypic vectors and compute their Hecke eigenvalues. The reducible principal-series cases lead to the even Weil and Steinberg representations, where the exact sequence relating these representations gives the dimensions of the type spaces and the corresponding operator actions.

For supercuspidal representations, we use Manderscheid's compact-induction classification \cite{ManderscheidI,ManderscheidII} in the explicit form recalled by Ishimoto, together with his description of the conductor and defect of the inducing types and his conductor and dimension formulas \cite{Ishimoto}. Our construction of newvectors adapts the ideas of Lansky--Raghuram \cite[\S3.3]{L-R} to the metaplectic setting. We then compute the Hecke action directly on these vectors by reducing it to explicit operators on the underlying linear inducing types. This makes explicit the distinction among the unramified packets of cardinality two and four and the ramified packets.

The first four sections of this paper had already been completed
when Ishimoto informed us of his manuscript
\emph{Kohnen--Ueda local newforms} \cite{IshimotoKU}.
Ishimoto defines the operators $U_{\varpi^2}$ and $R_\varpi$
on the tower of $\eta$-type spaces and uses their kernels and
images, together with level-changing maps, to define
Kohnen--Ueda old spaces and new quotients.
His approach and ours are complementary: we determine the
structure of Hecke subalgebras and compute their action directly
on explicit newvectors. In the supercuspidal case, his
calculations use a transfer to the Kirillov model, whereas our
construction of newvectors and computation of the Hecke action
remain in the compact-induction model.

His operator constructions motivated the comparison in
Section~5. Ishimoto's $R_p$ is the local unipotent average
corresponding, after normalization, to Ueda's quadratic twisting
operator. We show that its restriction to the space of
$(\ov{K_0(p^n)},\eta)$-fixed vectors, followed by a lifted
$\GL_2$-conjugation, gives the action of $\mathcal W_{n-1}$
up to an explicit scalar; when $n$ is odd, the conjugation
also interchanges the two quadratic types.
This compares the actions on the corresponding spaces of
fixed vectors, not the original operators, whose supports
and natural domains are different.

As mentioned above, the classical interpretation is one of the
principal motivations for the present local calculations, but we
do not pursue the global theory here. At the prime $2$, we would
use the $2$-adic Hecke operators developed in our earlier work at
levels $4M$ and $8M$ for $M$ odd and square-free \cite{B-PI,B-PII}. At the odd prime $p$, the
classical counterparts of the operators from the Hecke subalgebra
studied here should be compared with Ueda's twisting decomposition.
Our choice of a quadratic character $\eta$ is guided by these global
applications and by our earlier work on half-integral weight
forms with trivial or quadratic nebentypus. Under the Shimura correspondence, the
nebentypus $\chi$ becomes $\chi^2$, which is trivial in these cases.
Although we work over $\mathbb Q_p$ in order to keep the intended
classical applications and the normalizations explicit, the local
constructions are expected to extend, with appropriate changes
of notation and normalization, to non-archimedean local fields
of odd residual characteristic. In particular, the double-coset representatives and their
decompositions into right cosets, together with the residue-field
counts and Hensel lifting of square roots used in the Hecke
algebra computations, have direct analogues in this setting.

The paper is organized as follows. Section~2 introduces the Hecke subalgebra supported in $\ov{\SL_2(\mathbb Z_p)}$ and determines its double-coset support. Section~3 computes the relations among the generators of the
Hecke subalgebra and determines its characters and the corresponding irreducible
$\ov{\SL_2(\mathbb Z_p)}$-representations. Section~4 applies the algebra to local newforms, first for induced representations, then for special representations, and finally for compactly induced supercuspidal representations. Section~5 explains the relation among Ueda's twisting operator, Ishimoto's local operator $R_p$, and the operator $\mathcal W_{n-1}$ of this paper.

\section{A Hecke Subalgebra of $\widetilde G$ modulo $\ov{K_0(p^n)}$}\label{sec:hecke-subalgebra}

Let $p$ be a prime and put $G=\SL_2(\Q_p)$. We write $\widetilde G=\DSL_2(\Q_p)$ for the non-trivial central extension of $G$ by $\mu_2 = \{\pm1\}$:
\begin{equation*}
\begin{split}
1\ \longrightarrow \ &\mu_2 \ \longrightarrow \ \DSL_2(\Q_p)\
\longrightarrow\ \SL_2(\Q_p)\
\longrightarrow\ 1\\
\{(I,&\pm1)\}\ \quad (g, \pm1) \quad \longmapsto \quad g \qquad \qquad
\end{split}
\end{equation*}
with a group law determined by the following $2$-cocycle.
Let $\hs{\cdot,\cdot}$ be the Hilbert symbol over $\Q_p$.
For $A =\mat{a}{b}{c}{d} \in \GL_2(\Q_p)$, define
\[\tau(A) = \begin{cases}
         c & \text{if $c \ne 0$}\\
         d & \text{if $c = 0$}
        \end{cases}; \]
For $g=\mat{a}{b}{c}{d}\in\SL_2(\Q_p)$, set
\[s_p(g) = \begin{cases}
         \hs{c,d} & \text{if $cd \ne 0$ and $\ord_p(c)$ is odd}\\
         1 & \text{else}.
        \end{cases}
\]

For $g, h\in G$, define
$$
\sigma_0(g,h)
=
\left(
\frac{\tau(gh)}{\tau(g)},
\frac{\tau(gh)}{\tau(h)}
\right)_p
=
\bigl(\tau(gh)\tau(g),\tau(gh)\tau(h)\bigr)_p.
$$

Define the $2$-cocycle:
\begin{equation}\label{eq:cocylerel}
 c(g,h)=\sigma_0(g,h)s_p(g)s_p(h)s_p(gh).
\end{equation}

Thus $\widetilde G$ is the set $\SL_2(\Q_p) \times \mu_2$ with the group law given by
\[(g,\ \epsilon_1)(h,\ \epsilon_2) = (gh,\ \epsilon_1\epsilon_2c(g, h)).\]
For any subgroup $H$ of $\SL_2(\Q_p)$, we shall denote by $\ov{H}$ the complete
inverse image of $H$ in $\widetilde G$.

The $2$-cocycle $c$ used in the group law above is given by Gelbart; we note that in some references, for example in Manderscheid \cite{ManderscheidI, ManderscheidII}, Baruch--Mao \cite[pp.~247--248]{B-Mao}, and Ishimoto \cite{Ishimoto}, following Kubota, the group law uses the cocycle $\sigma_0$. To distinguish between the two, let $\widetilde G_c$ denote the realization of the double cover defined by $c$ and let
$\widetilde G_{\sigma_0}$ denote the realization defined by $\sigma_0$.  The two are explicitly isomorphic through
$$
\iota:\widetilde G_c\longrightarrow\widetilde G_{\sigma_0},
\qquad
\iota((g,\epsilon))=(g,\epsilon s_p(g)).
$$
Indeed, the relation~\eqref{eq:cocylerel} gives
$$
\iota((g,\epsilon)(h,\epsilon'))
=
\iota((g,\epsilon))\iota((h,\epsilon')).
$$
We henceforth identify $\widetilde G$ with our $c$-model $\widetilde G_c$.

Later in Subsection~\ref{subsec:supercuspidal} and Section~\ref{sec:Ueda-Ishimoto-Rp}, we will use the isomorphism $\iota$ when comparing results stated in the two cocycle models.

Let $K=\SL_2(\Z_p)$. By \cite[Proposition 2.8]{Gelbart} for odd primes $p$,
$\widetilde G$ splits over $K$. Thus
$\ov{K}$ is a group isomorphic to the direct product $K \times \mu_2$. We will consider the following subgroups of $K$:
\[K_0(p^n) =\left\{ \mat{a}{b}{c}{d} \in \SL_2(\Z_p)\ :\ c \in p^n\Z_p \right\},\]
\[K_1(p^n) =\left\{ \mat{a}{b}{c}{d} \in \SL_2(\Z_p)\ :\ c \in p^n\Z_p,\
a \equiv 1 \pmod{p^n\Z_p} \right\}.\]
For brevity we shall write $K_n:=K_0(p^n)$ throughout the paper; the subgroup $K_1(p^n)$ will always be written with its argument.
For odd primes $p$, $\ov{K_0(p^n)}$ is isomorphic to $K_0(p^n) \times \mu_2$ and $\ov{K_1(p^n)}$ is isomorphic to $K_1(p^n) \times \mu_2$.

For the rest of the paper we take $p$ to be an odd prime. In this case, by \cite[Corollary 2.13]{Gelbart}, the center $M_p$ of $\widetilde G$ is the direct product $\{\pm I\} \times \mu_2$. Thus a genuine central character is a character of $\{\pm I\}\times\mu_2$ that is non-trivial on the covering-kernel $\mu_2$ factor.

For a nontrivial additive character $\psi'$ of $\mathbb Q_p$,
its conductor exponent $c(\psi')$ is the integer $r$ such that
$\psi'$ is trivial on $p^r\mathbb Z_p$ but nontrivial on
$p^{r-1}\mathbb Z_p$.
For a smooth multiplicative character $\chi$ of
$\mathbb Q_p^\times$ or $\mathbb Z_p^\times$, we put
$c(\chi)=0$ if $\chi$ is trivial on $\mathbb Z_p^\times$;
otherwise, $c(\chi)$ is the least integer $r\geq1$ such that
$\chi$ is trivial on $1+p^r\mathbb Z_p$.

We also fix the Weil-factor conventions used below.  For a nontrivial
additive character $\psi'$ of $\Q_p$, let $\gamma(\psi')$ denote the Weil
index of the one-dimensional quadratic form $x\mapsto x^2$ with respect to
$\psi'$.  For $a\in\Q_p^\times$, put $\psi'_a(x)=\psi'(ax)$ and define the
relative Weil factor by
$$
\gamma(a,\psi')=\frac{\gamma(\psi'_a)}{\gamma(\psi')}.
$$
Note that \(c(\psi'_a)=c(\psi')-\operatorname{ord}_p(a)\).
We use the following standard identities; see, for example,
\cite{RangaRao}:
\begin{enumerate}
    \item $\gamma(ac^2,\psi')=\gamma(a,\psi')$,
    \item $\gamma(ab,\psi')=\gamma(a,\psi')\gamma(b,\psi')(a,b)_p$,
    \item $\gamma(a,\psi'_c)=(a,c)_p\gamma(a,\psi')$,
    \item $\gamma(-1,\psi')=\gamma(\psi')^{-2}$,
    \item $\gamma(ac,\psi')=\gamma(c,\psi')$ for any $a\in\Z_p^{\times}$ and $c\in\Q_p^{\times}$ such that $\operatorname{ord}_p(c)\equiv c(\psi')\pmod 2$.
\end{enumerate}
Fix a nontrivial additive character $\psi$ of conductor $0$. We put
$$
\epsilon_p:=\gamma(p,\psi).
$$
The factor $\epsilon_p$ depends on the fixed additive character $\psi$;
we suppress this dependence in the notation.

We identify $\mathbb Z_p/p\mathbb Z_p$ with $\mathbb F_p$.
Whenever representatives are needed in $p$-adic or classical formulas,
we use the integers $0,\ldots,p-1$, and $1,\ldots,p-1$ for
$\mathbb F_p^\times$.

Let \(\chi_p=\legendre{\cdot}{p}\), with \(\chi_p(0)=0\), and let \(\sqrt p\) denote the
positive square root of \(p\). Since \(\psi\) has conductor \(0\),
the map \(t\mapsto\psi(t/p)\) is a well-defined nontrivial additive
character of \(\mathbb F_p\). Its quadratic Gauss sum has the two equivalent expressions %
\begin{equation}\label{eq:weil-gauss-normalization}
 \sum_{t\in\mathbb F_p^*}\chi_p(t)\psi(t/p)
 =\sum_{x\in\mathbb F_p}\psi(x^2/p)
 =\epsilon_p\sqrt p.
\end{equation}
The explicit nonarchimedean Weil-index
formula gives
\[
 \gamma(\psi)=1,
 \qquad
 \gamma(\psi_p)=p^{-1/2}
 \sum_{x\in\mathbb F_p}\psi(x^2/p);
\]
see \cite[Theorem A.2(5), Theorem A.3(1)--(2), and Proposition A.11,
pp.~366, 369--370]{RangaRao}. Since the number of square roots of
\(t\in\mathbb F_p\) is \(1+\chi_p(t)\), summing first over the values
of \(x^2\), and using the vanishing of a nontrivial additive-character
sum, proves the first equality in \eqref{eq:weil-gauss-normalization}.
The displayed Weil-index formulas and the definition of $\epsilon_p$
give the second equality. In particular,
\(\epsilon_p^2=\chi_p(-1)\). We put
\[
 p^*:=\chi_p(-1)p
 \qquad\text{and choose}\qquad
 \sqrt{p^*}:=\epsilon_p\sqrt p.
\]
For the special choice satisfying $\psi(t/p)=e^{2\pi it/p}$ for
$t\in\mathbb Z$, one has $\epsilon_p=1$ if $p\equiv1\pmod4$ and
$\epsilon_p=i$ if $p\equiv3\pmod4$; see
\cite[Proposition~A.9(5)]{RangaRao}. These are the values used in
\cite{B-PI}. We do not impose this additional normalization on $\psi$ here.

We set up a few more notations.
For $s \in \Q_p$, $t \in \Q_p^\times$ let us define the following elements of
$\SL_2(\Q_p)$:
\[x(s) =\mat{1}{s}{0}{1},\ y(s)=\mat{1}{0}{s}{1},\
w(t) = \mat{0}{t}{-t^{-1}}{0},\  h(t)=\mat{t}{0}{0}{t^{-1}}.\]
We also use the canonical lifts
\[
\begin{aligned}
\bar x(s)&:=(x(s),1), & \bar y(s)&:=(y(s),1),\\
\bar h(t)&:=(h(t),1), & \bar w(t)&:=(w(t),1).
\end{aligned}
\]
Let $B_0\subset G$ be the standard upper-triangular Borel subgroup.  Let $N=\{(x(s), \epsilon): s \in \Q_p,\ \epsilon = \pm 1\}$,
$\bar{N}=\{(y(s), \epsilon): s \in \Q_p,\ \epsilon = \pm 1\}$
and $T=\{(h(t), \epsilon): t \in \Q_p^\times,\ \epsilon = \pm 1\}$
be the subgroups of $\widetilde G$. Then the normalizer
$N_{\widetilde G}(T)$ of $T$ in $\widetilde G$ consists of
elements $(h(t), \epsilon)$, $(w(t), \epsilon)$ for $t \in \Q_p^\times$. We also put $B=NT=\operatorname{pr}^{-1}(B_0)$. Thus $B$ denotes the full inverse image of the linear Borel $B_0$.  In the notation of Baruch--Mao, our $B_0$ is their $B_S$ and our $B$ is their $\overline{B}_S$.

Let $n\ge 2$. Following \cite{B-P}, we are interested in the subalgebra of the genuine Hecke algebra of $\widetilde G$ modulo $\ov{K_n}$ with quadratic characters whose elements are supported on $\ov{K}$. We define it below.

Let $\eta$ be a character of $K_n$ such that it is
trivial on $K_1(p^n)$.
Since $\frac{K_n}{K_1(p^n)} \cong (\Z_p/p^n\Z_p)^\times$,
where the quotient map is
$\mat{a}{b}{c}{d}\mapsto d\pmod{p^n\Z_p}$,
the character $\eta$ is induced by a character of
$(\Z/p^n\Z)^\times$. We use the same symbol $\eta$ for its genuine
extension to $\ov{K_n}$, defined by
$\eta((A, \epsilon)) = \epsilon\eta(A)$ for $A \in K_n$.
Recall that a genuine character acts nontrivially on the $\mu_2$ component.

Define
$H_n(\eta):=H(\ov{K}// \ov{K_n}, \eta)$ to be the space of functions
$f\in C_c^{\infty}(\widetilde G)$ supported on $\ov K$ and satisfying
\[
f(\tilde{k}\tilde{g}\tilde{k'})
=\overline{\eta}(\tilde{k})\overline{\eta}(\tilde{k'})f(\tilde{g})\
\text{for }\tilde{g}\in\widetilde G,\
\tilde{k},\tilde{k'}\in\ov{K_n}.
\]
This is a $\C$-algebra under the usual convolution: for
$f_1, f_2 \in H_n(\eta)$,
\[ f_1*f_2(\tilde{h}) =
\int_{\widetilde G} f_1(\tilde{g})f_2(\tilde{g}^{-1}\tilde{h}) d\tilde{g} =
\int_{\widetilde G} f_1(\tilde{h}\tilde{g})f_2(\tilde{g}^{-1}) d\tilde{g}, \]
where $d\tilde{g}$ is the Haar measure on $\widetilde G$ such that the measure
of $\ov{K_n}$ is one.

We take $\eta$ to be a quadratic character for the rest of the paper. Note that $\eta$ quadratic character of $(\Z/p^n\Z)^\times$ implies that $\eta$ is either trivial or given by $\kro{\cdot}{p}$. Let $c(\eta)$ denote the conductor exponent of the underlying character
of $\mathbb Z_p^\times$: it is $0$ for the trivial character and $1$
for the nontrivial quadratic character modulo $p$. At level $n=0$,
we take $\eta((k,\epsilon))=\epsilon$ for $k\in K_0=K$.

Let $(\pi,V)$ be a smooth genuine representation of $\widetilde G$.
For every integer $n\geq c(\eta)$, its $(\ov{K_n},\eta)$-fixed space is
\[
V_\eta^{K_n}(\pi)
=\{v\in V:\pi(\tilde k)v=\eta(\tilde k)v
\text{ for every }\tilde k\in\ov{K_n}\}.
\]
This space is also called the $\eta$-isotypic subspace for
$\ov{K_n}$, or the $(\ov{K_n},\eta)$-isotypic space.
Its vectors are said to be of $(\ov{K_n},\eta)$-type,
or simply of prescribed $\ov{K_n}$-type or $\eta$-type
when the character or the subgroup, respectively, is understood.
Although the superscript is $K_n$,
the defining condition is imposed on its full inverse image
$\ov{K_n}$. We also write $\pi_\eta^{K_n}$ for
$V_\eta^{K_n}(\pi)$ when the representation space is clear
from the context.
The $\eta$-conductor of $\pi$ is
\[
c_\eta(\pi)
=\min\{n\geq c(\eta):V_\eta^{K_n}(\pi)\ne0\},
\]
with $c_\eta(\pi)=\infty$ if this set is empty.

For $F\in H_n(\eta)$ and $v\in V$, define the Hecke action by
\[
\pi(F)v:=\int_{\widetilde G}F(\tilde g)\pi(\tilde g)v\,d\tilde g,
\]
using the Haar measure fixed above; compare \cite[\S3]{L-S}.
 By the left equivariance of $F$, this operator maps all of $V$
into the space of $(\ov{K_n},\eta)$-fixed vectors:
\[
\pi(F):V\longrightarrow V_\eta^{K_n}(\pi).
\]
Moreover, $\pi(F_1*F_2)=\pi(F_1)\pi(F_2)$.
Hence $V_\eta^{K_n}(\pi)$ is a left $H_n(\eta)$-module.
The same construction applies to the full genuine Hecke algebra
$H(\widetilde G//\ov{K_n},\eta)$, where the support need not lie
in $\ov K$. We write $Fv$ for $\pi(F)v$ when the representation
is understood.

If $\operatorname{supp}(F)=\bigsqcup_i\tilde g_i\ov{K_n}$, then for
$v\in V_\eta^{K_n}(\pi)$ the normalization
$\operatorname{vol}(\ov{K_n})=1$ gives
\[
\pi(F)v=\sum_i F(\tilde g_i)\pi(\tilde g_i)v.
\]
Indeed, on each right coset the factors $\eta(\tilde k)^{-1}$ from
$F(\tilde g_i\tilde k)$ and $\eta(\tilde k)$ from
$\pi(\tilde g_i\tilde k)v$ cancel.
\vskip 2mm

We would like to describe $H_n(\eta)$ using generators and relations.
From now on we will write $H=H_n(\eta)$ when the level and character are understood, and use the same notation at any other level at which $\eta$ is defined.

\begin{prop}\label{prop:rep1}
A complete set of double coset representatives of
$\ov{\SL_2(\Z_p)}$ modulo $\ov{K_n}$ are given by
$ (I,1),\ (w(1),1),\  (y(p^r),1),\ (y(\ld p^r),1)$
where $r$ runs from $1$ to $n-1$ and $\ld$ is a nonsquare modulo $p$.
\end{prop}

The support condition is especially simple for these representatives.  Since
the chosen cover splits over $K$, we identify $\ov K$ with
$K\times\mu_2$ for the following calculation.  If $g\in K$, then the
double coset $\ov{K_n}(g,1)\ov{K_n}$ supports a nonzero
$\eta$-biequivariant function if and only if
\begin{equation}\label{eq:compact-support-criterion}
 \eta(k)=\eta(g^{-1}kg)
 \qquad\text{for every }k\in K_n\cap gK_ng^{-1}.
\end{equation}
Indeed, if $k$ belongs to this intersection, then
$kg=g(g^{-1}kg)$, and the two applications of biequivariance must give
the same value.  Conversely, condition
\eqref{eq:compact-support-criterion} makes the value prescribed at $(g,1)$
independent of the choice of its left--right expression.  There is no
additional cocycle factor, since all the elements in this calculation lie
in the split compact subgroup.

\begin{prop}\label{prop:compact-support}
 Every double coset listed in Proposition~\ref{prop:rep1} supports a
 nonzero element of $H$.
\end{prop}

\begin{proof}
The assertion is immediate for $g=I$.  Let
$g=w(1)$ and let $k=\mat{a}{b}{c}{d}$ belong to
$K_n\cap gK_ng^{-1}$.  Then both $b$ and $c$ belong to $p^n\Z_p$, and
\[
 g^{-1}kg=\mat{d}{-c}{-b}{a}.
\]
Since $ad-bc=1$, we have $ad\equiv1\pmod p$.  Hence
$\eta(a)=\eta(d^{-1})=\eta(d)$, where the last equality uses that
$\eta$ is quadratic.  Thus \eqref{eq:compact-support-criterion} holds.

Next let $g=y(t)$, where $t=\delta p^r$ with
$\delta\in\Z_p^\times$ and $1\leq r\leq n-1$.  For
$k=\mat{a}{b}{c}{d}\in K_n\cap gK_ng^{-1}$, direct multiplication gives
\[
 g^{-1}kg=
 \mat{a+bt}{b}{c+t(d-a)-t^2b}{d-tb}.
\]
Since $t\in p\Z_p$, its lower-right entry is congruent to $d$ modulo
$p$.  The character $\eta$ is either trivial or the quadratic character
modulo $p$, and therefore $\eta(d-tb)=\eta(d)$.  This verifies
\eqref{eq:compact-support-criterion} for both square classes
$\delta=1$ and $\delta=\ld$, and proves the proposition.
\end{proof}

We denote by $\mathcal I$ and $\mathcal U_0$ the elements of $H$ supported
on $(I,1)$ and $(w(1),1)$, respectively, normalized by
$\mathcal I((I,1))=1$ and $\mathcal U_0((w(1),1))=1$.  For
$1\leq r\leq n-1$ and $a\in\Z_p^\times$, let $\mathcal V_{r,a}$ denote
the normalized Hecke function supported on
$\ov{K_n}(y(ap^r),1)\ov{K_n}$ and satisfying
$\mathcal V_{r,a}((y(ap^r),1))=1$.

The function $\mathcal V_{r,a}$ depends only on the square class of $a$.
Indeed, if $a=bs^{-2}$ with $s\in\Z_p^\times$, then
$$
(y(ap^r),1)=(h(s),1)(y(bp^r),1)(h(s)^{-1},1).
$$
The two type-character factors associated with $(h(s),1)$ and
$(h(s)^{-1},1)$ cancel, and hence
$\mathcal V_{r,a}=\mathcal V_{r,b}$.  Consequently, the two possibilities
are represented by $\mathcal V_{r,1}$ and $\mathcal V_{r,\ld}$.

Note that since $\eta$ is a character of
$\frac{K_n}{K_1(p^n)}$, $\eta(h(u), \epsilon)=
\eta(u^{-1})\epsilon$.

We next compute the relations using the notation $K_n=K_0(p^n)$ fixed above.

We will use the following lemma which easily follows from the definition of convolution.
\begin{lem} \label{lem:rel1}
Let $f_1,\ f_2 \in H$ such that $f_1$ is supported on
$S\tilde{x}S=\bigcup_{i=1}^{m}\tilde{\alpha}_i S = \bigcup_{i=1}^{m}S \tilde{\gamma}_i $ and
$f_2$ is supported on
$S\tilde{y}S=\bigcup_{j=1}^{n}\tilde{\beta}_j S= \bigcup_{j=1}^{n} S \tilde{\delta}_j$. Then
\[
f_1*f_2(\tilde{h})=
\sum_{i=1}^{m}f_1(\tilde{\alpha}_i)f_2(\tilde{\alpha}_{i}^{-1}\tilde{h})=
\sum_{j=1}^{n}f_1(\tilde{h}\tilde{\delta}_j^{-1})f_2(\tilde{\delta}_j)\]
where in the first sum the nonzero summands are precisely for those $i$ for which there exist
a $j$ such that $\tilde{h}\in \tilde{\alpha}_i \tilde{\beta}_j S$ while in the second sum the nonzero summands are precisely for those $j$ for which there exist
a $i$ such that $\tilde{h}\in S \tilde{\gamma}_i \tilde{\delta}_j$.
\end{lem}
\begin{proof} We prove the second expression.
\begin{align*}
 f_1*f_2(\tilde{h}) =
\int_{\widetilde G} f_1(\tilde{h}\tilde{g})f_2(\tilde{g}^{-1}) d\tilde{g}.
\end{align*}
 We have $\tilde{g}^{-1}\in S\tilde{y}S$ if and only if
$\tilde{g}\in (S\tilde{y}S)^{-1}=\bigcup_{j=1}^{n}\tilde{\delta}_j^{-1}S$. Hence
\begin{align*}
 f_1*f_2(\tilde h)
 &=\sum_{j=1}^n\int_{\tilde\delta_j^{-1}S}
 f_1(\tilde h\tilde g)f_2(\tilde g^{-1})\,d\tilde g \\
 &=\sum_{j=1}^n f_1(\tilde h\tilde\delta_j^{-1})f_2(\tilde\delta_j).
\end{align*}
The nonzero summands are precisely those for which
$\tilde h\tilde\delta_j^{-1}\in S\tilde xS=\bigcup_{i=1}^mS\tilde\gamma_i$, equivalently those $j$ for which there is an $i$ with
$\tilde h\in S\tilde\gamma_i\tilde\delta_j$.

\end{proof}

In the following lemma we express the double cosets with our coset representatives into a union of single cosets.
\begin{lem}\label{lem:cosetdecomp}\
 \begin{enumerate}
  \item $\displaystyle K_n w(1) K_n= \bigcup_{s\in \Z_p / p^n\Z_p}
x(s)w(1)K_n=\bigcup_{s\in \Z_p / p^n\Z_p}K_nw(1)x(s)$.
\item Let $\ld\in \Z_p^\times$ and $1\le r\le n-1$.  Then
\[
\begin{aligned}
K_ny(\ld p^r)K_n
&=\bigcup_{u \in \Z_p^{*}/\sqrt{1+p^{n-r}\Z_p}}
h(u)y(\ld p^{r})K_n\\
&=\bigcup_{u \in \Z_p^{*}/\sqrt{1+p^{n-r}\Z_p}}
K_ny(\ld p^{r})h(u).
\end{aligned}
\]
 \end{enumerate}
\end{lem}
\begin{proof}
 (1) follows using triangular decomposition of $K_n$. \\
 For (2) we use the same idea as \cite[Lemmas 3.7, 3.9]{B-P}.
 Recall that $\mat{a}{b}{c}{d}\in K_n^{y(\ld p^r)}$ %
 provided the matrix is in $K_n$ and $a-d+bp^r\ld \in p^{n-r}\Z_p$. Any element $g$ of the double coset is of the form $x(s)h(u)y(\ld p^r)K_n= h(u)x(u^{-2}s)y(\ld p^r)K_n$. If $r\geq n/2$, $x(u^{-2}s)\in K_n^{y(\ld p^r)}$ and so $g$ is of the form $h(u)y(\ld p^r)K_n$. Now $h(u)\in K_n^{y(\ld p^r)}$ if and only if $u^2-1 \in p^{n-r}\Z_p$. Thus $g$ is of the form $h(u)y(\ld p^r)K_n$ where $u \in \Z_p^{*}/\sqrt{1+p^{n-r}\Z_p}$. Further, $y(-\ld p^{r})h(u_1u_2^{-1})y(\ld p^r)\in K_n \iff u_1^{-1}u_2\in \sqrt{1+p^{n-r}\Z_p}$, thus the union is disjoint.

 \par\smallskip\noindent
Now assume that $1\le r< n/2$. As before, any element $g$ of the double coset can be written as $g=h(u)x(s)y(\ld p^r)k$ for  $s\in \Z_p$, $u$ a unit and $k\in K_n$. Rewrite
 $g= h(u)h(1+ pv)^{-1}h(1+ pv)x(s)y(\ld p^r)k$ where $v\in\Z_p$ such that $h(1+ pv)x(s) \in K_n^{y(\ld p^r)}$; this is possible by taking $v=\frac{\sqrt{(1+s\ld p^r)^{-1}}-1}{p}$ (we take the square root in $1+p^r\Z_p$ so that $v\in \Z_p$).
Thus $g$ is of the form $h(u')y(\ld p^r)K_n$ and we can follow the argument as above.

 \par\smallskip\noindent
A similar argument gives the right coset decomposition.

 \end{proof}

\noindent{\bf Remark.}
We record a correction to the proof of \cite[Lemma 3.9]{B-P}.  This
concerns the integral-weight setting of that paper, where
$K_0=K_0(p^n)\subset\GL_2(\Z_p)$ and
$d(s)=\mat{s}{0}{0}{1}$.  Put $t=p^r$ and write the diagonal parameter
appearing at the beginning of that proof as $s_0$.  The assertion there
that $d(1+tu)x(u)$ belongs to
$K_0^{y(t)}=y(t)K_0y(-t)\cap K_0$ is not valid in general.  Instead, set
\[
 s=(1+ut)^{-1}\in\Z_p^\times.
\]
Then direct multiplication gives
\[
 y(-t)d(s)x(u)y(t)
 =\mat{s(1+ut)}{su}{t\{1-s(1+ut)\}}{1-sut}
 =\mat{1}{su}{0}{1-sut}\in K_0.
\]
Consequently $d(s)x(u)y(t)=y(t)k_1$ for some $k_1\in K_0$, and hence
\[
 \begin{aligned}
 g
 &=d(s_0)d(s)^{-1}d(s)x(u)y(t)k_0\\
 &=d(s_0s^{-1})y(t)k_1k_0.
 \end{aligned}
\]
This proves the required coset decomposition; the disjointness argument in
\cite[Lemma 3.9]{B-P} is unchanged.

We note the following simple observation.
\begin{lem}\label{lem:sqrt}
 We have $\sqrt{1+p^r\Z_p}= (1+p^r\Z_p) \bigcup (-1+p^r\Z_p)$.
 In particular the coset representatives of $\Z_p^*/ \sqrt{1+p^r\Z_p}$ are given by $x_0 + x_1 p + \cdots + x_{r-1}p^{r-1}$ where $x_0 \in \{1,2,\ldots, (p-1)/2\}$ and $x_1, \ldots, x_{r-1} \in \{0,1,\ldots, p-1\}$ and so
  $\# \left(\Z_p^*/ \sqrt{1+p^r\Z_p}\right)= \frac{p-1}{2}p^{r-1}$.
\end{lem}
\begin{proof}
 Let $a\in \sqrt{1+p^r\Z_p}$. Then $a^2\in 1+ p^r\Z_p$. In particular $a=\pm 1$ modulo $p$. Let $a= 1+x_1 p +x_2 p^2+ \cdots + x_r p^r$ with $x_r\in \Z_p$ and $x_1, \ldots, x_{r-1} \in \{0,1,\ldots, p-1\}$. Then $a^2=1+2x_1 p$ modulo $p^2\Z_p$ which implies that $x_1=0$. Continuing this argument, we get that $a\in 1+p^r\Z_p$. If $a\equiv -1 \pmod{p\Z_p}$, the same argument will imply that $a\in (-1+p^r\Z_p)$.
\end{proof}

\section{Relations}

The algebra $H=H_n(\eta)=H(\ov{K}//\ov{K_n},\eta)$, where $\eta$ is a quadratic character modulo $p$ (either the trivial character or $\kro{\cdot}{p}$), has the following $2n$ basis elements:
\[
\mathcal I,\quad \mathcal U_0,\quad
\mathcal V_{1,1},\ldots,\mathcal V_{n-1,1},\quad
\mathcal V_{1,\ld},\ldots,\mathcal V_{n-1,\ld}.
\]
Here $\ld$ is a nonsquare modulo $p$, and $\mathcal I$ is the identity element.

Throughout this section all the matrices occurring in the convolution
calculations belong to $K$.  We therefore use the fixed splitting over $K$
and identify $\ov K$ with $K\times\mu_2$.  In particular, products and
inverses of the canonical lifts appearing below introduce no additional
cocycle factor.

We will use the following well-known result.
\begin{lem}\label{lem:modp}
 Let $p$ be an odd prime, and let $a,b$ be integers such that $p\nmid ab$.
 Then the number of solutions $(x,y)$ to
 $ax^2+by^2\equiv1\pmod p$ is $p-\kro{-ab}{p}$.
 \end{lem}

\begin{prop}\label{prop:V}
 Suppose $p\equiv 1 \pmod{4}$. Then
 \begin{align*}
  \mathcal{V}_{r,1}*\mathcal{V}_{r,1} &= \frac{p-1}{2}p^{n-r-1}\mathcal{Y}_{r+1}+\frac{p-5}{4}p^{n-r-1}\mathcal{V}_{r,1}+ \frac{p-1}{4}p^{n-r-1}\mathcal{V}_{r,\ld} \\
  \mathcal{V}_{r,\ld}*\mathcal{V}_{r,\ld} &= \frac{p-1}{2}p^{n-r-1}\mathcal{Y}_{r+1}+\frac{p-1}{4}p^{n-r-1}\mathcal{V}_{r,1}+ \frac{p-5}{4}p^{n-r-1}\mathcal{V}_{r,\ld}\\
  \mathcal{V}_{r,1}*\mathcal{V}_{r,\ld} &= \mathcal{V}_{r,\ld}*\mathcal{V}_{r,1} = \frac{p-1}{4} p^{n-r-1}( \mathcal{V}_{r,1} + \mathcal{V}_{r,\ld}).
  \end{align*}

  Suppose $p\equiv 3 \pmod{4}$. Then
 \begin{align*}
  \mathcal{V}_{r,1}*\mathcal{V}_{r,1} &= \frac{p-3}{4}p^{n-r-1}\mathcal{V}_{r,1}+ \frac{p+1}{4}p^{n-r-1}\mathcal{V}_{r,\ld} \\
  \mathcal{V}_{r,\ld}*\mathcal{V}_{r,\ld} &= \frac{p+1}{4}p^{n-r-1}\mathcal{V}_{r,1}+ \frac{p-3}{4}p^{n-r-1}\mathcal{V}_{r,\ld}\\
  \mathcal{V}_{r,1}*\mathcal{V}_{r,\ld} &= \mathcal{V}_{r,\ld}*\mathcal{V}_{r,1} = \frac{p-1}{2}p^{n-r-1}\mathcal{Y}_{r+1} + \frac{p-3}{4}p^{n-r-1} ( \mathcal{V}_{r,1} + \mathcal{V}_{r,\ld}).
  \end{align*}
Where in both cases $\mathcal{Y}_r=\mathcal{I}+\sum\limits_{j=r}^{n-1}\mathcal{V}_j$ with $\mathcal{V}_j=\mathcal{V}_{j,1}+\mathcal{V}_{j,\ld}$ and $\mathcal{Y}_n=\mathcal I$.
\end{prop}

\begin{proof}
Let $\d, \mu \in \{1,\ld\}$.
It follows from Lemmas~\ref{lem:rel1} and \ref{lem:cosetdecomp} that $\mathcal{V}_{r,\d}* \mathcal{V}_{r,\mu}$ is supported on those $(g,1) \in \bar{K}$ for which there exist $u,\ v \in \Z_p^*/ \sqrt{1+p^{n-r}\Z_p}$
such that
\[(h(u)y(\d p^r)h(v)y(\mu p^r))^{-1}g = \mat{(uv)^{-1}}{0}{- p^r u^{-1}(\d v + \mu v^{-1})}{uv}g \in K_n.\]
We check this condition for $g =I,\ w(1),\ y(p^j), y(\ld p^j)$ with $j<r$ and $j\geq r$.

\par\smallskip\noindent
For $g=w(1)$, the condition clearly does not hold for any value of $u$, $v$, so no support.

\par\smallskip\noindent
For $g=I$, the condition becomes
$v^2 + \mu \d^{-1} \equiv 0 \pmod{p}$. So if $\mu=\d$, such $v$ exists if and only if $p\equiv 1 \pmod{4}$, while if $\mu\ne\d$, such $v$ exists if and only if $p\equiv 3 \pmod{4}$.

\par\smallskip\noindent
For $g=y(p^j\xi)$ with $\xi\in\{1,\ld\}$ and $j<r$ the condition becomes $uvp^j\xi - p^r u^{-1}(\d v + \mu v^{-1})\equiv 0\pmod{p^n} \iff uv\xi - p^{r-j} u^{-1}(\d v + \mu v^{-1})\equiv 0\pmod{p^{n-j}}$ which is obviously never satisfied.
\par\smallskip\noindent
For $g= y(p^r)$, the condition is equivalent to
$\mu^{-1}(uv)^2- \d\mu^{-1}v^2 \equiv 1\pmod{p}$ by Hensel's Lemma. By Lemma~\ref{lem:modp}, the number of solutions to $\mu^{-1}x^2- \d\mu^{-1}y^2 \equiv 1\pmod{p}$ equals $p-\kro{\d}{p}$. For support, we are interested in existence of such $(x,y)$ with $x,y$ both units. If $\d=\mu=1$ then the  number of such $(x,y)$ is $p-5$ if $p\equiv 1 \pmod{4}$, while it is $p-3$ if $p\equiv 3 \pmod{4}$.
If $\d=\mu=\ld$ then the number of such $(x,y)$ is $p-1$ if $p\equiv 1 \pmod{4}$, while it is $p+1$ if $p\equiv 3 \pmod{4}$.
If $\mu\ne \d$, then the number of such $(x,y)$ is $p-1$ if $p\equiv 1 \pmod{4}$, while it is $p-3$ if $p\equiv 3 \pmod{4}$.

\par\smallskip\noindent
For $g= y(\ld p^r)$, the condition is equivalent to
$\ld \mu^{-1} (uv)^2- \d\mu^{-1}v^2 \equiv 1\pmod{p}$. By Lemma~\ref{lem:modp}, the number of solutions to $\ld\mu^{-1}x^2- \d\mu^{-1}y^2 \equiv 1\pmod{p}$ equals $p-\kro{\ld\d}{p}$. As before we are only interested in solutions $(x,y)$ with $x,y$ both units. If $\d=\mu=1$ then the  number of such $(x,y)$ is $p-1$ if $p\equiv 1 \pmod{4}$, while it is $p+1$ if $p\equiv 3 \pmod{4}$.
If $\d=\mu=\ld$ then the number of such $(x,y)$ is $p-5$ if $p\equiv 1 \pmod{4}$, while it is $p-3$ if $p\equiv 3 \pmod{4}$.
If $\mu\ne \d$, then the number of such $(x,y)$ is $p-1$ if $p\equiv 1 \pmod{4}$, while it is $p-3$ if $p\equiv 3 \pmod{4}$.

\par\smallskip\noindent
For $g=y(p^j\xi)$ with $\xi\in\{1,\ld\}$ and $j>r$ the condition becomes $(u^2p^{j-r}\xi - \d) v^2 - \mu \equiv 0\pmod{p^{n-r}} $ which has a solution if and only if $\d v^2+\mu\equiv 0\pmod p$. This implies as before that either $\mu=\d$ and $p\equiv 1\pmod 4$ or $\mu\ne\d$ and $p\equiv 3\pmod 4$.

Thus, for $p\ge 7$, $\mathcal{V}_{r,\d}* \mathcal{V}_{r,\d}$ is supported on $(I,1)$, $(y(p^j),1)$, $(y(\ld p^j),1)$ with $j\geq r$ if $p\equiv 1\pmod{4}$, while  supported on  $(y(p^r),1)$, $(y(\ld p^r),1)$ if $p\equiv 3\pmod{4}$. On the other hand, for $p\ge 7$, $\mathcal{V}_{r,\d}* \mathcal{V}_{r,\mu}$, for $\d\ne \mu$, is supported on $(y(p^r),1)$ and $(y(\ld p^r),1)$ if $p\equiv 1\pmod{4}$, while it is supported on $(I,1)$ and $(y(p^j),1)$, $(y(\ld p^j),1)$ with $j\geq r$ if $p\equiv 3\pmod{4}$. For $p=3$, $\mathcal{V}_{r,1}* \mathcal{V}_{r,1}$ is only supported on $(y(\ld p^r),1)$, $\mathcal{V}_{r,\ld}* \mathcal{V}_{r,\ld}$ is only supported on $(y(p^r),1)$ while $\mathcal{V}_{r,\d}* \mathcal{V}_{r,\mu}$, for $\d\ne \mu$ is supported on $(I,1)$ and $(y(p^j),1)$, $(y(\ld p^j),1)$ with $j\geq r+1$. For $p=5$, $\mathcal{V}_{r,1}* \mathcal{V}_{r,1}$ is supported on $(I,1)$, $(y(\ld p^r),1)$ and on $(y(p^j),1)$, $(y(\ld p^j),1)$ with $j\geq r+1$; $\mathcal{V}_{r,\ld}* \mathcal{V}_{r,\ld}$ is supported on $(I,1)$, $(y(p^r),1)$ and on $(y(p^j),1)$, $(y(\ld p^j),1)$ with $j\geq r+1$; while $\mathcal{V}_{r,\d}* \mathcal{V}_{r,\mu}$, for $\d\ne \mu$, is supported on $(y(p^r),1)$ and $(y(\ld p^r),1)$. Thus, the support agrees with the relations in the statement.

\par\medskip\noindent
Now we compute the values of $\mathcal{V}_{r,\d}* \mathcal{V}_{r,\mu}$ on the coset representatives where they are supported. 
\par\smallskip\noindent
We have
\begin{align*}
\mathcal{V}_{r,\d}*\mathcal{V}_{r,\mu}((g,1))
&=\sum_{u \in \Z_p^{*}/\sqrt{1+p^{n-r}\Z_p}}
\mathcal{V}_{r,\d}((h(u),1)(y(\d p^r),1))\\
&\hspace{2.2cm}\cdot
\mathcal{V}_{r,\mu}((y(-\d p^r),1)(h(u^{-1}),1)(g,1))\\
&=\sum_{u \in \Z_p^{*}/\sqrt{1+p^{n-r}\Z_p}}
\ov{\eta}(u^{-1})\\
&\hspace{2.2cm}\cdot
\mathcal{V}_{r,\mu}((y(-\d p^r),1)(h(u^{-1}),1)(g,1)).
\end{align*}

By the splitting over $K$,
\[
 (y(-\d p^r),1)(h(u^{-1}),1)
 =\left(\mat{u^{-1}}{0}{-\d p^ru^{-1}}{u},1\right),
\]
so
\[\mathcal{V}_{r,\d}* \mathcal{V}_{r,\mu} ((g,1)) = \sum_{u \in \Z_p^{*}/\sqrt{1+p^{n-r}\Z_p}} \ov{\eta}(u^{-1})\mathcal{V}_{r,\mu}((\mat{u^{-1}}{0}{-\d p^ru^{-1}}{u}, 1)(g,1)).\]

\par\smallskip\noindent
Let $g=I$ and write $k_1=\mat{u^{-1}}{0}{-\d p^ru^{-1}}{u}$. We seek
$A,k_2\in K_n$ such that
\[
(k_1,1)=(A^{-1},1)(y(\mu p^r),1)(k_2,1).
\]

We find $A=\mat{a}{b}{c}{d} \in K_n$ such that
$y(-\mu p^r)\mat{a}{b}{c}{d}\mat{u^{-1}}{0}{-\d p^ru^{-1}}{u} \in K_n$, i.e.,
\[\mat{au^{-1}-b\d p^ru^{-1}}{bu}{-p^r u^{-1} (\mu a + d \d)+cu^{-1}+p^{2r} \mu \d bu^{-1}}{du} \in K_n. \]
This boils down to finding $a$ such that $ a^2 + \d \mu^{-1} \equiv 0\pmod{p}$.

Let $\d=\mu$. If $p\equiv 3 \pmod{4}$, no such $a$ exist. If $p\equiv 1 \pmod{4}$, take $a=\sqrt{-1}$, $d=1/a$, $b=c=0$. For this choice of $A$, $k_2= h(au^{-1})$. Thus,
\begin{align*}
\mathcal{V}_{r,\d}*\mathcal{V}_{r,\mu}((I,1))
&=\sum_{u \in \Z_p^{*}/\sqrt{1+p^{n-r}\Z_p}}
\ov{\eta}(u^{-1})\mathcal{V}_{r,\mu}((k_1,1))\\
&=\sum_{u \in \Z_p^{*}/\sqrt{1+p^{n-r}\Z_p}}
\ov{\eta}(u^{-1})\\
&\hspace{1.5cm}\cdot
\mathcal{V}_{r,\mu}((h(1/\sqrt{-1}),1)(y(\mu p^r),1)
(h(u^{-1}\sqrt{-1}),1))\\
&=\sum_{u \in \Z_p^{*}/\sqrt{1+p^{n-r}\Z_p}}
\ov{\eta}(u^{-1})\ov{\eta}(u)\\
&=\#\{u\in \Z_p^{*}/\sqrt{1+p^{n-r}\Z_p}\}
=\frac{p-1}{2}p^{n-r-1}.
\end{align*}

Let $\d\ne \mu$. So $\d\mu^{-1}$ is a non-square modulo $p$.  If $p\equiv 1 \pmod{4}$, no such $a$ exist. If $p\equiv 3 \pmod{4}$, take $a=\sqrt{-\d\mu^{-1}}$, $d=1/a$, $b=c=0$; in this case, as before,
$\mathcal{V}_{r,\d}* \mathcal{V}_{r,\mu} ((I,1)) =
\frac{p-1}{2}p^{n-r-1}$.

\par\smallskip\noindent
Let $g=y(\b p^r)$ where $\b \in \{1,\ld\}$. Then,
\begin{align*}
 (y(-\d p^r),1)(h(u^{-1}),1)(y(\b p^r),1)
 &= \left(\mat{u^{-1}}{0}{-\d p^ru^{-1}}{u}, 1\right)(y(\b p^r),1)\\
 &=\left(\mat{u^{-1}}{0}{p^ru^{-1}\b (u^2-\frac{\d}{\b})}{u}, 1\right).
\end{align*}

Thus,
\[
\begin{aligned}
\mathcal{V}_{r,\d}*\mathcal{V}_{r,\mu}((y(\b p^r),1))
&=\sum_{u \in \Z_p^{*}/\sqrt{1+p^{n-r}\Z_p}}
\ov{\eta}(u^{-1})\\
&\quad\cdot
\mathcal{V}_{r,\mu}\!\left(
\left(\mat{u^{-1}}{0}
{p^ru^{-1}\b (u^2-\frac{\d}{\b})}{u},1\right)\right).
\end{aligned}
\]

\par\smallskip\noindent
For $u \in \Z_p^*/ \sqrt{1+p^{n-r}\Z_p}$, put
$v=u^{-1}\b (u^2-\frac{\d}{\b})$.  We want $A=\mat{a}{b}{c}{d}\in K_n$ such that
\[
y(-\mu p^r)A\mat{u^{-1}}{0}{-\d p^ru^{-1}+\b p^r u}{u}
=
\begin{pmatrix}
au^{-1}+bp^rv & bu\\
* & du-\mu p^rbu
\end{pmatrix}
\in K_n,
\]
where
\[
*=-\mu p^r(au^{-1}+bp^rv)+cu^{-1}+p^rdv.
\]
This simplifies to finding $a$, $u$ units in $\Z_p$ such that
\[-\frac{\mu}{\d}a^2+ \frac{\b}{\d}u^2=1 \pmod{p^{n-r}}.\]
Suppose for some $u$, we have a solution $a$ mod $p$ as above, then we can lift this to $a\in \Z_p$ and take $A= h(a)$ so that the matrix above becomes $h(au^{-1})$.
For such $u$ we verify that
\[(\mat{u^{-1}}{0}{p^ru^{-1}\b (u^2-\frac{\d}{\b})}{u}, 1)
= (h(a^{-1}),1)(y(\mu p^r),1)(h(au^{-1}),1),\]
and so the contribution from this $u$ in the sum for $\mathcal{V}_{r,\d}* \mathcal{V}_{r,\mu} ((y(\b p^r),1))$ would be  $\ov{\eta}(u^{-1})\ov{\eta}(a)\ov{\eta}(a^{-1}u)$ which equals $1$.

\par\smallskip\noindent
So, we need to just count the number of solutions $(a,u)$ with $a$, $u$ units to the above congruence modulo $p^{n-r}$. We know however the number of solutions $\pmod{p}$ by Lemma \ref{lem:modp}. Since for $u$ there are four possibilities $(\pm a , \pm u)$ and
$u \in \Z_p^*/ \sqrt{1+p\Z_p}$, we will divide by $4$ to get the contribution.

\par\smallskip\noindent
The counting is similar to the counting that we did while checking the support, so we will illustrate it only for one case. By Lemma~\ref{lem:modp}, the  total number of $(a,u)$ mod $p$ (not necessarily units) is $p-\kro{\mu \b}{p}$. Let $\mu=\b$ and $p\equiv 1 \pmod{4}$. In this case, solutions of the type $(\pm *, 0)$ exists if and only if $\mu = \d$. Also,  solutions of the type $(0, \pm *)$ exists if and only if $\b = \d$. So, if $\b=\mu=\d$, then the total number of solutions $(a,u)$ with both $a$, $u$ units equals $p-1-4 =p-5$. For each $u\in \Z_p^{*}/\sqrt{1+p^{n-r}\Z_p}$ however that satisfies $-\frac{\mu}{\d}a^2+ \frac{\b}{\d}u^2=1 \pmod{p}$ we can find a solution $a\in\Z_p$ and thus there are $\#\{u\in \Z_p^{*}/\sqrt{1+p^{n-r}\Z_p}|\exists a\in \Z_p: -\frac{\mu}{\d}a^2+ \frac{\b}{\d}u^2=1 \pmod{p}\}=\frac{p-5}{4}p^{n-r-1}$ solutions in total.
Hence, $\mathcal{V}_{r,1}*\mathcal{V}_{r,1}((y(p^r),1))= \frac{p-5}{4}p^{n-r-1}$ if $p\equiv 1 \pmod{4}$. The other cases follow similarly.
\par\smallskip\noindent
For $g=(y(p^j\b),1)$ with $j>r$ and $\beta\in\{1,\ld\}$ we want that  $A=\mat{a}{b}{c}{d}\in K_n$ such that
\[y(-\mu p^r)A\mat{u^{-1}}{0}{-\d p^ru^{-1}+\b p^j u}{u}\in  K_n.\] This is equivalent to the condition $-\mu a+(d-\mu p^rb)(\b p^{j-r}u^2-\d)\equiv 0\pmod{p^{n-r}}$ which is equivalent to lifting a solution of $a^2+\d\mu^{-1}\equiv 0\pmod{p}$. The rest of the proof now follows as in the case of $g=(I,1)$.
\end{proof}

\begin{prop} \label{prop:ViVj}
    \[\mathcal{V}_{r,\mu}*\mathcal{V}_{j,\delta}=\mathcal{V}_{j,\delta}*\mathcal{V}_{r,\mu}=\frac{p-1}{2}p^{n-j-1}\mathcal{V}_{r,\mu}, \text{ for } j>r.\]
\end{prop}
\begin{proof}
     Assume $j>r$, then by Lemmas~\ref{lem:rel1} and \ref{lem:cosetdecomp}, $\mathcal{V}_{r,\mu}*\mathcal{V}_{j,\delta}$ is supported on the $(g,1)$ for which there exist $u\in \Z_p^{*}/\sqrt{1+p^{n-r}\Z_p}$ and $v\in \Z_p^{*}/\sqrt{1+p^{n-j}\Z_p}$ such that
     \[
     \mat{(uv)^{-1}}{0}{-p^ru^{-1}v\mu-p^j(uv)^{-1}\delta}{uv}g\in K_n.
     \]
     It is obvious then that for $g=w(1)$ there is no support.

\par\smallskip\noindent
For $g=y(p^i\b)$ we want equivalently $uvp^i\beta - p^ru^{-1}v\mu-p^j(uv)^{-1}\d\in p^n\Z_p$. For $i\ne r$ it is easy to check that this is impossible. For $i=r$ the condition becomes $uv\b-u^{-1}v\mu-p^{j-r}(uv)^{-1}\d\in p^{n-r}\Z_p$. This is equivalent to finding a solution to $u^2-\b^{-1}\mu\equiv 0\pmod{p}$ which exists if and only if $\b=\mu$. We thus calculate the value for $g=y(p^r\mu)$.
\par\smallskip\noindent
As before
\[
\begin{aligned}
\mathcal{V}_{r,\mu}*\mathcal{V}_{j,\delta}((y(p^r\mu),1))
&=\sum_{u \in \Z_p^{*}/\sqrt{1+p^{n-r}\Z_p}}
\ov{\eta}(u^{-1})\\
&\quad\cdot
\mathcal{V}_{j,\d}((y(-p^r\mu)h(u^{-1})y(p^r\mu),1)).
\end{aligned}
\]
We want $A=\mat{a}{b}{c}{d}\in K_n$ such that
\[y(-\d p^j)A\mat{u^{-1}}{0}{\mu p^r(u-u^{-1})}{u}\in K_n.\]
A direct multiplication shows that this is equivalent to
\[
\mu(d-p^jb\d)(u-u^{-1})-\d p^{j-r}a u^{-1}\in p^{n-r}\Z_p.
\]
Thus necessarily $u-u^{-1}=p^{j-r}\zeta$ with $\zeta\in\Z_p^\times$. Taking $b=0$ and using $ad=1$, the remaining congruence is
\[
\mu d^2\zeta u\equiv \d\pmod{p^{n-j}}.
\]
Hence, writing $\zeta=(u-u^{-1})/p^{j-r}$, the contribution is
\begin{align*}
\mathcal{V}_{r,\mu}*\mathcal{V}_{j,\delta}((y(p^r\mu),1))
&=\#\left\{u\in \Z_p^{*}/\sqrt{1+p^{n-r}\Z_p}\;\middle|\;
\zeta\in\Z_p^\times, u\in \frac{\mu\zeta}{\d}(\Z_p^\times)^2\right\}\\
&=\frac{p-1}{2}p^{n-j-1}.
\end{align*}
\par\smallskip\noindent
From the second part of Lemma~\ref{lem:rel1},
\begin{align*}
\mathcal V_{j,\d}*\mathcal V_{r,\mu}((y(p^r\mu),1))
&=\sum_{u\in\Z_p^*/\sqrt{1+p^{n-r}\Z_p}}
\ov\eta(u^{-1})\\
&\quad\cdot
\mathcal V_{j,\d}((y(p^r\mu)h(u)y(-p^r\mu),1))\\
&=\mathcal V_{r,\mu}*\mathcal V_{j,\delta}((y(p^r\mu),1)).
\end{align*}
\end{proof}
\begin{prop} \label{prop:U}
 If $\eta=1$ then with notation as before
 \[\mathcal{U}_0*\mathcal{U}_0= p^n\mathcal{Y}_1 + p^{n-1}(p-1)\mathcal{U}_0.\]
 If $\eta\ne 1$ then
 \[\mathcal{U}_0*\mathcal{U}_0= \eta(-1)p^n\mathcal{Y}_1.\]
\end{prop}
\begin{proof}
 By Lemmas~\ref{lem:rel1} and \ref{lem:cosetdecomp},
 \begin{align*}
\mathcal{U}_0*\mathcal{U}_0((g,1))
&=\sum_{s \in \Z_p/p^n\Z_p}
\mathcal{U}_0((x(s),1)(w(1),1))\\
&\hspace{2.3cm}\cdot
\mathcal{U}_0((w(-1),1)(x(-s),1)(g,1))\\
&=\sum_{s \in \Z_p/p^n\Z_p}
\mathcal{U}_0((\mat{0}{-1}{1}{-s},1)(g,1)).
\end{align*}
where the support of $\mathcal{U}_0* \mathcal{U}_0$ is on those $(g,1)$  for which there exist $s, t\in \Z_p$ such that $y(t)x(-s)g\in K_n$. It is easy to check that for $g=I, w(1), y(\b p^r)$ some such $s, t$ exist.

\par\smallskip\noindent
Let $g=I$. Note that $(\mat{0}{-1}{1}{-s},1)= (w(1),1) (\mat{-1}{s}{0}{-1},1)$. Hence, $\mathcal{U}_0* \mathcal{U}_0 ((I,1))=p^n\eta(-1)$.

\par\smallskip\noindent
Let $g=w(1)$. Then $(\mat{0}{-1}{1}{-s}, 1)(g,1)= (y(s),1)$. To compute $\mathcal{U}_0* \mathcal{U}_0$ at $(w(1),1)$, for $s\in \Z_p/p^n\Z_p$, we want $\mat{a}{b}{c}{d}\in K_n$ such that $w(-1) \mat{a}{b}{c}{d} y(s) = \mat{-c-ds}{-d}{a+bs}{b}\in K_n$. Clearly $s$ must be in $\Z_p^*$ and for such $s$, choosing $a=d=1, c=0, b=-1/s$, the above matrix becomes $\mat{-s}{-1}{0}{-1/s}$.
Then
\begin{align*}
\mathcal U_0*\mathcal U_0((w(1),1))
&=\sum_{s\in\Z_p/p^n\Z_p}\mathcal U_0((y(s),1))\\
&=\sum_{\substack{s\in\Z_p/p^n\Z_p\\(s,p)=1}}
\mathcal U_0\!\left((x(1/s),1)(w(1),1)
\left(\mat{-s}{-1}{0}{-1/s},1\right)\right)\\
&=\sum_{\substack{s\in\Z_p/p^n\Z_p\\(s,p)=1}}\ov\eta(-1/s)
=\begin{cases}
 p^n-p^{n-1},&\eta=1,\\
 0,&\eta=\kro{\cdot}{p}.
\end{cases}
\end{align*}

\par\smallskip\noindent
Let $g=y(\b p^r)$ where $\b \in \{1,\ld\}$. Then
\[
(\mat{0}{-1}{1}{-s},1)(g,1)
=
(\mat{-\b p^r}{-1}{1-\b p^rs}{-s},1).
\]
Doing computations as above, we find that
\begin{align*}
\mathcal U_0*\mathcal U_0((y(\b p^r),1))
&=\sum_{s\in\Z_p/p^n\Z_p}
\mathcal U_0\!\left(
\left(
\left(\begin{smallmatrix}
-\b p^r&-1\\
1-\b p^rs&-s
\end{smallmatrix}\right),1
\right)\right)\\
&=\sum_{s\in\Z_p/p^n\Z_p}
\mathcal U_0\!\left(
\begin{gathered}
\left(x\!\left(\frac{-\b p^r}{1-\b p^rs}\right),1\right)\\[-1mm]
{}\cdot(w(1),1)\\[-1mm]
{}\cdot\left(
\left(\begin{smallmatrix}
-(1-\b p^rs)&s\\
0&-1/(1-\b p^rs)
\end{smallmatrix}\right),1\right)
\end{gathered}\right)\\
&=\sum_{s\in\Z_p/p^n\Z_p}
\ov\eta\!\left(\frac{-1}{1-\b p^rs}\right)\\
&=p^n\eta(-1).
\end{align*}
\end{proof}

\begin{prop} \label{prop:UV}
Let $\d \in \{1, \ld\}$. Then
 \[\mathcal{U}_0*\mathcal{V}_{r,\d} = \frac{p-1}{2}p^{n-r-1}\mathcal{U}_0 = \mathcal{V}_{r,\d}*\mathcal{U}_0.\]
\end{prop}
\begin{proof}
We use Lemmas~\ref{lem:rel1} and \ref{lem:cosetdecomp}.
For $\mathcal{U}_0*\mathcal{V}_{r,\d}$, we use the first expression in Lemma~\ref{lem:rel1} while for $\mathcal{V}_{r,\d}*\mathcal{U}_0$, we use the second sum.
 For $s\in \Z_p, t\in \Z_p^*$, note that $(x(s)w(1)h(t)y(\d p^r))^{-1}= \mat{0}{-t^{-1}}{t}{\d p^r t^{-1}-ts}$ while $(y(\d p^r)h(t)w(1)x(s))^{-1}= \mat{\d p^r t -t^{-1}s}{-t}{t^{-1}}{0}$. It follows that $\mathcal{U}_0*\mathcal{V}_{r,\d}$ and $\mathcal{V}_{r,\d}*\mathcal{U}_0$ are only supported on $(w(1),1)$.

 We first compute $\mathcal{U}_0*\mathcal{V}_{r,\d}((w(1),1))$. For $s\in \Z_p/p^n\Z_p$, we want $\mat{a}{b}{c}{d}\in K_n$ such that $y(-\d p^r) \mat{a}{b}{c}{d} y(s) \in K_n$. So, we must have  $(d-\d p^rb)s-\d p^ra \equiv 0 \pmod{p^n}$. So, $s$ must be of the form $p^ru$ where $u$ is a unit such that $u/\d$ is a square modulo $p$. Note that there are $\frac{p-1}{2}p^{n-r-1}$ such $s$ in $\Z_p/p^n\Z_p$.

 For each such $s=p^ru$, take $a=\sqrt{u/\d}$, $d=1/a$, $b=c=0$. Then,
\begin{align*}
\mathcal{U}_0*\mathcal{V}_{r,\d}((w(1),1))
&=\sum_{s \in \Z_p/p^n\Z_p}
\mathcal{U}_0((x(s),1)(w(1),1))\\
&\hspace{2.2cm}\cdot
\mathcal{V}_{r,\d}((w(-1),1)(x(-s),1)(w(1),1))\\
&=\sum_{s \in \Z_p/p^n\Z_p}
\mathcal{V}_{r,\d}((y(s),1))\\
&=\sum_{\substack{s \in \Z_p/p^n\Z_p\\
s=p^ru:\ \kro{u/\d}{p}=1}}
\mathcal{V}_{r,\d}\bigl((h(\sqrt{\d/u}),1)(y(\d p^r),1)\\
&\hspace{5.0cm}\cdot(h(\sqrt{u/\d}),1)\bigr)\\
&=\frac{p-1}{2}p^{n-r-1}.
\end{align*}

 We also have however
 \begin{align*}
\mathcal{V}_{r,\d}*\mathcal{U}_0((w(1),1))
&=\sum_{s \in \Z_p/p^n\Z_p}
\mathcal{V}_{r,\d}((w(1),1)(x(-s),1)(w(-1),1))\\
&\hspace{2.2cm}\cdot
\mathcal{U}_0((w(1),1)(x(s),1))\\
&=\sum_{s \in \Z_p/p^n\Z_p}\mathcal{V}_{r,\d}((y(-s),1))
=\frac{p-1}{2}p^{n-r-1}.
\end{align*}
 where the last equality follows as before.

\end{proof}

\begin{thm}
The algebra $H(\ov{K}//\ov{K_n},\eta)$ is a $2n$-dimensional commutative algebra
with generators $\{\mathcal{I},\ \mathcal{U}_0,\ \mathcal{V}_{i,1},\ \mathcal{V}_{i,\ld}\}$ with $1\leq i\leq n-1$ where
$\kro{\ld}{p}=-1$ and relations given by Propositions~\ref{prop:V}, ~\ref{prop:ViVj}, \ref{prop:U}, \ref{prop:UV}.
\end{thm}

Define $\mathcal{W}_r:=\mathcal{V}_{r,1}- \mathcal{V}_{r,\ld}$.
It follows from Propositions ~\ref{prop:V}, ~\ref{prop:ViVj}, \ref{prop:U} and \ref{prop:UV} that
\begin{cor} \label{cor:VW}
\begin{enumerate}
 \item $\mathcal{V}_r^2 = (p-2)p^{n-r-1}\mathcal{V}_r + (p-1)p^{n-r-1}\mathcal{Y}_{r+1}$.
 \item $\mathcal{V}_r\mathcal{V}_j=(p-1)p^{n-j-1}\mathcal{V}_r=\mathcal{V}_j\mathcal{V}_r$ where $r<j$.
 \item $\mathcal{W}_r^2 = \kro{-1}{p}p^{n-r-1}((p-1)\mathcal{Y}_{r+1}-\mathcal{V}_r)$.
 \item $\mathcal{W}_r\mathcal{W}_j=0$ for every $r\ne j$.
 \item $\mathcal{V}_r\mathcal{W}_r = \mathcal{W}_r\mathcal{V}_r=-p^{n-r-1}\mathcal{W}_r$.
 \item $\mathcal{V}_r\mathcal{W}_j = \mathcal{W}_j\mathcal{V}_r=\begin{cases}
 0
 & \text{ if \ }r<j \\
 p^{n-r-1}(p-1)\mathcal{W}_j & \text{ if }r>j
 \end{cases}$.
 \item  $\mathcal{U}_0\mathcal{V}_r= p^{n-r-1}(p-1)\mathcal{U}_0 = \mathcal{V}_r\mathcal{U}_0 $.
 \item  $\mathcal{U}_0\mathcal{W}_r= 0 = \mathcal{W}_r\mathcal{U}_0 $.
 \item For $\eta=1$, $\mathcal{U}_0^2=p^n\mathcal{Y}_1+p^{n-1}(p-1)\mathcal{U}_0$; for
 $\eta=\kro{\cdot}{p}$, $\mathcal{U}_0^2=\eta(-1)p^n\mathcal{Y}_1$.
 \item With $p^*=\kro{-1}{p}p$ as fixed in Section~2, the terminal operator satisfies
 $\mathcal{W}_{n-1}^3=p^*\mathcal{W}_{n-1}$; equivalently, it satisfies
 $X(X^2-p^*)=0$.
\end{enumerate}
\end{cor}

Indeed, at the terminal index $r=n-1$, we have
$\mathcal{Y}_n=\mathcal{I}$ and
$\mathcal{Y}_{n-1}=\mathcal{I}+\mathcal{V}_{n-1}$; hence parts~(3) and~(5)
give
\[
\mathcal{W}_{n-1}^2
=\kro{-1}{p}\bigl(p\mathcal{I}-\mathcal{Y}_{n-1}\bigr),
\qquad
\mathcal{Y}_{n-1}\mathcal{W}_{n-1}=0.
\]
Multiplying the first identity by $\mathcal{W}_{n-1}$ proves part~(10).

\begin{cor}
 The algebra $H(\ov{K}//\ov{K_n},\eta)$ is a $2n$-dimensional commutative algebra
with generators $\mathcal{I},\mathcal{U}_0,\mathcal{V}_r,\mathcal{W}_r$
for $1\leq r\leq n-1$, and relations given by
Proposition~\ref{prop:U} and Corollary~\ref{cor:VW}.
\end{cor}

It follows from the above corollary that
\[
\begin{aligned}
\mathcal{Y}_r\mathcal{Y}_j
&=p^{n-j}\mathcal{Y}_r=\mathcal{Y}_j\mathcal{Y}_r,
&&\forall j\geq r,\\
\mathcal{U}_0\mathcal{Y}_r
&=p^{n-r}\mathcal{U}_0=\mathcal{Y}_r\mathcal{U}_0,
&&\forall r\in\{1,\dots,n-1\},
\end{aligned}
\]
\[
\mathcal{Y}_r\mathcal{W}_j
=\mathcal{W}_j\mathcal{Y}_r
=
\begin{cases}
p^{n-r}\mathcal{W}_j,&\text{if }r>j,\\
0,&\text{if }r\leq j,
\end{cases}
\]
\[
\begin{aligned}
\mathcal{W}_j\mathcal{W}_r&=0
&&\text{for any }r\ne j,\\
\mathcal{W}_r\mathcal{U}_0
&=\mathcal{U}_0\mathcal{W}_r=0
&&\forall r\in\{1,\dots,n-1\},
\end{aligned}
\]
\[
\mathcal{W}_r^2
=\kro{-1}{p}p^{n-r-1}
(p\mathcal{Y}_{r+1}-\mathcal{Y}_r),
\]
\[
\mathcal{U}_0^2=p^n\mathcal{Y}_1+p^{n-1}(p-1)\mathcal{U}_0
\quad\text{if }\eta=1,
\qquad
\mathcal{U}_0^2=\eta(-1)p^n\mathcal{Y}_1
\quad\text{if }\eta=\kro{\cdot}{p}.
\]

\begin{remark}\label{rem:Y-averaging}
The elements $\mathcal Y_r$ admit an averaging interpretation. For
$1\leq r\leq n$, the character $\eta$ extends from $\ov{K_n}$ to
$\ov{K_r}$ by the same formula as in Section~2, since its underlying
character factors modulo $p$. The double cosets $\ov{K_n}$ and
$\ov{K_n}(y(\delta p^j),1)\ov{K_n}$, with $r\leq j\leq n-1$
and $\delta\in\{1,\lambda\}$, partition $\ov{K_r}$.
At each chosen representative, both $\eta$ and the corresponding
normalized Hecke function have value $1$. Biequivariance and the
multiplicativity of $\eta$ on $\ov{K_r}$ therefore give
\[
\mathcal Y_r(\tilde g)=
\begin{cases}
\eta(\tilde g)^{-1},&\tilde g\in\ov{K_r},\\
0,&\tilde g\notin\ov{K_r}.
\end{cases}
\]
With $\operatorname{vol}(\ov{K_n})=1$, we have
$\operatorname{vol}(\ov{K_r})=[K_r:K_n]=p^{n-r}$.
Thus, on the $(\ov{K_n},\eta)$-fixed space of a smooth genuine
representation $(\pi,V)$, the operator
\[
p^{r-n}\pi(\mathcal Y_r)v
=\frac{1}{\operatorname{vol}(\ov{K_r})}
\int_{\ov{K_r}}\eta(\tilde k)^{-1}\pi(\tilde k)v\,d\tilde k
\]
is the projection onto the $(\ov{K_r},\eta)$-fixed space.
\end{remark}

\subsection{Representations of $\ov{K}$ having a vector of prescribed $\ov{K_n}$-type}
In \cite{B-P}, we studied representations of $\GL_2(\Z_p)$ containing an
$H_0(p^n)$-fixed vector, where $H_0(p^n)$ consists of matrices in
$\GL_2(\Z_p)$ which reduce to upper-triangular matrices modulo $p^n\Z_p$.

We would like to study the irreducible representations of
$\ov{K}$ having a nonzero vector of $(\ov{K_n},\ov\eta)$-type. Let
\[
I(n):=Ind^{\ov{K}}_{\ov{K_n}}\ov{\eta}
=\left\{\phi:\ov K\rightarrow\C\;\middle|\;
\begin{array}{l}
\phi(k_0k)=\ov{\eta}(k_0)\phi(k),\\[-1mm]
\forall k_0\in\ov{K_n},\ k\in\ov K
\end{array}
\right\}.
\]
$\phi$'s are smooth (locally constant).
Then $I(n)$ is a right representation of $\ov K$ via right translation, 
and the dimension of this representation is $[\ov{K}:\ov{K_n}] = p^{n-1}(p+1)$.

The set
\[
\begin{aligned}
I(n)_{\ov\eta}^{K_n}
=\{\phi\in I(n):\;&\phi(kk_0)=\ov{\eta}(k_0)\phi(k),\\
&\forall k_0\in\ov{K_n},\ k\in\ov K\}
\end{aligned}
\]
is the collection of vectors of $(\ov{K_n},\ov\eta)$-type in $I(n)$ and
hence equals the Hecke algebra $H$. Consequently
$\dim I(n)_{\ov\eta}^{K_n}=2n$.

By Frobenius reciprocity,
\[
\begin{aligned}
\Hom_{\ov K}(I(n),I(n))
&\cong \Hom_{\ov{K_n}}
\bigl(\ov\eta,I(n)|_{\ov{K_n}}\bigr)\\
&\cong I(n)_{\ov\eta}^{K_n}\cong H.
\end{aligned}
\]

Since $\ov K$ is compact, $I(n)$ is completely reducible.  Moreover,
$H\cong\operatorname{End}_{\ov K}(I(n))$ is commutative and has dimension
$2n$.  It follows that $I(n)$ is a multiplicity-free direct sum of $2n$
irreducible representations.  Any irreducible smooth representation of
$\ov K$ containing a nonzero vector of $(\ov{K_n},\ov\eta)$-type is
isomorphic to a subrepresentation of $I(n)$. We describe these
subrepresentations and their exact-level type vectors.
\begin{cor}
    Let $n>1$. There exist exactly two irreducible constituents
    $\sigma(n),\sigma'(n)$ of $I(n)$ that do not occur in $I(n-1)$.
    Each contains, up to scalar multiple, a unique vector of
    $(\ov{K_n},\ov\eta)$-type and no nonzero vector of
    $(\ov{K_{n-1}},\ov\eta)$-type. At level $1$, the two irreducible
    constituents of $I(1)$ have dimensions summing to $p+1$, while for
    $n\geq2$ one has
    $\dim(\sigma(n))+\dim(\sigma'(n))=p^{n-2}(p^2-1)$.
\end{cor}
If we restrict ourselves to $\eta\equiv 1$, then
let $S$ be the subalgebra of $H$ generated by $\mathcal{I}, \mathcal{U}_0, \mathcal{Y}_r$ for every $r\in\{1,\dots,n-1\}$. These elements satisfy exactly the same relations as elements of the subalgebra $H(\GL_2(\Z_p) // H_0(p^n))$ and hence are isomorphic. The algebra $H$ is generated by $S$ together with $\mathcal W_1,\dots,\mathcal W_{n-1}$.

We have a left action $\pi_L$ of $H$ on $I(n)$ given by $\pi_L(f)(\phi) = f * \phi$ for $f\in H$, $\phi\in I(n)$, 
in fact $\pi_L(f) \in \Hom_{\ov{K}}(I(n),I(n))$ and so acts by scalar on irreducible components of $I(n)$.

With the notation $p^*=\kro{-1}{p}p$ introduced above, as in
\cite[Proposition 3.17]{B-P}, a basis of eigenvectors of $H$ under
$\pi_L$ consists of
\[
\begin{array}{c@{\qquad\qquad}c}
\eta\equiv 1 & \eta\not\equiv 1 \\[2pt]
 u_1=\mathcal U_0+\mathcal Y_1
 & u_1=\mathcal U_0+\sqrt{p^*}\mathcal Y_1 \\[2pt]
 u_2=\mathcal U_0-p\mathcal Y_1
 & u_2=\mathcal U_0-\sqrt{p^*}\mathcal Y_1
\end{array}
\]
and, in both cases,
\[
\begin{aligned}
v_r&=p\mathcal Y_{r+1}-\mathcal Y_r
       +\frac{p}{\sqrt{p^*}}\mathcal W_r,\\
v'_r&=p\mathcal Y_{r+1}-\mathcal Y_r
       -\frac{p}{\sqrt{p^*}}\mathcal W_r,
\qquad 1\leq r\leq n-1.
\end{aligned}
\]
For $\eta\equiv1$, the eigenvalues are as follows:
\begin{center}
\centering\scriptsize
\setlength{\tabcolsep}{3pt}
\begin{tabular}{c|ccccccccc}
    & $\mathcal{U}_0$ & $\mathcal{Y}_1$ & $\mathcal{Y}_2$ & $\dots$ & $\mathcal{Y}_{n-1}$ & $\mathcal{W}_1$ & $\mathcal{W}_2$ & $\dots$ & $\mathcal{W}_{n-1}$  \\
[0.5ex]
\hline
$u_1$ & $p^n$ & $p^{n-1}$ & $p^{n-2}$ & $\dots$ & $p$  & $0$ & $0$ & $\dots$ & $0$ \\
$u_2$ & $-p^{n-1}$ & $p^{n-1}$ & $p^{n-2}$ & $\dots$ & $p$ & $0$ & $0$ & $\dots$ & $0$ \\
$v_1$ & $0$ & $0$ & $p^{n-2}$ & $\dots$ & $p$ & $p^{n-2}\sqrt{p^*}$ & $0$ & $\dots$ & $0$ \\
$v'_1$ & $0$ & $0$ & $p^{n-2}$ & $\dots$ & $p$ & $-p^{n-2}\sqrt{p^*}$ & $0$ & $\dots$ & $0$ \\
$v_2$ & $0$ & $0$ & $0$  & $\dots$ & $p$ & $0$ & $p^{n-3}\sqrt{p^*}$ & $\dots$ & $0$ \\
$v'_2$ & $0$ & $0$ & $0$  & $\dots$ & $p$ & $0$ & $-p^{n-3}\sqrt{p^*}$ & $\dots$ & $0$ \\
$\vdots$ & $\vdots$ & $\vdots$ & $\vdots$ & $\vdots$ & $\vdots$ & $\vdots$ & $\vdots$ & $\vdots$ & $\vdots$ \\
$v_{n-1}$ & $0$ & $0$ & $0$  & $\dots$ & $0$ & $0$ & $0$ & $\dots$ & $\sqrt{p^*}$ \\
$v'_{n-1}$ & $0$ & $0$ & $0$  & $\dots$ & $0$ & $0$ & $0$ & $\dots$ & $-\sqrt{p^*}$
\end{tabular}
\end{center}
For $\eta\not\equiv1$, all entries remain the same except that the
$\mathcal U_0$-eigenvalues on $u_1,u_2$ are
$p^{n-1}\sqrt{p^*}$ and $-p^{n-1}\sqrt{p^*}$, respectively.

\begin{lem} \label{trace-zer}
    The operators $\pi_L(\mathcal{U}_0),\pi_L(\mathcal{V}_r)$ and $\pi_L(\mathcal{W}_r)$ have zero trace.
\end{lem}
\begin{proof}
Choose representatives $(g,1)$ for the left cosets of $\ov{K_n}$ in
$\ov K$.  Let $\phi_g$ be supported on $\ov{K_n}(g,1)$ and normalized by
\[
 \phi_g(\tilde k(g,1))=\ov\eta(\tilde k),
 \qquad \tilde k\in\ov{K_n}.
\]
These functions form a basis of $I(n)$.  Since the supporting double coset
of each of $\mathcal U_0,\mathcal V_{r,1},\mathcal V_{r,\ld}$ is disjoint
from $\ov{K_n}$, a direct support argument gives
\[
\pi_L(\mathcal{V}_{r,1})(\phi_g)((g,1))
=
\pi_L(\mathcal{V}_{r,\ld})(\phi_g)((g,1))
=
\pi_L(\mathcal{U}_0)(\phi_g)((g,1))
=0.
\]
Thus $\pi_L(\mathcal{U}_0)$, $\pi_L(\mathcal{V}_{r,1})$ and
$\pi_L(\mathcal{V}_{r,\ld})$ are traceless. By linearity of trace, so are
$\pi_L(\mathcal{V}_r)$ and $\pi_L(\mathcal{W}_r)$.
\end{proof}
For each of the displayed basis vectors $v$, put
$d_v=\dim\Span(\pi_R(\ov K)v)$. By Lemma~\ref{trace-zer},
\[
0=\operatorname{tr}\pi_L(\mathcal W_r)
=p^{n-r-1}\sqrt{p^*}(d_{v_r}-d_{v'_r}),
\]
so $d_{v_r}=d_{v'_r}$ for $1\leq r\leq n-1$.
The $\mathcal U_0$-trace equation gives
\[
d_{u_2}=
\begin{cases}
p\,d_{u_1},&\eta\equiv1,\\
d_{u_1},&\eta\not\equiv1.
\end{cases}
\]
The $\mathcal V_r$-eigenvalues are identical for both characters.
After canceling the common factor $p^{n-r-1}$, their trace equations give
\[
(p-1)\left(d_{u_1}+d_{u_2}
+\sum_{j=1}^{r-1}(d_{v_j}+d_{v'_j})\right)
=d_{v_r}+d_{v'_r},
\qquad 1\leq r\leq n-1.
\]
Together with
\[
d_{u_1}+d_{u_2}+\sum_{r=1}^{n-1}(d_{v_r}+d_{v'_r})
=\dim I(n)=p^{n-1}(p+1),
\]
these equations give the following dimensions.
\begin{cor}
The representation $I(n)$ is the multiplicity-free sum of the $2n$
irreducible representations
\begin{align*}
S_i&=\Span(\pi_R(\ov K)u_i),&&i=1,2,\\
T_r&=\Span(\pi_R(\ov K)v_r),&
T'_r&=\Span(\pi_R(\ov K)v'_r),\qquad 1\leq r\leq n-1.
\end{align*}
Their dimensions are
\[
(\dim S_1,\dim S_2)=
\begin{cases}
(1,p),&\eta\equiv1,\\
\bigl(\frac{p+1}{2},\frac{p+1}{2}\bigr),&\eta\not\equiv1,
\end{cases}
\quad
\dim T_r=\dim T'_r=p^{r-1}\frac{p^2-1}{2}.
\]
\end{cor}

\section{Action of the Hecke algebra $H_m(\eta)$ on newvectors
of prescribed $\eta$-type}
In this section, we determine the action of the Hecke algebra
$H_m(\eta)$ on the space of
$(\overline{K_0(p^m)},\eta)$-fixed vectors in an irreducible
admissible genuine representation $\pi$ of $\widetilde G$,
where $m=c_\eta(\pi)$ is the minimal level at which such
vectors exist. In Subsection~4.1, we study induced
representations and the principal series. Subsection~4.2
treats the even Weil and Steinberg representations, and
Subsection~4.3 treats supercuspidal representations.
\subsection{Induced Representations}
Retain the additive character \(\psi\) of conductor $0$ fixed in Section 2.
Let $\mu$ be a
character of $\mathbb{Q}_p^\times$ of conductor exponent $n$. We first
define explicitly the genuine character used in the induction. For
$a\in\mathbb Q_p^\times$, $b\in\mathbb Q_p$, and
$\epsilon\in\{\pm1\}$, put
\[
\chi_\psi((h(a)x(b),\epsilon))
=\epsilon\gamma(a,\psi)^{-1}.
\]
The restriction of our cocycle to the torus is
$c(h(a),h(a'))=(a,a')_p$. Hence the identity
$\gamma(aa',\psi)=\gamma(a,\psi)\gamma(a',\psi)(a,a')_p$,
shows that $\chi_\psi$ is a genuine character of $B$; it is trivial on the
canonical lift of the upper unipotent subgroup. Define
\[
\widetilde\mu_\psi((h(a)x(b),\epsilon))
=\mu(a)\chi_\psi((h(a)x(b),\epsilon))
=\epsilon\mu(a)\gamma(a,\psi)^{-1}.
\]
We use normalized induction and set
\[
\pi_\psi(\mu)=\operatorname{Ind}_{B}^{\widetilde G}\widetilde{\mu}_\psi,
\]
with representation space $V(\mu)$. Thus $V(\mu)$ consists of the locally
constant functions $f:\widetilde G\to\mathbb C$ satisfying
\[
f((h(a)x(b),\epsilon)g)
=|a|\,\widetilde\mu_\psi((h(a)x(b),\epsilon))f(g)
=\epsilon|a|\mu(a)\gamma(a,\psi)^{-1}f(g).
\]
This is the standard genuine torus character and normalized induced model
used by Baruch--Mao, written there with a different Weil-index notation;
see \cite{B-Mao}.  It also agrees with the
normalization in \cite[\S3.2]{Ishimoto}.  Baruch--Mao and Ishimoto formulate
the construction in the $\sigma_0$-model.  Since $s_p(g)=1$ for every
$g\in B_0$, the isomorphism $\iota$ is the identity on the full inverse
image of $B_0$.  Therefore the genuine character, the factor $|a|$ coming
from normalized induction, and the displayed transformation law are
unchanged in our $c$-model. For brevity,
write
\[
\mu_B((h(a)x(b),\epsilon))
:=|a|\,\widetilde\mu_\psi((h(a)x(b),\epsilon)).
\]
Since $c(\psi)=0$, the Weil-factor identities in Section~2 give
$\gamma(u,\psi)=1$ for $u\in\mathbb Z_p^\times$. Consequently, on the
compact diagonal elements used in the support calculations below,
$\mu_B((h(u),\epsilon))=\epsilon\mu(u)$.

Now let $\eta$ be either the trivial character or the character
$\eta=\legendre{.}{p}$, with the conventions of
Section~\ref{sec:hecke-subalgebra}. For $m\geq c(\eta)$, we write
\[
\pi_\psi(\mu)_\eta^{K_m}
=V(\mu)_\eta^{K_m}
=V_\eta^{K_m}\bigl(\pi_\psi(\mu)\bigr).
\]
In the induced model this space is
\[
V(\mu)_\eta^{K_m}
=\{f\in V(\mu):f(g\tilde k)=\eta(\tilde k)f(g)
\text{ for all }g\in\widetilde G,\ \tilde k\in\ov{K_m}\}.
\]
Recall that $c_\eta(\pi_\psi(\mu))$ is the least such $m$ for which
this space is nonzero. %

For each $m$, the Hecke subalgebra $H_m(\eta)=H(\ov K//\ov{K_m},\eta)$ of Section~\ref{sec:hecke-subalgebra} acts on $V(\mu)_{\eta}^{K_m}$. We write $H_m=H_m(\eta)$ when the character is understood. We shall determine the required simultaneous eigenspaces explicitly. Fix again a quadratic non-residue $\lambda\in\mathbb{Z}_p$ as in the previous section.
\begin{lem}\label{lem:B-double-cosets}
The double cosets
\[
\begin{gathered}
 B\ov{K_m},\qquad
 B(w(1),1)\ov{K_m},\\
 B(y(p^r),1)\ov{K_m},\qquad
 B(y(\lambda p^r),1)\ov{K_m}
\end{gathered}
\]
where $r\in\{1,\ldots,m-1\}$ and $\lambda$ is a nonsquare modulo $p$, form a complete set of double coset representatives for
$B\backslash \widetilde G/\ov{K_m}$.
\end{lem}
\begin{proof}
The fact that these cosets cover $\widetilde G$ follows from the Iwasawa decomposition $\widetilde G=B\ov K$ and Proposition~\ref{prop:rep1}.

To show disjointness, assume $y(t_1)k_0y(-t_2)\in B_0$ for some $\begin{pmatrix}
       a & b \\
       c & d
   \end{pmatrix}=k_0\in K_m$. A direct calculation gives
\[
t_2a+c-t_1t_2b-t_1d=0.
\]
Write $t_i=\ell_i p^{r_i}$ with $\ell_i\in\mathbb Z_p^\times$ and
$1\leq r_i<m$. If $r_1\ne r_2$, the term having the smaller valuation
cannot be cancelled by any of the other three terms. Hence $r_1=r_2=r$.
After division by $p^r$ and reduction modulo $p$, the determinant condition
gives
\[
\ell_2a-\ell_1a^{-1}\equiv0\pmod p.
\]
Thus $\ell_1$ and $\ell_2$ have the same square class modulo $p$. The
remaining comparisons, involving $1$, $w(1)$, and $y(t)$, follow by the same
lower-row calculation. This proves disjointness.
\end{proof}
Throughout the remainder of this subsection, we assume the central
compatibility condition $\mu(-1)=\eta(-1)$. This is necessary for a
nonzero $(\ov{K_m},\eta)$-fixed vector, since $(h(-1),1)$ acts on $V(\mu)$
by the scalar $\mu(-1)$.

\begin{prop} \label{prop:forms}
    Assume $n>1$. Then $V(\mu)_{\eta}^{K_{2n}}$ is the $2$-dimensional space spanned by the functions
    $$f_1(g)=\begin{cases}
        \mu_B(b)\eta(k_0) & \text{if } g=b(y(p^n),1)k_0 \text{, for some } b\in B, k_0\in \ov{K_{2n}}\\
        0 & \text{otherwise}
    \end{cases}$$
    $$f_{\lambda}(g)=\begin{cases}
        \mu_B(b)\eta(k_0) & \text{if } g=b(y(\lambda p^n),1)k_0 \text{, for some } b\in B, k_0\in \ov{K_{2n}}\\
        0 & \text{otherwise}
    \end{cases}$$
    Consequently $c_{\eta}(\pi_\psi(\mu))=2n$.
\end{prop}
\begin{proof}
We give the support argument at an arbitrary level $K_M$. By
Lemma~\ref{lem:B-double-cosets}, a vector in $V(\mu)_\eta^{K_M}$ is a sum of
functions supported on the displayed $B$--$K_M$ double cosets. A nonzero
function supported on $Bg\ov{K_M}$ exists precisely when the inducing
character and the right $\eta$-type agree on
$B\cap g\ov{K_M}g^{-1}$. Since all the representatives lie in $K$ and the cover
splits over $K$, no additional cocycle factor occurs in this check.

The cosets $B\ov{K_M}$ and $B(w(1),1)\ov{K_M}$ cannot support such a
function. Indeed, applying the compatibility condition to the diagonal
elements $h(u)$, $u\in\mathbb Z_p^\times$, would give respectively
$\mu|_{\mathbb Z_p^\times}=\eta$ or
$\mu^{-1}|_{\mathbb Z_p^\times}=\eta$. Either equality contradicts
$c(\mu)=n>1$, since $c(\eta)\leq1$.

It remains to examine $g=(y(\delta p^k),1)$, where
$\delta\in\{1,\lambda\}$ and $1\leq k<M$. Writing
$b=(\begin{pmatrix}
        \alpha & \beta \\
        0 & \alpha^{-1}
    \end{pmatrix},1)\in B$, the condition
$g^{-1}bg\in\ov{K_M}$ is equivalent, on the linear group, to
\[
\begin{pmatrix}
        \alpha+\beta\delta p^k & \beta \\
        p^k(\alpha^{-1}-\alpha-\beta\delta p^k) & \alpha^{-1}-\beta\delta p^k
    \end{pmatrix}\in K_M.
\]
Thus $\beta\in\mathbb Z_p$, $\alpha\in\mathbb Z_p^\times$, and
\[
\alpha^{-1}-\alpha-\beta\delta p^k\equiv0\pmod{p^{M-k}}.
\]
For such an element the support compatibility is
\[
\mu(\alpha^{-1})
=\eta(\alpha^{-1}-\beta\delta p^k).
\]
If $k<n$, choose $\alpha\in1+p^k\mathbb Z_p$ and
$\beta=(\alpha^{-1}-\alpha)/(\delta p^k)$. The support compatibility then
forces $\mu$ to be trivial on $1+p^k\mathbb Z_p$, contrary to
$c(\mu)=n$. If $M-k<n$, take $\beta=0$ and
$\alpha\in1+p^{M-k}\mathbb Z_p$; the same compatibility forces $\mu$ to be
trivial on $1+p^{M-k}\mathbb Z_p$, again a contradiction. Therefore an
$y$-coset supporting a nonzero type function must satisfy
\[
k\geq n,\qquad M-k\geq n.
\]
In particular, no such coset occurs when $M<2n$, proving
$V(\mu)_\eta^{K_M}=0$ for every $M<2n$.

At $M=2n$, these two inequalities force $k=n$. In this case the preceding
congruence gives
$\alpha^2\equiv1\pmod{p^n}$. Since $p$ is odd, this means
$\alpha\in\pm(1+p^n\mathbb Z_p)$. As $\mu$ is trivial on
$1+p^n\mathbb Z_p$ and $c(\eta)\leq1$, the support compatibility is
equivalent to $\mu(-1)=\eta(-1)$. Hence precisely the two cosets represented
by $y(p^n)$ and $y(\lambda p^n)$ support nonzero type functions, namely
$f_1$ and $f_\lambda$. They are linearly independent, and each supporting
space is one-dimensional. This proves both the dimension assertion and
$c_\eta(\pi_\psi(\mu))=2n$.
\end{proof}
\begin{remark}
Ishimoto independently obtained the same conductor and dimension statement
in \cite[Theorem~4.1]{Ishimoto}, using Frobenius reciprocity and Mackey
decomposition.  He also describes the corresponding functions supported on
the two square-class double cosets in \cite[Corollary~4.3(4)]{Ishimoto}; in
the notation used here, these are the two support lines generated by $f_1$
and $f_\lambda$.  The proof above is a direct support computation.
\end{remark}
\begin{remark}
Suppose that $n=1$.  In the cancelling case
$\eta=\mu^{\pm1}|_{\mathbb Z_p^\times}$, the quadraticity of $\eta$ implies
that both $\mu|_{\mathbb Z_p^\times}$ and
$\mu^{-1}|_{\mathbb Z_p^\times}$ equal $\eta$.  Consequently
$V(\mu)_\eta^{K_1}$ is two-dimensional, spanned by the standard type
functions supported on $B\ov K_1$ and $B(w(1),1)\ov K_1$, and
$c_\eta(\pi_\psi(\mu))=1$.  In the remaining case,
$c(\eta\mu)>0$ and $c(\eta\mu^{-1})>0$; then
$V(\mu)_\eta^{K_1}=0$, while the two functions supported on
$B(y(p),1)\ov K_2$ and $B(y(\lambda p),1)\ov K_2$ span
$V(\mu)_\eta^{K_2}$.  Thus the $\eta$-conductor is $2$, and this case is
treated in Proposition~\ref{prop:level1}; compare
\cite[Theorem~4.1 and Corollary~4.3]{Ishimoto}.
\end{remark}
We will now use the results on eigenvalues from the previous section to identify the eigenvalues of the newforms inside this two-dimensional space. On $\mathbb Z_p^\times$ we write $\overline\mu=\mu^{-1}$; since the restriction of $\mu$ to $\mathbb Z_p^\times$ has finite order, this is also its complex conjugate.
\begin{prop} \label{prop:values}
   Let $\bar{g}=(g,1)\in \widetilde G$. For $n>1$ we obtain the following equalities:
    $$\mathcal{V}_{2n-1,\kappa}(f_{\kappa})(\bar{y}(\delta p^n))=\begin{cases}
        \frac{-1\pm \sqrt{p^*}}{2} & \text{if } \delta=\kappa \\
        0 & \text{if } \delta\ne\kappa
    \end{cases}$$
    $$\mathcal{V}_{2n-1,\beta}(f_{\kappa})(\bar{y}(\delta p^n))=\begin{cases}
        \frac{-1\mp \sqrt{p^*}}{2} & \text{if } \delta=\kappa \\
        0 & \text{if } \delta\ne\kappa
    \end{cases}$$ where $\beta\ne\kappa$ in the second case and $\beta,\kappa,\delta\in\{1,\lambda\}$. The signs depend on $\mu$.
\end{prop}
\begin{proof}
    For $\beta\in\{1,\lambda\}: \mathcal{V}_{2n-1,\beta}(f)(g)=\sum\limits_{u\in\Z_p^\times/\sqrt{1+p\Z_p}}f(g\bar{h}(u)\bar{y}(\beta p^{2n-1}))\eta(u)$.
\par\smallskip\noindent
Now taking $g=\bar{y}(\delta p^n)$ and $f=f_{\kappa}$ the argument in the summation becomes
\[
\begin{pmatrix}
        1 & 0 \\
        \delta p^n & 1
    \end{pmatrix}\begin{pmatrix}
        u & 0 \\
        0 & u^{-1}
    \end{pmatrix}\begin{pmatrix}
        1 & 0 \\
        \beta p^{2n-1} & 1
    \end{pmatrix}=\begin{pmatrix}
        u & 0 \\
        0 & u^{-1}
    \end{pmatrix}\begin{pmatrix}
        1 & 0 \\
        \delta p^nu^2+\beta p^{2n-1} & 1
    \end{pmatrix} .
\]

\par\smallskip\noindent
From a straightforward calculation it follows that there exists a matrix $\begin{pmatrix}
        a & b \\
        0 & a^{-1}
    \end{pmatrix}\in B_0$ such that $\begin{pmatrix}
        1 & 0 \\
        -\kappa p^n & 1
    \end{pmatrix}\begin{pmatrix}
        a & b \\
        0 & a^{-1}
    \end{pmatrix}\begin{pmatrix}
        1 & 0 \\
        \delta p^nu^2+\beta p^{2n-1} & 1
    \end{pmatrix}\in K_{2n}$ if and only if $\delta u^2+\beta p^{n-1}\equiv\kappa a^2\pmod{p^n}, b\in\Z_p$ and $a\in\Z_p^\times$.
\par\smallskip\noindent
This relation has a solution if and only if $\delta=\kappa$.  In this case
we choose the square root by setting
$$
a=u\sqrt{1+\frac{\beta}{\kappa}p^{n-1}u^{-2}},
\qquad
\sqrt{1+\frac{\beta}{\kappa}p^{n-1}u^{-2}}
\in1+p^{n-1}\mathbb Z_p.
$$
This is the unique choice for which $a\equiv u\pmod p$.  Since
$c(\eta)\leq1$, it follows that $\eta(a)=\eta(u)$.
\par\smallskip\noindent
First let $\kappa=\delta=\beta$, from which we obtain
    $$\mathcal{V}_{2n-1,\kappa}(f_{\kappa})(\bar y(\kappa p^n))=\sum\limits_{u\in\Z_p^\times/\sqrt{1+p\Z_p}}\bar{\mu}(\frac{\sqrt{u^2+p^{n-1}}}{u})\eta(\sqrt{u^2+p^{n-1}})\eta(u).$$ With the choice just made, $\eta(\sqrt{u^2+p^{n-1}})=\eta(u)$ and $\frac{\sqrt{u^2+p^{n-1}}}{u}=\sqrt{1+p^{n-1}u^{-2}}=1+\frac{1}{2}p^{n-1}u^{-2}\pmod{p^n}$, so the last sum evaluates to $\sum\limits_{u\in\Z_p^\times/\sqrt{1+p\Z_p}}\bar{\mu}(1+\frac{1}{2}p^{n-1}u^{-2})$.
\par\smallskip\noindent
Since $c(\mu)=n>1$, the map $t\longmapsto\overline\mu\!\left(1+\tfrac12p^{n-1}t\right)$
defines a nontrivial additive character of $\mathbb F_p$. Put
$\zeta=\overline\mu\!\left(1+\tfrac12p^{n-1}\right)$. Its quadratic Gauss sum, in the two equivalent forms of
Section~\ref{sec:hecke-subalgebra}, is
\[
g_\zeta(1;p):=\sum_{x\in\mathbb F_p}\zeta^{x^2}
=\sum_{t\in\mathbb F_p^*}\chi_p(t)\zeta^t
=\pm\sqrt{p^*}.
\]
The sign is determined by $\mu$, relative to the choice of $\sqrt{p^*}$
fixed in Section~\ref{sec:hecke-subalgebra}.
Reindexing by $u\mapsto u^{-1}$, which permutes the quotient, gives
\begin{align*}
\mathcal{V}_{2n-1,\kappa}(f_{\kappa})(\bar{y}(\kappa p^n))
&=\sum\limits_{u\in\Z_p^\times/\sqrt{1+p\Z_p}}\zeta^{u^2}
=\frac{1}{2}\sum\limits_{u\in\Z_p^\times/(1+p\Z_p)}\zeta^{u^2}\\
&=\frac{g_\zeta(1;p)-1}{2}=\frac{-1\pm\sqrt{p^*}}{2}.
\end{align*}
\par\smallskip\noindent
In the same way for $\beta\ne\kappa=\delta$ we obtain $$\mathcal{V}_{2n-1,\beta}(f_{\kappa})(\bar{y}(\kappa p^n))=\frac{1}{2}\sum\limits_{u\in\Z_p^\times/(1+p\Z_p)}\zeta^{\frac{\beta}{\kappa}u^2}=\frac{-1-g_\zeta(1;p)}{2}=\frac{-1\mp\sqrt{p^*}}{2}.$$
\par\smallskip\noindent
Notice that in all of our calculations we remain inside $K$ and $\widetilde G$ splits over $K$, so the cocycle component plays no part in the computation.
\end{proof}
\begin{cor}
    The functions $f_1,f_{\lambda}$ defined in Proposition \ref{prop:forms} are the common eigenvectors in $V(\mu)_{\eta}^{K_{2n}}$ of the operators $\mathcal{Y}_{2n-1}$ and $\mathcal{W}_{2n-1}$ from the previous section. The forms and their corresponding eigenvalues are given in the table below.
\end{cor}
\begin{proof}
Proposition~\ref{prop:values} shows directly that each of $f_1$ and
$f_{\lambda}$ is an eigenvector for both $\mathcal V_{2n-1,1}$ and
$\mathcal V_{2n-1,\lambda}$. Since at the terminal index
\[
\mathcal Y_{2n-1}
=
I+\mathcal V_{2n-1,1}+\mathcal V_{2n-1,\lambda},
\qquad
\mathcal W_{2n-1}
=
\mathcal V_{2n-1,1}-\mathcal V_{2n-1,\lambda},
\]
the displayed eigenvalues in Proposition~\ref{prop:values} give
$\mathcal Y_{2n-1}f_\kappa=0$ and
$\mathcal W_{2n-1}f_\kappa=\pm\sqrt{p^*}\,f_\kappa$, with opposite signs for
$\kappa=1$ and $\lambda$.
\end{proof}
\begin{center}
\begin{tabular}{cc}
\begin{tabular}{ |c|c|c| }
 \hline
   Case 1 & $f_1$ & $f_{\lambda}$ \\
 \hline
 $\mathcal{V}_{2n-1,1}$ & $\frac{-1\pm\sqrt{p^*}}{2}$ & $\frac{-1\mp\sqrt{p^*}}{2}$  \\
 \hline
 $\mathcal{V}_{2n-1,\lambda}$ & $\frac{-1\mp\sqrt{p^*}}{2}$  & $\frac{-1\pm\sqrt{p^*}}{2}$ \\
 \hline
\end{tabular}

&

\begin{tabular}{ |c|c|c| }
 \hline
   Case 2 & $f_1$ & $f_{\lambda}$ \\
 \hline
 $\mathcal{V}_{2n-1,1}$ & $\frac{-1\mp\sqrt{p^*}}{2}$ & $\frac{-1\pm\sqrt{p^*}}{2}$  \\
 \hline
 $\mathcal{V}_{2n-1,\lambda}$ & $\frac{-1\pm\sqrt{p^*}}{2}$  & $\frac{-1\mp\sqrt{p^*}}{2}$ \\
 \hline
\end{tabular}
\end{tabular}
\end{center}
\vspace{3mm}
\begin{center}

\begin{tabular}{cc}
\begin{tabular}{ |c|c|c| }
 \hline
   Case 1 & $f_1$ & $f_{\lambda}$ \\
 \hline
 $\mathcal{Y}_{2n-1}$ & $0$ & $0$  \\
 \hline
 $\mathcal{W}_{2n-1}$ & $\sqrt{p^*}$  & $-\sqrt{p^*}$ \\
 \hline
\end{tabular}

&

\begin{tabular}{ |c|c|c| }
 \hline
   Case 2 & $f_1$ & $f_{\lambda}$ \\
 \hline
 $\mathcal{Y}_{2n-1}$ & $0$ & $0$  \\
 \hline
 $\mathcal{W}_{2n-1}$ & $-\sqrt{p^*}$  & $\sqrt{p^*}$ \\
 \hline
\end{tabular}
\end{tabular}
\end{center}
\begin{prop} \label{prop:level1}
    Let $n=1$ and assume $c(\eta\mu)>0$ and $c(\eta\mu^{-1})>0$, so that $c_{\eta}(\pi_\psi(\mu))=2$. Then the $2$-dimensional space spanned by the functions
    $$f_1(g)=\begin{cases}
        \mu_B(b)\eta(k_0) & \text{if } g=b(y(p),1)k_0 \text{, for some } b\in B, k_0\in \ov{K_2}\\
        0 & \text{otherwise}
    \end{cases}$$
    $$f_{\lambda}(g)=\begin{cases}
        \mu_B(b)\eta(k_0) & \text{if } g=b(y(\lambda p),1)k_0 \text{, for some } b\in B, k_0\in \ov{K_2}\\
        0 & \text{otherwise}
    \end{cases}$$
has a basis $\{v_1, v_2\}$ such that $\mathcal U_0(v_1)=\mathcal U_0(v_2)=\mathcal{Y}_{1}(v_1)=\mathcal{Y}_{1}(v_2)=0$ and $\mathcal{W}_{1}(v_1)=\sqrt{p^*}v_1, \mathcal{W}_{1}(v_2)=-\sqrt{p^*}v_2$.
\end{prop}
\begin{proof}
We have that
\[
\mathcal U_0(f_{\kappa})(\bar{y}(\delta p))
=\sum\limits_{s\in\Z_p/p^2\Z_p}f_{\kappa}(\bar{y}(\delta p)\bar{x}(s)\bar{w}(1)).
\]
Let
\[
M=\begin{pmatrix}
        -s & 1 \\
        -1-\delta ps & \delta p
    \end{pmatrix},
\]
then the sum is zero unless $Mk_0y(-\kappa p)\in B_0$ for some $k_0\in K_2$. But
\begin{align*}
&\begin{pmatrix}
        -s & 1 \\
        -1-\delta ps & \delta p
    \end{pmatrix}
\begin{pmatrix}
        a & b \\
        p^2c & d
    \end{pmatrix}
\begin{pmatrix}
        1 & 0 \\
        -\kappa p & 1
    \end{pmatrix}\\
&\qquad=
\begin{pmatrix}
        -s & 1 \\
        -1-\delta ps & \delta p
    \end{pmatrix}
\begin{pmatrix}
        a-\kappa pb & b \\
        p^2c-\kappa pd & d
    \end{pmatrix}\\
&\qquad=
\begin{pmatrix}
        * & * \\
        (-1-\delta ps)(a-\kappa pb)+\delta p(p^2c-\kappa p d) & *
    \end{pmatrix}\not\in B_0
\end{align*}
since $a\in\Z_p^\times$.
It follows that
$\mathcal U_0(f_\kappa)(\bar y(\delta p))=0$ for every
$\delta,\kappa\in\{1,\lambda\}$. Since $\mathcal U_0(f_\kappa)$ belongs to
the space spanned by $f_1,f_\lambda$, and these two values determine a vector
in that space, $\mathcal U_0$ acts by zero on all of
$V(\mu)_\eta^{K_2}$.

Since $c_\eta(\pi_\psi(\mu))=2$, the space $V(\mu)_\eta^{K_1}$ is zero. At level $2$, Remark~\ref{rem:Y-averaging} shows that $\mathcal Y_1=\mathcal I+\mathcal V_{1,1}+\mathcal V_{1,\lambda}$ is supported on $\ov{K_1}$ and takes the value $\eta(\tilde k)^{-1}$ at each $\tilde k\in\ov{K_1}$. Hence $p^{-1}\mathcal Y_1$ acts as the averaging projection onto $V(\mu)_\eta^{K_1}$, and $\mathcal Y_1$ acts by zero on $V(\mu)_\eta^{K_2}$.

\par\smallskip\noindent
The relation $\mathcal W_1^2=p^*I$ on this space shows that
$\mathcal W_1$ is diagonalizable and that its eigenvalues belong to
$\{\sqrt{p^*},-\sqrt{p^*}\}$. We show that both occur. From the same
calculations as in Proposition~\ref{prop:values} for $n=1$ we obtain
\[
\mathcal{V}_{1,1}(f_1)(\bar{y}(p))
=\mathcal{V}_{1,\lambda}(f_{\lambda})(\bar{y}(\lambda p))
=\sum\limits_{\underset{\legendre{u^2+1}{p}=1}{u\in\Z_p^\times/\sqrt{1+p\Z_p}}}
\bar{\mu}\left(\frac{\sqrt{u^2+1}}{u}\right)\eta(u\sqrt{u^2+1})
\tag{*}
\]

\par\smallskip\noindent
Assume that only one $\mathcal W_1$-eigenvalue occurs. Since
$\mathcal W_1$ is diagonalizable, it is then scalar. As $\mathcal Y_1=0$,
both $\mathcal V_{1,1}$ and $\mathcal V_{1,\lambda}$ are scalar as well. We
write their scalar values in the following table:
\begin{center}
\begin{tabular}{ |c|c|c| }
 \hline
    & $f_1$ & $f_{\lambda}$ \\
 \hline
 $\mathcal{V}_{1,1}$ & $\ell_{1,1}$ & $\ell_{1,\lambda}$  \\
 \hline
 $\mathcal{V}_{1,\lambda}$ & $\ell_{\lambda,1}$  & $\ell_{\lambda,\lambda}$ \\
 \hline
\end{tabular}
\end{center}
Relation (*) gives $\ell_{1,1}=\ell_{\lambda,\lambda}$. Since the two
operators are scalar, we also have
$\ell_{1,1}=\ell_{1,\lambda}$ and
$\ell_{\lambda,1}=\ell_{\lambda,\lambda}$. Hence
$\mathcal V_{1,1}$ and $\mathcal V_{1,\lambda}$ act by the same scalar, so
$\mathcal W_1=0$, contradicting $\mathcal W_1^2=p^*I$. Both eigenvalues
therefore occur. Choosing corresponding eigenvectors $v_1,v_2$ proves the
proposition.
\end{proof}

\subsection{Reducible Principal Series: Even Weil and Steinberg Representations}
We study the two irreducible constituents of the reducible principal series
in Proposition~\ref{prop:short-exact}. We first recall the Weil
representation in the conventions fixed in Section~2; compare
\cite[\S3.2]{Ishimoto}.
$$
\begin{gathered}
{}[\omega_\psi((h(a), \epsilon)) \cdot \varphi](y) = \epsilon|a|^{\frac{1}{2}}\gamma(a, \psi)^{-1}\varphi(ay), \\
{}[\omega_\psi((x(b),1)) \cdot \varphi](y) = \psi(by^2)\varphi(y), \\
{}[\omega_\psi((w(1), 1)) \cdot \varphi](y) = \gamma(\psi) \int_{\mathbb{Q}_p} \varphi(x)\psi(2xy)d_{\psi_2}x,
\end{gathered}
$$
Here $d_{\psi_2}x$ is self-dual for the character
$x\mapsto\psi(2x)$. Since $p$ is odd and $c(\psi)=0$, we take
$\operatorname{vol}_{d_{\psi_2}x}(\Z_p)=1$. Let
$\mathcal{S}^+(\mathbb{Q}_p)$ and $\mathcal{S}^-(\mathbb{Q}_p)$ denote the
subspaces of even and odd functions, respectively.

Thus
$$
\mathcal S(\Q_p)=\mathcal S^+(\Q_p)\oplus\mathcal S^-(\Q_p),
\qquad
\omega_\psi=\omega_\psi^+\oplus\omega_\psi^-.
$$
The even and odd Weil representations are irreducible; the former is
non-supercuspidal and the latter is supercuspidal. For
$a\in\Q_p^\times$, put
$\mu_a(x)=(a,x)_p$, and recall $\psi_a(x)=\psi(ax)$. The isomorphism class of $\omega_{\psi_a}$ depends
only on the square class of $a$, and we write
$$
\omega_{\psi,\mu_a}^{\pm}:=\omega_{\psi_a}^{\pm}.
$$
See \cite[\S3.2]{Ishimoto} for these facts.

The reducibility properties of these representations are summarized in \cite[Proposition 3.3]{Ishimoto} as follows.

\begin{prop}

\label{prop:short-exact}

\textit{(1) The representation $\pi_\psi(\mu)$ is irreducible if and only if $\mu^2 \neq |\cdot|^{\pm 1}$. In this case we have $\pi_\psi(\mu) \cong \pi_\psi(\mu^{-1})$.}

\textit{(2) If $\mu = \mu_a |\cdot|^{\frac{1}{2}}$, where $\mu_a$ is a quadratic or trivial character, then $\pi_\psi(\mu)$ is reducible and there are an irreducible representation $St_{\psi,\mu_a}$ and a short exact sequence
$$0 \to St_{\psi,\mu_a} \to \pi_\psi(\mu) \to \omega_{\psi,\mu_a}^+ \to 0.$$}

We shall call the representation $St_{\psi,\mu_a}$ the Steinberg representation associated to $\psi$ and $\mu_a$.

\textit{(3) If $\mu = \mu_a |\cdot|^{-\frac{1}{2}}$, where $\mu_a$ is a quadratic or trivial character, then $\pi_\psi(\mu)$ is reducible and there is a short exact sequence
$$0 \to \omega_{\psi,\mu_a}^+ \to \pi_\psi(\mu) \to St_{\psi,\mu_a} \to 0.$$}

The proposition gives all the irreducible non-supercuspidal genuine representations of $\widetilde{G}$; see \cite[Proposition 3.3]{Ishimoto}.
\end{prop}
Choose a nonsquare unit $\beta\in\Z_p^\times$ and henceforth take
$a\in\{1,\beta,p,p\beta\}$.  Let $\eta$ be either the trivial character
or $\legendre{\cdot}{p}$, assume the compatibility condition
$\eta(-1)=\mu_a(-1)$, and put
$\mu=\mu_a|\cdot|^{1/2}$.  We write $c(\mu_a)$ and $c(\eta)$ for the
respective conductor exponents.

Proposition~\ref{prop:short-exact} gives
$$
0\longrightarrow St_{\psi,\mu_a}
\longrightarrow \pi_\psi(\mu)
\xrightarrow{\mathcal M}\omega_{\psi,\mu_a}^+
\longrightarrow0.
$$
For every level at which $\eta$ defines a character of $\ov{K_m}$, let
$\mathcal M_m$ be the restriction of $\mathcal M$ to the
$(\ov{K_m},\eta)$-isotypic space.  By \cite[Lemma 4.9]{Ishimoto}, the map
$\mathcal M_m$ is surjective for every $m$.  We therefore have a short
exact sequence
\begin{equation}\label{eq:steinberg-type-exact}
0\longrightarrow (St_{\psi,\mu_a})_\eta^{K_m}
\longrightarrow \pi_\psi(\mu)_\eta^{K_m}
\xrightarrow{\mathcal M_m}(\omega_{\psi,\mu_a}^+)_\eta^{K_m}
\longrightarrow0.
\end{equation}
This is also an exact sequence of modules for the corresponding Hecke
algebra $H_m$.  Indeed, since $\mathcal M$ is a $\widetilde G$-intertwining
map, for $T\in H_m$ and $v\in\pi_\psi(\mu)_\eta^{K_m}$ one has
\[
\mathcal M_m(T\cdot v)
=\mathcal M_m\left(\int_{\widetilde G}T(g)\pi_\psi(\mu)(g)v\,dg\right)
=T\cdot\mathcal M_m(v).
\]
Thus both the Steinberg kernel and the even Weil quotient are preserved by
the Hecke action.

The compatibility condition leaves the following four cases:
\begin{enumerate}
\item $c(\mu_a)=c(\eta)=0$: $a\in\{1,\beta\}$ and $\eta=1$;
\item $c(\mu_a)=1$, $c(\eta)=0$: $a\in\{p,p\beta\}$,
      $\eta=1$, and necessarily $p\equiv1\pmod4$;
\item $c(\mu_a)=0$, $c(\eta)=1$: $a\in\{1,\beta\}$,
      $\eta=\legendre{\cdot}{p}$, and necessarily $p\equiv1\pmod4$;
\item $c(\mu_a)=c(\eta)=1$: $a\in\{p,p\beta\}$ and
      $\eta=\legendre{\cdot}{p}=\mu_a|_{\Z_p^\times}$.
\end{enumerate}

The dimension formulas for the induced and even Weil representations are
\cite[Theorems 4.1 and 4.4]{Ishimoto}; in particular,
$$
c_\eta(\omega_{\psi,\mu_a}^+)
=2c(\eta\mu_a)+c(\mu_a).
$$
Together with \eqref{eq:steinberg-type-exact}, these formulas give the
following conductor and dimension data.

\begin{prop}\label{prop:steinberg-dimensions}
In Cases $1$--$4$, respectively, the $\eta$-conductors are
\[
\begin{array}{c|c|c|c}
\text{Case}
&c_\eta(\pi_\psi(\mu))
&c_\eta(\omega_{\psi,\mu_a}^+)
&c_\eta(St_{\psi,\mu_a})\\ \hline
1&0&0&1\\
2&2&3&2\\
3&2&2&2\\
4&1&1&1
\end{array}
\]
The dimensions at the levels needed below are
\[
\begin{array}{c|c|c|c|c}
\text{Case}&m
&\dim\pi_\psi(\mu)_\eta^{K_m}
&\dim(\omega_{\psi,\mu_a}^+)_\eta^{K_m}
&\dim(St_{\psi,\mu_a})_\eta^{K_m}\\ \hline
1&0&1&1&0\\
1&1&2&1&1\\
2&2&2&0&2\\
3&2&2&1&1\\
4&1&2&1&1
\end{array}
\]
In Case $2$ all three isotypic spaces vanish for $m<2$.  In Case $3$
the three $(\ov{K_1},\eta)$-isotypic spaces vanish.  For the ramified type in Cases
$3$ and $4$, we only use levels $m\geq1$, where $\eta$ is a character of
$\ov{K_m}$.
\end{prop}

For each level under consideration, normalize Haar measure by
$\operatorname{vol}(\ov{K_m})=1$. For $m\geq2$, Section~3 shows that
$H_m$ is finite-dimensional, commutative, and semisimple; the level-$1$
Hecke subalgebras supported in $\ov K$ are studied in \cite{B-PI}.
Thus the finite-dimensional $(\ov{K_m},\eta)$-isotypic spaces admit
simultaneous eigenbases for the action of $H_m$. We will explicitly
distinguish the affine Iwahori operator $U_1$, whose support is not
contained in $\ov K$, when it is used in Case~4.

With $\epsilon_p=\gamma(p,\psi)$ as fixed in Section~2, the conductor $0$ normalization satisfies
$\epsilon_p^2=\left(\frac{-1}{p}\right)$.

\begin{prop} \label{prop:eigenvalues_steinberg}
    Let the notation be as above. The following Hecke actions hold on the minimal-level Steinberg spaces:
    \begin{enumerate}
        \item \textbf{Case 1 ($c(\mu_a)=0, c(\eta)=0$):} At level $K_1$, the Iwahori-Hecke operator $\widetilde{Q}_p$ acts on the one-dimensional Steinberg component with eigenvalue $-1$.
        \item \textbf{Case 2 ($c(\mu_a)=1, c(\eta)=0$):} At level $K_2$, the two-dimensional Steinberg component has a common eigenbasis on which $\mathcal{Y}_1$ acts by $0$ and $\mathcal{W}_1$ acts by $\sqrt{p^*}$ and $-\sqrt{p^*}$, respectively.
        \item \textbf{Case 3 ($c(\mu_a)=0, c(\eta)=1$):} At level $K_2$, the operator $\mathcal Y_1$ acts by $0$ on both the one-dimensional Steinberg component and the one-dimensional even Weil component. On the even Weil component $\mathcal W_1$ acts by
        $$
        \legendre{a}{p}\sqrt{p^*},
        $$
        and on the Steinberg component it acts by
        $$
        -\legendre{a}{p}\sqrt{p^*}.
        $$
        \item \textbf{Case 4 ($c(\mu_a)=1, c(\eta)=1$):} At level $K_1$, the affine Iwahori operator $U_1$, supported on $K_1w(p^{-1})K_1$ and normalized by $U_1((w(p^{-1}),1))=\epsilon_p$, acts on the one-dimensional Steinberg component with eigenvalue $-\epsilon_p$.
    \end{enumerate}
\end{prop}

\begin{proof}
    We analyze the action of the Hecke operators in each case by considering the natural projections between successive levels.

    For \textbf{Case 1}, let $f_1$ and $f_w$ be the standard basis of
    $\pi_\psi(\mu)_1^{K_1}$ supported on $BK_1$ and $B(w(1),1)K_1$,
    normalized by $f_1(1)=f_w((w(1),1))=1$.  Let
    $\widetilde Q_p$ be the Iwahori Hecke operator supported on
    $K_1w(1)K_1$ and normalized by
    $\widetilde Q_p((w(1),1))=1$.  The standard right-coset calculation
    gives
    $$
    \widetilde Q_p f_1=f_w,
    \qquad
    \widetilde Q_p f_w=pf_1+(p-1)f_w.
    $$
    Consequently
    $$
    \widetilde Q_p(f_1+f_w)=p(f_1+f_w),
    \qquad
    \widetilde Q_p(pf_1-f_w)=-(pf_1-f_w).
    $$
    The decomposition
    $K_0=K_1\sqcup K_1w(1)K_1$ shows that
    $\widetilde Q_p+I$ is supported on $\ov{K_0}$ and takes the value
    $\epsilon$ at $(k,\epsilon)$ for $k\in K_0$. Its restriction to the
    split copy $K_0\times\{1\}$ is therefore the characteristic function
    of $K_0$. Since $\operatorname{vol}(\ov{K_1})=1$, we have
    $\operatorname{vol}(\ov{K_0})=[K_0:K_1]=p+1$. Consequently convolution by
    $(p+1)^{-1}(\widetilde Q_p+I)$ is the normalized averaging projection
    from the $(\ov{K_1},1)$-fixed space onto the $(\ov{K_0},1)$-fixed space.  The two displayed
    identities show explicitly that this projector fixes the line
    $\mathbb C(f_1+f_w)$ and annihilates the line
    $\mathbb C(pf_1-f_w)$.  The former is the $(\ov{K_0},1)$-fixed line
    $\pi_\psi(\mu)_1^{K_0}$.  Since $\mathcal M_0$ is surjective and
    $(St_{\psi,\mu_a})_1^{K_0}=0$, it maps isomorphically onto
    $(\omega_{\psi,\mu_a}^+)_1^{K_0}$.  Moreover,
    $(\omega_{\psi,\mu_a}^+)_1^{K_0}$ and
    $(\omega_{\psi,\mu_a}^+)_1^{K_1}$ are both one-dimensional, so they
    coincide.  Thus the even Weil $(\ov{K_1},1)$-fixed vector is already an oldvector
    coming from level $K_0$.  Exactness of
    \eqref{eq:steinberg-type-exact} therefore identifies the latter line,
    $\mathbb C(pf_1-f_w)$, with the Steinberg newvector line.

    For \textbf{Case 2}, Proposition~\ref{prop:steinberg-dimensions} gives
    $(\omega_{\psi,\mu_a}^+)_1^{K_2}=0$.  Hence the entire
    two-dimensional induced space at level $K_2$ is the Steinberg space.
    Proposition~\ref{prop:level1}, applied with $n=1$, gives a common
    eigenbasis on which
    $$
    \mathcal Y_1=0,
    \qquad
    \mathcal W_1=\sqrt{p^*}\quad\text{and}\quad-\sqrt{p^*}.
    $$

    For \textbf{Case 3}, Proposition~\ref{prop:steinberg-dimensions} gives
    a two-dimensional induced space at level $K_2$, containing one
    Steinberg line and mapping onto one even Weil line, whereas all three
    $(\ov{K_1},\eta)$-isotypic spaces vanish. By Remark~\ref{rem:Y-averaging},
    the element $\mathcal Y_1=\mathcal I+\mathcal V_{1,1}+\mathcal V_{1,\lambda}$
    is supported on $\ov{K_1}$ and takes the value $\eta(\tilde k)^{-1}$
    at each $\tilde k\in\ov{K_1}$. Under the normalization
    $\operatorname{vol}(\ov{K_2})=1$, one has
    $\operatorname{vol}(\ov{K_1})=[K_1:K_2]=p$, so
    $p^{-1}\mathcal Y_1$ is the normalized averaging projection onto the
    $(\ov{K_1},\eta)$-isotypic space.  That space is zero in Case~3; therefore
    $$
    \mathcal Y_1=0
    $$
    on the entire two-dimensional induced space.

    We now adapt the argument used in Proposition~\ref{prop:level1}; its
    hypotheses do not apply literally here because $\mu_a$ is unramified.
    Let $f_1$ and $f_\lambda$ be the two standard type functions supported
    on $B(y(p),1)K_2$ and $B(y(\lambda p),1)K_2$, respectively.  The
    terminal relation from Corollary~\ref{cor:VW} gives
    $$
    \mathcal W_1^2=p^*\mathcal I
    $$
    on this space.  Thus $\mathcal W_1$ is diagonalizable and its
    eigenvalues belong to
    $\{\sqrt{p^*},-\sqrt{p^*}\}$.

    It remains to show that both eigenvalues occur.  The same direct
    double-coset calculation, now with $c(\mu_a)=0$ and $c(\eta)=1$,
    gives the equality of the two diagonal coefficients
    \[
    \mathcal V_{1,1}(f_1)(\bar y(p))
    =
    \mathcal V_{1,\lambda}(f_\lambda)(\bar y(\lambda p)).
    \tag{**}
    \]
    Suppose that only one $\mathcal W_1$-eigenvalue occurred.  Since
    $\mathcal W_1$ is diagonalizable, it would be scalar.  Because
    $\mathcal Y_1=0$, the identities
    $$
    \mathcal V_{1,1}+\mathcal V_{1,\lambda}=-\mathcal I,
    \qquad
    \mathcal V_{1,1}-\mathcal V_{1,\lambda}=\mathcal W_1
    $$
    would then make both $\mathcal V_{1,1}$ and
    $\mathcal V_{1,\lambda}$ scalar.  Equality~$(**)$ would force their
    two scalar values to be equal, and hence $\mathcal W_1=0$, contrary
    to $\mathcal W_1^2=p^*\mathcal I$.  Thus the two eigenvalues of
    $\mathcal W_1$ are $\sqrt{p^*}$ and $-\sqrt{p^*}$.

    We now determine which constituent receives which sign directly in the
    Schr\"odinger model. Write
    $$
    \chi=\eta=\legendre{\cdot}{p}
    \qquad\text{and}\qquad
    \phi_\eta(z)=\chi(z)\mathbf 1_{\mathbb Z_p^\times}(z).
    $$
    This character factor is essential: the plain characteristic function
    $\mathbf 1_{\mathbb Z_p^\times}$ has trivial diagonal type. Since
    $a\in\{1,\beta\}$ is a unit, $\psi_a$ has conductor $0$ and
    $\gamma(u,\psi_a)=1$ for every $u\in\mathbb Z_p^\times$. Hence
    $$
    \omega_{\psi_a}((h(u),1))\phi_\eta
    =\chi(u)\phi_\eta
    =\eta(h(u))\phi_\eta.
    $$
    The verification of the three type conditions in the proof of
    \cite[Theorem~4.4]{Ishimoto}, applied here with
    $c(\psi_a)=0$, $c(\eta)=1$, and $m=2$, shows that
    $x(\mathbb Z_p)$ fixes $\phi_\eta$ and that its Fourier transform is
    supported on $p^{-1}\mathbb Z_p^\times$; consequently
    $y(p^2\mathbb Z_p)$ also fixes $\phi_\eta$. Thus $\phi_\eta$ belongs
    to the $(\ov{K_2},\eta)$-isotypic space. Since $p\equiv1\pmod4$, one has
    $\chi(-1)=1$, and therefore $\phi_\eta$ is even. It consequently
    spans the one-dimensional space $(\omega_{\psi_a}^+)_\eta^{K_2}$.

    We compute the action of $\mathcal W_1$ on this vector. For
    $\delta\in\{1,\lambda\}$, the right-coset decomposition in
    Lemma~\ref{lem:cosetdecomp} gives
    \begin{align*}
    \mathcal V_{1,\delta}\phi_\eta
    &={}
    \sum_{u\in\mathbb Z_p^\times/\sqrt{1+p\mathbb Z_p}}
    \eta(u)
    \omega_{\psi_a}((h(u),1)(y(\delta p),1))\phi_\eta\\
    &={}
    \sum_{u\in\mathbb Z_p^\times/\sqrt{1+p\mathbb Z_p}}
    \omega_{\psi_a}((y(\delta u^{-2}p),1))\phi_\eta.
    \end{align*}
    Here we used
    $h(u)y(\delta p)=y(\delta u^{-2}p)h(u)$ and the two factors
    $\eta(u)$ and $\chi(u)$ multiply to $1$, since
    $\eta=\chi$ is quadratic. As $u^{-2}$ runs once through
    the nonzero squares in $\mathbb F_p$, subtraction of the two square
    classes gives
    \begin{equation}\label{eq:case3-W1-square-class-sum}
    \mathcal W_1\phi_\eta
    =\sum_{t\in\mathbb F_p^*}\chi(t)
    \omega_{\psi_a}((y(pt),1))\phi_\eta.
    \end{equation}

    Put $\mathcal F_a=\omega_{\psi_a}((w(1),1))$. Since the cover splits
    over $K$,
    $$
    (y(pt),1)=(w(1),1)(x(-pt),1)(w(1),1)^{-1}
    $$
    with no additional cocycle factor. Thus the right-hand side of
    \eqref{eq:case3-W1-square-class-sum} is
    $$
    \mathcal F_a
    \left(
    \sum_{t\in\mathbb F_p^*}\chi(t)
    \omega_{\psi_a}((x(-pt),1))
    \right)
    \mathcal F_a^{-1}\phi_\eta.
    $$
    Up to a nonzero constant,
    $$
    [\mathcal F_a^{-1}\phi_\eta](z)
    =\int_{\mathbb Z_p^\times}\chi(x)\psi(-2axz)\,dx,
    $$
    which is supported on $p^{-1}\mathbb Z_p^\times$. If
    $z=p^{-1}v$ with $v\in\mathbb Z_p^\times$, the operator in parentheses
    acts by the scalar
    \begin{align*}
    \sum_{t\in\mathbb F_p^*}\chi(t)\psi_a(-ptz^2)
    &=\sum_{t\in\mathbb F_p^*}\chi(t)
      \psi\left(-\frac{atv^2}{p}\right)\\
    &=\chi(-a)
      \sum_{t\in\mathbb F_p^*}\chi(t)\psi(t/p)\\
    &=\chi(a)\epsilon_p\sqrt p
     =\chi(a)\sqrt{p^*}.
    \end{align*}
    In the last line we used $\chi(-1)=1$ and
    \eqref{eq:weil-gauss-normalization}. It follows that
    $$
    \mathcal W_1\phi_\eta
    =\legendre{a}{p}\sqrt{p^*}\,\phi_\eta.
    $$
    The exact sequence~\eqref{eq:steinberg-type-exact} is equivariant for
    the action of the Hecke subalgebra $H_2$. Since the Steinberg line and the even Weil
    quotient are the two opposite $\mathcal W_1$-eigenlines found above,
    $\mathcal W_1$ acts on the Steinberg component by
    $-\legendre{a}{p}\sqrt{p^*}$.

    Finally, in \textbf{Case 4} both the even Weil and Steinberg
    constituents have a one-dimensional $(\ov{K_1},\eta)$-isotypic space.  We
    use the affine Iwahori operator $U_1$ supported
    on $K_1w(p^{-1})K_1$.  This operator separates the two
    constituents, although it is not an element of the Hecke subalgebra
    supported on $\ov K$.  We normalize
    $$
    U_1\bigl((w(p^{-1}),1)\bigr)=\epsilon_p,
    $$
    where $\epsilon_p=\gamma(p,\psi)$ as above. The relation of
    \cite[Proposition~3.10(ii)]{B-PI} remains valid with this normalization:
    $$U_1^2 = \epsilon_p(p-1)U_1 + \chi_p(-1)p I.$$

In this case, put \(w=w(1)\).
Let $f_1$ and $f_w$ be the standard basis of
$\pi_\psi(\mu)_\eta^{K_1}$ supported on $BK_1$ and
$B(w,1)K_1$, respectively, with the same normalization as in Case~1:
$$
f_1(1)=1,
\qquad
f_w((w,1))=1.
$$

We have
$$
\mu_a(-p)
=
\begin{cases}
1,&a=p,\\
-1,&a=p\beta.
\end{cases}
$$
Since
$$
\mu_a(-1)=(-1,a)_p=(-1,p)_p
=
\left(\frac{-1}{p}\right),
$$
we have
$$
\mu_a(p)=\mu_a(-p)\left(\frac{-1}{p}\right).
$$

Consequently,
$$
{}[U_1(f)](g)
=
\epsilon_p
\sum_{s\in\mathbb F_p}
f\bigl(g(y(ps),1)(w(p^{-1}),1)\bigr).
$$

We first compute $U_1(f_1)$ at $g=1$. The term $s=0$ lies in
the $B(w,1)\ov{K_1}$-cell and therefore does not contribute.
For $s\in\mathbb F_p^\times$, put
\[
\begin{aligned}
k_s&=\begin{pmatrix}s&0\\p&s^{-1}\end{pmatrix}\in K_1,\\
b_s&:=y(ps)w(p^{-1})k_s=x(p^{-1}s^{-1})\in B_0.
\end{aligned}
\]
Since $y(ps)w(p^{-1})=b_sk_s^{-1}$, the character and cocycle
factors are
\[
\begin{aligned}
\mu_B((b_s,1))=c(b_s,k_s^{-1})&=1,\\
\eta((k_s^{-1},1))=c(y(ps),w(p^{-1}))&=\chi_p(s).
\end{aligned}
\]
Consequently,
\[
[U_1(f_1)](1)
=\epsilon_p\sum_{s\in\mathbb F_p^\times}\chi_p(s)^2
=\epsilon_p(p-1).
\]
We next evaluate at $g=w$. Since
$$
w\,y(ps)\,w(p^{-1})
=
\begin{pmatrix}
-p&s\\
0&-p^{-1}
\end{pmatrix}
=
h(-p)x(-sp^{-1})
\in B_0,
$$
all terms contribute to the $f_w$-coefficient. The cocycle
calculation in this case gives $$c(w(1),y(ps)w(p^{-1}))=\begin{cases}
    1 \text{, if } s=0\\ (s,p)_p \text{, if } s\ne 0
\end{cases}$$
which for $s\ne 0$ equals $c(y(ps),w(p^{-1}))$, so there is no remaining cocycle contribution. Thus every summand is
$\mu_B((h(-p)x(-sp^{-1}),1))$. Since
$\mu=\mu_a|\cdot|^{1/2}$,
$$
\mu(-p)=\mu_a(-p)p^{-1/2}.
$$
Moreover, $\gamma(-1,\psi)=1$, $\gamma(p,\psi)=\epsilon_p$, and hence
$$
\gamma(-p,\psi)
=\gamma(-1,\psi)\gamma(p,\psi)(-1,p)_p
=\epsilon_p\left(\frac{-1}{p}\right).
$$
Consequently,
$$
\begin{aligned}
{}[U_1(f_1)](w)
&=\epsilon_p\sum_{s\in\mathbb F_p}
|-p|\,\mu(-p)\gamma(-p,\psi)^{-1}\\
&=\epsilon_p p\cdot p^{-1}\cdot\mu_a(-p)p^{-1/2}
\left(\frac{-1}{p}\right)\epsilon_p^{-1}\\
&=\mu_a(-p)\left(\frac{-1}{p}\right)p^{-1/2}.
\end{aligned}
$$
Therefore
$$
U_1(f_1)
=
\epsilon_p(p-1)f_1
+
\mu_a(-p)\left(\frac{-1}{p}\right)p^{-1/2}f_w.
$$
We now compute $U_1(f_w)$. Since
$$
w\,y(ps)\,w(p^{-1})
=
\begin{pmatrix}
-p&s\\
0&-p^{-1}
\end{pmatrix}
\in B_0,
$$
we have
$$
{}[U_1(f_w)](w)=0.
$$
At $g=1$,
$$
y(ps)w(p^{-1})
=
\begin{pmatrix}
0&p^{-1}\\
-p&s
\end{pmatrix}.
$$
We claim that this lies in $BwK_1$ only when
$s\equiv0\pmod p$. Indeed, suppose that
$$
y(ps)w(p^{-1})=bwk,
\qquad
k=
\begin{pmatrix}
\alpha&\beta_0\\
\gamma&\delta
\end{pmatrix}\in K_1.
$$
Writing $b=h(r)x(u)$, comparison of the second row gives
$$
\alpha=rp,
\qquad
\beta_0=-rs.
$$
If $s\not\equiv0\pmod p$, then $s\in\mathbb Z_p^\times$. Since
$\beta_0\in\mathbb Z_p$, it follows that $r\in\mathbb Z_p$, and hence
$$
\alpha=rp\in p\mathbb Z_p.
$$
Since also $\gamma\in p\mathbb Z_p$, this implies
$$
\alpha\delta-\beta_0\gamma\in p\mathbb Z_p,
$$
contradicting $\det(k)=1$. Thus only $s=0$ contributes.

For $s=0$, we use only the $Bw$-decomposition of the element
$w(p^{-1})$. In the metaplectic cover,
$$
(w(p^{-1}),1)
=
\left(
h(p^{-1}),
\left(\frac{-1}{p}\right)
\right)(w,1),
$$
because
$$
c(h(p^{-1}),w)
=
(-1,p)_p
=
\left(\frac{-1}{p}\right).
$$
Therefore
$$
\begin{aligned}
{}[U_1(f_w)](1)
&=
\epsilon_p\,
p\,\mu(p^{-1})
\left(\frac{-1}{p}\right)
\gamma(p^{-1},\psi)^{-1}.
\end{aligned}
$$
Since $\mu=\mu_a|\cdot|^{1/2}$ and $\mu_a$ is quadratic,
$$
\mu(p^{-1})
=\mu_a(p^{-1})|p^{-1}|^{1/2}
=\mu_a(-p)\left(\frac{-1}{p}\right)p^{1/2}.
$$
Moreover, $p^{-1}$ and $p$ have the same square class, so
$$
\gamma(p^{-1},\psi)
=
\gamma(p,\psi)
=
\epsilon_p.
$$
It follows that
$$
\begin{aligned}
{}[U_1(f_w)](1)
&=
\epsilon_p\,
p\,
\mu_a(-p)
\left(\frac{-1}{p}\right)^2
p^{1/2}
\epsilon_p^{-1}\\
&=
\mu_a(-p)p^{3/2}.
\end{aligned}
$$
Therefore
$$
U_1(f_w)=\mu_a(-p)p^{3/2}f_1.
$$

Thus, with the usual convention that the columns record the images of
$f_1$ and $f_w$, the matrix of $U_1$ in the ordered basis
$(f_1,f_w)$ is
$$
{}[U_1]_{(f_1,f_w)}
=
\begin{pmatrix}
\epsilon_p(p-1)&\mu_a(-p)p^{3/2}\\
\mu_a(-p)\left(\dfrac{-1}{p}\right)p^{-1/2}&0
\end{pmatrix}.
$$

Its characteristic polynomial is
\[
\begin{aligned}
\det(XI-U_1)
&=X^2-\epsilon_p(p-1)X-\chi_p(-1)p\\
&=(X-\epsilon_pp)(X+\epsilon_p).
\end{aligned}
\]
For both $a=p$ and $a=p\beta$, the eigenspaces are therefore:
\begin{center}
\begin{tabular}{c|c}
\textbf{Eigenvalue}&\textbf{Eigenvector}\\
\hline
$\epsilon_pp$&$p^{3/2}f_1+\mu_a(-p)\epsilon_pf_w$\\[4pt]
$-\epsilon_p$&$p^{1/2}f_1-\mu_a(-p)\epsilon_pf_w$
\end{tabular}
\end{center}

We now compute the action on the even Weil representation. Write $a=pu$,
where $u\in\{1,\beta\}$, and let $\phi=\mathbf{1}_{\mathbb{Z}_p}$
in the Schr\"odinger model of $\omega_{\psi_a}^+$.
Since $a=p$ or $p\beta$, the character $\psi_a$ has conductor $-1$.
The standard Weil formulas show that $\phi$ spans the one-dimensional
$(\ov{K_1},\eta)$-isotypic space of $\omega_{\psi_a}^+$. By the normalization of $U_1$,
$$
{}[U_1\phi](y) = \epsilon_p \sum_{s\in\mathbb F_p} \omega_{\psi_a} \bigl((y(ps),1)(w(p^{-1}),1)\bigr)(\phi)(y).
$$
Using
$$
(y(ps),1)(w(p^{-1}),1)
\left(I,\left(\frac{-1}{p}\right)\right)
= (h(p^{-1}),1)(w,1)(x(-sp^{-1}),1),
$$
together with the Schr\"odinger-model formulas, we obtain
$$
\begin{aligned}
{}[U_1\phi](y) &= \epsilon_p \left(\frac{-1}{p}\right) p^{1/2} \gamma(p^{-1},\psi_a)^{-1} \gamma(\psi_a)\\
&\qquad \times \sum_{s\in\mathbb F_p} \int_{\mathbb{Z}_p} \psi_a\left(-\frac{sz^2}{p}\right) \psi_a\left(\frac{2zy}{p}\right) \,d_{\psi_a,2}z.
\end{aligned}
$$
Since $c(\psi_a)=-1$, the self-dual measure is
$d_{\psi_a,2}z=p^{-1/2}\,dz$. For $z,s\in\mathbb Z_p$,
\[
\psi_a\left(-\frac{sz^2}{p}\right)=\psi(-usz^2)=1,
\qquad
\psi_a\left(\frac{2zy}{p}\right)=\psi(2uzy),
\]
so the sum over $s$ contributes a factor $p$.
It follows that
$$
\begin{aligned}
{}[U_1\phi](y) &= \epsilon_p \left(\frac{-1}{p}\right) p\, \gamma(p^{-1},\psi_a)^{-1} \gamma(\psi_a) \int_{\mathbb{Z}_p}\psi(2uzy)\,dz\\
&= \epsilon_p \left(\frac{-1}{p}\right) p\, \gamma(p^{-1},\psi_a)^{-1} \gamma(\psi_a) \,\mathbf{1}_{\mathbb{Z}_p}(y).
\end{aligned}
$$

By the definition of the relative Weil factor, and since $\psi_u$
has conductor $0$,
\[
\gamma(p^{-1},\psi_a)
=\frac{\gamma(\psi_u)}{\gamma(\psi_a)}
=\gamma(\psi_a)^{-1}.
\]
Moreover, $\gamma(-1,\psi)=1$ and $(-1,a)_p=\chi_p(-1)$, so the
Weil-factor identities give
\[
\begin{aligned}
\gamma(p^{-1},\psi_a)^{-1}\gamma(\psi_a)
&=\gamma(\psi_a)^2\\
&=\gamma(-1,\psi_a)^{-1}=\chi_p(-1).
\end{aligned}
\]
We conclude that
\[
[U_1\phi](y)=\epsilon_p\chi_p(-1)^2p\,\mathbf1_{\mathbb Z_p}(y)
=\epsilon_pp\,\phi(y).
\]
Comparing this with the preceding eigenspace decomposition
and using the $U_1$-equivariance of
\eqref{eq:steinberg-type-exact}, we conclude that the Steinberg kernel is
the $-\epsilon_p$-eigenline, spanned by
$$
p^{1/2}f_1-\mu_a(-p)\epsilon_p f_w.
$$

\end{proof}

\subsection{Supercuspidal Representations}\label{subsec:supercuspidal}
We now consider irreducible genuine supercuspidal representations of $\widetilde G$. For the classification and the corresponding linear theory we refer to \cite[\S5.1]{Ishimoto}, \cite[\S3.3]{L-R}, and the original work of Manderscheid \cite{ManderscheidI,ManderscheidII}.
Let $\alpha=\begin{pmatrix}
    p  & 0\\ 0 &1
\end{pmatrix}$ and consider the maximal compact subgroups of $G$:
$$
\begin{aligned}
K^0&=K=\SL_2(\Z_p),\\
K^1&=\alpha^{-1}K^0\alpha=
\left\{
\begin{pmatrix}a&b\\c&d\end{pmatrix}\in\SL_2(\Q_p):
\begin{array}{c}
a,d\in\Z_p,\quad b\in p^{-1}\Z_p,\\
c\in p\Z_p
\end{array}
\right\}.
\end{aligned}
$$
We retain $K=K^0=\SL_2(\mathbb Z_p)$ throughout.
When treating the two maximal compact subgroups uniformly, we write
$\mathcal K\in\{K^0,K^1\}$.

We use the isomorphism $\iota:\widetilde G_c\to\widetilde G_{\sigma_0}$
from Section~2 to compare the compact splittings.
For $\delta\in\{0,1\}$, let $s_{\sigma_0}^{\delta}$ be Ishimoto's
splitting of $K^\delta$ in the $\sigma_0$-model; see
\cite[\S3.1]{Ishimoto}. Pulling it back through $\iota^{-1}$ gives
\[
s_c^\delta(k)=s_p(k)s_{\sigma_0}^\delta(k),
\qquad k\in K^\delta.
\]
On the standard generators below, $s_p$ equals $1$: the lower-left
entry is zero for $x(b)$ and $h(a)$; for $y(c)$ the factor is
$(c,1)_p=1$; and $w(1)h(p^\delta)$ has lower-right entry zero.
Thus the two cocycle models give the same splitting values on these
generators. Let $\psi_{-}$ be any additive character of conductor $-1$.
Writing $s^\delta=s_c^\delta$, we have:

\begin{tabular}{ |c|c| }
 \hline
 $s^0$ & $s^1$ \\
 \hline

 $s^0(x(b))=1,\ b\in\mathbb Z_p$
 &
 $s^1(x(b))=1,\ b\in p^{-1}\mathbb Z_p$
 \\
 \hline

 $s^0(h(a))=1,\ a\in\mathbb Z_p^{\times}$
 &
 $s^1(h(a))=\gamma(a,\psi_{-}),\ a\in\mathbb Z_p^{\times}$
 \\
 \hline

 $s^0(y(c))=1,\ c\in\mathbb Z_p$
 &
 $s^1(y(c))=1,\ c\in p\mathbb Z_p$
 \\
 \hline

 $s^0\!\left(
 \begin{pmatrix}
 0 & 1\\
 -1 & 0
 \end{pmatrix}
 \right)=1$
 &
 $s^1\!\left(
 \begin{pmatrix}
 0 & p^{-1}\\
 -p & 0
 \end{pmatrix}
 \right)=1$
 \\
 \hline
\end{tabular}

With $\psi$ the fixed conductor-zero character of Section~2, we can write
$\psi_{-}=\psi_{pb}$ for some $b\in\mathbb Z_p^\times$.
The Weil-factor identities give, for $u\in\mathbb Z_p^\times$,
$$
s^1(h(u))=\gamma(u,\psi_{-})=(u,pb)_p\gamma(u,\psi)
=(u,p)_p=\left(\frac{u}{p}\right),
$$
since $\gamma(u,\psi)=1$ and $(u,b)_p=1$ for units $u,b$.

For $\mathcal K=K^\delta$, write $s=s^\delta$. The splitting identifies
irreducible representations of $\mathcal K$ with genuine irreducible
representations of $\widetilde{\mathcal K}$. Explicitly, a representation
$(\rho,W)$ of $\mathcal K$ gives the genuine representation
$\widetilde\rho$ defined by
\[
\widetilde\rho((k,\epsilon))=\epsilon s(k)\rho(k),
\qquad k\in\mathcal K,\quad \epsilon\in\{\pm1\}.
\]
This notation will also be used after replacing $\rho$ by a conjugate
representation or by a representation induced from the Iwahori subgroup.

Put
\[
N_r=\{x(t):t\in p^r\mathbb Z_p\},
\qquad
\overline N_r=\{y(t):t\in p^r\mathbb Z_p\},
\]
and, inside $K^\delta$, consider
\[
J_\ell^\delta=
\left\{\begin{pmatrix}a&b\\c&d\end{pmatrix}\in K^\delta:
\begin{array}{l}
a,d\in1+p^\ell\mathbb Z_p,\\
b\in p^{\ell-\delta}\mathbb Z_p,\quad
c\in p^{\ell+\delta}\mathbb Z_p
\end{array}\right\}.
\]
Following \cite[\S5.1]{Ishimoto} and Manderscheid
\cite{ManderscheidI,ManderscheidII}, we define the conductor $c(\rho)$
of a representation $\rho$ of $K^\delta$ as the minimal $m$ such that
$\rho$ is trivial on $J_m^\delta$. We say that $\rho$ is strongly cuspidal if
\[
\Hom_{N_{c(\rho)-2\delta-1}}
\bigl(\rho|_{N_{c(\rho)-2\delta-1}},1\bigr)=0.
\]
A strongly cuspidal representation $\rho$ has \emph{defect $1$} if
\[
\Hom_{N_{c(\rho)-\delta-1}}
\bigl(\rho|_{N_{c(\rho)-\delta-1}},1\bigr)\ne0;
\]
otherwise it has \emph{defect $0$}. We denote the defect by $d(\rho)$.
If $d(\rho)=1$, then necessarily $\delta=1$.

We use Manderscheid's classification
\cite{ManderscheidI,ManderscheidII} in the form stated in
\cite[Theorem~5.2]{Ishimoto}. The isomorphism $\iota$ identifies the
corresponding genuine lifts and preserves compact induction, so the
theorem applies in our $c$-model: for every irreducible strongly cuspidal
representation $\rho$ of $\mathcal K=K^\delta$,
\[
\operatorname{c-Ind}_{\widetilde{\mathcal K}}^{\widetilde G}
\widetilde\rho
\]
is irreducible and genuine supercuspidal; every irreducible genuine
supercuspidal representation of $\widetilde G$ arises in this way;
and inequivalent strongly cuspidal types induce inequivalent
representations. Our aim is to find the newvectors in these compact
inductions.

We first recall the linear data from which the maximal-compact types used
below arise.  Let
$$
K_{\GL}=\GL_2(\mathbb Z_p),
\qquad
I_{\GL}=\left\{
\begin{pmatrix}a&b\\c&d\end{pmatrix}\in K_{\GL}:c\in p\mathbb Z_p
\right\},
$$
let $Z$ be the center of $\GL_2(\mathbb Q_p)$, and let $Z'$ be the subgroup
generated by
$$
\varpi_I=\begin{pmatrix}0&1\\p&0\end{pmatrix}.
$$
Thus $ZK_{\GL}$ and $Z'I_{\GL}$ are the two compact-mod-center subgroups
used in the Kutzko--Sally construction.  Following
\cite[Definition~3.3.1]{L-R}, the linear very cuspidal data occur in two
cases:
$$
\begin{array}{ccl}
\text{unramified:}&&\tau\text{ is a representation of }ZK_{\GL},\\
\text{ramified:}&&\tau\text{ is a representation of }Z'I_{\GL}.
\end{array}
$$
In the corresponding case, level $\ell\geq1$ means that $\tau$ is trivial
on
$$
K_{\GL}(\ell)=1+p^\ell M_2(\mathbb Z_p)
$$
or on
$$
I_{\GL}(\ell)=\left\{
\begin{pmatrix}a&b\\c&d\end{pmatrix}\in I_{\GL}:
a,d\in1+p^\ell\mathbb Z_p,\quad b\in p^\ell\mathbb Z_p,\quad
c\in p^{\ell+1}\mathbb Z_p
\right\},
$$
and in either case
$$
\Hom_{N_{\ell-1}}(1,\tau)=0.
$$

The restriction of an unramified datum $\tau$ to
$K_{\GL}\cap G=K^0$ is usually irreducible; in the exceptional case
$\ell=1$ it may split into two irreducible constituents.  The restriction
of a ramified datum $\tau$ to
$$
I=I_{\GL}\cap G
$$
always splits as $\sigma_1\oplus\sigma_2$, and conjugation by a diagonal
nonsquare-unit element interchanges the two constituents; see
\cite[\S3.3]{L-R}.

For clarity, we now define the level of these Iwahori constituents.  Put
$$
I_\ell=I_{\GL}(\ell)\cap G.
$$
We say that a representation $\sigma$ of $I$ has \emph{Iwahori level
$\ell$} if
$$
\sigma|_{I_\ell}=1
\qquad\text{and}\qquad
\Hom_{N_{\ell-1}}(1,\sigma)=0.
$$
Each constituent of a ramified very cuspidal datum of level $\ell$ has
Iwahori level $\ell$.  In addition, Lansky--Raghuram
\cite[Remark~3.3.2]{L-R} prove for the ramified datum $\tau$ that
$$
\Hom_{N_{\ell-1}}(1,\tau)
=
\Hom_{\overline N_\ell}(1,\tau)
=0.
$$
Passing to either constituent gives
$$
\Hom_{\overline N_\ell}(1,\sigma_i)=0,
\qquad i=1,2.
$$
This lower-unipotent assertion uses the extension to the ramified very
cuspidal representation of $Z'I_{\GL}$; it need not hold for an arbitrary
representation of $I$ having the two properties in the definition above.

We record one consequence needed for the odd ramified construction.  If
$$
\rho=\operatorname{Ind}_I^{K^0}\sigma_i,
$$
then $J_{\ell+1}^0\subset I_\ell$ and the normality of
$J_{\ell+1}^0$ in $K^0$ give $c(\rho)\leq\ell+1$.  On the other hand,
$\sigma_i$ is trivial on $\overline N_{\ell+1}\subset I_\ell$ and has no
$\overline N_\ell$-fixed vector.  Hence $\rho$ is nontrivial on $J_\ell^0$,
through
$$
J_\ell^0/I_\ell\simeq
\overline N_\ell/\overline N_{\ell+1},
$$
and therefore
$$
c(\rho)=\ell+1.
$$

The following observation also determines the defect after passing from
$K^0$ to $K^1$.

\begin{lem}\label{lem:odd-ramified-defect}
Let $\sigma$ be an irreducible Iwahori constituent of a ramified very
cuspidal datum of level $\ell$, and put
$$
\rho=\operatorname{Ind}_I^{K^0}\sigma,
\qquad
\rho^\alpha(k)=\rho(\alpha k\alpha^{-1}),\qquad k\in K^1.
$$
Then $\rho^\alpha$ is an irreducible strongly cuspidal representation of
$K^1$ with
$$
c(\rho^\alpha)=\ell+1,
\qquad
d(\rho^\alpha)=1.
$$
\end{lem}

\begin{proof}
Put $I'=\alpha^{-1}I\alpha$ and define
$$
\sigma^\alpha(i')=\sigma(\alpha i'\alpha^{-1}),
\qquad i'\in I'.
$$
Functoriality of induction under the isomorphism
$\operatorname{Ad}_\alpha:K^1\to K^0$ gives
$$
\rho^\alpha\simeq\operatorname{Ind}_{I'}^{K^1}\sigma^\alpha.
$$
Let
$$
w_1=w(p^{-1})=\begin{pmatrix}0&p^{-1}\\-p&0\end{pmatrix}\in K^1.
$$
A direct matrix calculation gives
$$
w_1I'w_1^{-1}=I.
$$
Conjugating the induced model by $w_1$ therefore gives
$$
\rho^\alpha\simeq\operatorname{Ind}_I^{K^1}\sigma',
$$
where
$$
\begin{aligned}
\sigma'(i)
&=\sigma^\alpha(w_1^{-1}iw_1)\\
&=\sigma(\gamma i\gamma^{-1}),
\qquad
\gamma=\alpha w_1^{-1}
=\begin{pmatrix}0&-1\\p&0\end{pmatrix}.
\end{aligned}
$$
Since
$$
\gamma=\varpi_I\begin{pmatrix}1&0\\0&-1\end{pmatrix}
\in Z'I_{\GL},
$$
the representation $\sigma'$ is again an irreducible constituent of the
restriction of the same ramified very cuspidal datum to $I$.  Ishimoto's
\cite[Lemma~5.5(2)]{Ishimoto}, which records the corresponding
Kutzko--Sally--Manderscheid result, now shows that
$\operatorname{Ind}_I^{K^1}\sigma'$ is irreducible and strongly cuspidal,
of conductor $\ell+1$ and defect $1$.
\end{proof}

We next introduce the weight spaces used in the explicit vectors.  Put
$$
U=\{h(u):u\in\mathbb Z_p^\times\}.
$$
For a character $\eta:\mathbb Z_p^\times\to\mathbb C^\times$ satisfying the
usual central compatibility condition, let
$$
W^U_{\eta}
=
\{w\in W:\rho(h(u))w=\eta(u)w
\text{ for every }u\in\mathbb Z_p^\times\}.
$$
When a level $\ell$ is fixed below, we write $\eta_\ell$ for the
character occurring in the linear $U$-weight space.  In the unramified
construction,
$$
\eta_\ell=
\begin{cases}
\eta,&\ell\ \text{even},\\
\eta(\frac{\cdot}{p}),&\ell\ \text{odd},
\end{cases}
$$
while in the ramified normalization we use
$\eta_\ell=\eta(\frac{\cdot}{p})$.
Here $\eta$ itself always denotes the character defining the
$(\ov{K_r},\eta)$-isotypic space and the corresponding Hecke action.
We recall the dimensions of the spaces that occur in the supercuspidal
constructions.  By \cite[Propositions 3.3.4, 3.3.6, 3.3.8]{L-R} and the
Kutzko--Sally classification recalled in \cite[\S5.1]{Ishimoto} (see also \cite{KutzkoSally}), if the corresponding
unramified supercuspidal $L$-packet has cardinality two, the restriction of
the very cuspidal $GL_2$-type to $K^0$ is irreducible and
$\dim W^U_{\eta_\ell}=2$.
For the exceptional unramified packet of cardinality four one has $\ell=1$
and the restriction splits as $\sigma_1\oplus\sigma_2$; for the space $W_i$
of each irreducible constituent, $\dim (W_i)^U_{\eta_\ell}=1$.
In the ramified case the restriction to the Iwahori subgroup of $\SL_2$
splits as $\sigma_1\oplus\sigma_2$, and again
$\dim (W_i)^U_{\eta_\ell}=1$ for each constituent.  Thus the two-dimensional $U$-weight space belongs to
a single inducing type only in the unramified packet of cardinality two.
Fix a nonsquare unit $\lambda\in\mathbb Z_p^\times$.

We shall keep three levels distinct.  The integer $\ell$ is the level of
the underlying very cuspidal or Iwahori datum.  The associated strongly
cuspidal maximal-compact type $(\rho,W)$ has conductor and defect
$$
(c(\rho),d(\rho))=
\begin{cases}
(\ell,0),&\text{in the unramified case},\\
(\ell+1,1),&\text{in the ramified case}.
\end{cases}
$$
Finally, for the choices of maximal compact subgroup made below and for
$c(\eta)\leq1$, the compact induction $\widetilde\pi$ has
$\eta$-conductor $2\ell$ in the unramified case and $2\ell+1$ in the
ramified case.  For the unramified construction and the even ramified
construction, the assertions about $(c(\rho),d(\rho))$ follow from
\cite[Lemma~5.5]{Ishimoto}, and the $\eta$-conductors then follow from
\cite[Theorem~5.11 and Corollary~5.12]{Ishimoto}.  The odd ramified
construction in Proposition~\ref{prop:supercuspidal-newvectors-four-cases}(iv)
is our separate construction.  Its conductor and defect are established
in Lemma~\ref{lem:odd-ramified-defect}.  Thus ``level
$\ell$'' for the linear datum, $c(\rho)$ for the maximal-compact type,
and $c_\eta(\widetilde\pi)$ are not interchangeable.

\begin{lem}\label{lem:linear-supercuspidal-hecke}
Let $\mathcal K$ be either $K^0$ or $K^1$, let $(\rho,W)$ be a representation of
$\mathcal K$, and let
$\chi:\mathbb Z_p^\times\to\mathbb C^\times$ be a character.  Suppose that
$w\in W^U_{\chi}$ is fixed by $\overline N_\ell$ and that
$$
\Hom_{\overline N_{\ell-1}}(1,\rho)=0.
$$
For $\delta\in\{1,\lambda\}$ define
$$
\mathcal T_{\ell,\delta}(w)
=
\sum_{u\in\mathbb Z_p^\times/\sqrt{1+p\mathbb Z_p}}
\chi(u)^{-1}\rho\bigl(h(u)y(\delta p^{\ell-1})\bigr)w.
$$
Then $\mathcal T_{\ell,\delta}(w)\in W^U_\chi$ for each
$\delta\in\{1,\lambda\}$, and
$$
\bigl(I+\mathcal T_{\ell,1}+\mathcal T_{\ell,\lambda}\bigr)w=0.
$$
\end{lem}

\begin{proof}
We give the calculation exactly in the linear model.  For
$\delta\in\{1,\lambda\}$ we have
$$
\begin{aligned}
\mathcal T_{\ell,\delta}(w)
&=
\sum_{u\in\mathbb Z_p^\times/\sqrt{1+p\mathbb Z_p}}
\chi(u)^{-1}\rho\bigl(h(u)y(\delta p^{\ell-1})\bigr)w\\
&=
\sum_{u\in\mathbb Z_p^\times/\sqrt{1+p\mathbb Z_p}}
\chi(u)^{-1}\rho\bigl(y(\delta p^{\ell-1}u^{-2})h(u)\bigr)w\\
&=
\sum_{u\in\mathbb Z_p^\times/\sqrt{1+p\mathbb Z_p}}
\rho\bigl(y(\delta p^{\ell-1}u^{-2})\bigr)w,
\end{aligned}
$$
where we used
$$
h(u)y(t)=y(tu^{-2})h(u)
$$
and $\rho(h(u))w=\chi(u)w$.

Now
$$
\overline N_{\ell-1}/\overline N_\ell
\simeq p^{\ell-1}\mathbb Z_p/p^\ell\mathbb Z_p
\simeq \mathbb F_p.
$$
The zero class gives the identity operator.  The nonzero classes split into
the two square classes: the elements
$$
\delta u^{-2}\pmod p,
\qquad
u\in\mathbb Z_p^\times/\sqrt{1+p\mathbb Z_p},
\qquad
\delta\in\{1,\lambda\},
$$
run exactly once through $\mathbb F_p^*$.  Hence, with Haar measure on
$\overline N_{\ell-1}$ normalized by
$\operatorname{vol}(\overline N_\ell)=1$, the preceding two sums together
with the identity term give
$$
\bigl(I+\mathcal T_{\ell,1}+\mathcal T_{\ell,\lambda}\bigr)w
=
\int_{\overline N_{\ell-1}}\rho(n)w\,dn.
$$
The last integral is a nonzero scalar multiple of the averaging projection
onto the trivial $\overline N_{\ell-1}$-isotypic subspace.  Since
$$
\Hom_{\overline N_{\ell-1}}(1,\rho)=0
$$
by hypothesis, the integral is zero.  Therefore
$$
\bigl(I+\mathcal T_{\ell,1}+\mathcal T_{\ell,\lambda}\bigr)w=0.
$$
\end{proof}

\begin{corlem}[For the inducing types]\label{cor:linear-supercuspidal-hecke}
\begin{enumerate}[label=(\roman*)]
\item Let $(\rho,W)$ be a strongly cuspidal representation of $K^0$ with
$c(\rho)=\ell$ and $d(\rho)=0$.  Then every $w\in W^U_\chi$ is fixed by
$\overline N_\ell$ and
$$
\Hom_{\overline N_{\ell-1}}(1,\rho)=0.
$$
Thus every such $w$ satisfies the hypotheses of
Lemma~\ref{lem:linear-supercuspidal-hecke}.

\item Let $(\rho,W)$ be a strongly cuspidal representation of $K^1$ with
$c(\rho)=\ell+1$ and defect $1$.  Then
$$
\Hom_{\overline N_\ell}(1,\rho)=0.
$$
Consequently, every $\phi\in W^U_\chi$ that is fixed by
$\overline N_{\ell+1}$ satisfies Lemma~\ref{lem:linear-supercuspidal-hecke}
with $\ell+1$ in place of $\ell$.
\end{enumerate}
\end{corlem}

\begin{proof}
For (i), the inclusion $\overline N_\ell\subset J_\ell^0$ shows that
every vector of $W$ is $\overline N_\ell$-fixed.  Strong cuspidality gives
$\Hom_{N_{\ell-1}}(1,\rho)=0$, and conjugation by $w(1)\in K^0$ carries
$N_{\ell-1}$ onto $\overline N_{\ell-1}$.

Here and below we use the equivalence between the absence of invariant
vectors and the absence of invariant linear forms.  It follows from
complete reducibility for finite-dimensional representations of compact
groups and reconciles this notation with the definition of strong
cuspidality above.

For (ii), strong cuspidality gives
$\Hom_{N_{\ell-2}}(1,\rho)=0$.  Conjugation by $w(p^{-1})\in K^1$ carries
$N_{\ell-2}$ onto $\overline N_\ell$, and hence
$\Hom_{\overline N_\ell}(1,\rho)=0$.  Together with the assumed
$\overline N_{\ell+1}$-fixedness of $\phi$, these are precisely the two
hypotheses of Lemma~\ref{lem:linear-supercuspidal-hecke} with $\ell+1$
in place of $\ell$.
\end{proof}

Corollary~\ref{cor:linear-supercuspidal-hecke} records the filtration
facts needed for the types and vectors in
Proposition~\ref{prop:supercuspidal-newvectors-four-cases}; it will be
used again in the Hecke-action calculation of
Proposition~\ref{lem:metaplectic-linear-bridge}.

We next give the uniform formula for the vectors in the compactly induced
representation, exactly so that the cocycle verification is performed only
once.

\begin{lem}\label{lem:general-fw}
Let $\mathcal K=K^0$ or $K^1$, let $(\widetilde\rho,W)$ be a genuine representation
of $\widetilde{\mathcal K}$, and put
$$
\widetilde\pi=\operatorname{c-Ind}_{\widetilde{\mathcal K}}^{\widetilde G}
\widetilde\rho.
$$
For $w\in W$ define
$$
f_w((g,\epsilon))=
\begin{cases}
 c(k,h(p^m))\,\epsilon\,\widetilde\rho((k,1))w,
 &g=kh(p^m),\quad k\in\mathcal K,\\
 0,&g\notin\mathcal Kh(p^m).
\end{cases}
$$
Then $f_w\in\widetilde\pi$.
\end{lem}

\begin{proof}
We have to prove the defining left $\widetilde{\mathcal K}$-equivariance.  Namely, for
$(k_1,\epsilon_1)\in\widetilde{\mathcal K}$ and $(g,\epsilon)\in\widetilde G$ we want
$$
f_w\bigl((k_1,\epsilon_1)(g,\epsilon)\bigr)
=
\widetilde\rho((k_1,\epsilon_1))f_w((g,\epsilon)).
$$
If $g\notin\mathcal Kh(p^m)$, then also $k_1g\notin\mathcal Kh(p^m)$, and both sides are zero.
Assume therefore that
$$
g=kh(p^m),
\qquad k\in\mathcal K.
$$
Recall the cocycle identity
$$
c(g_1,g_2)c(g_1g_2,g_3)
=
c(g_1,g_2g_3)c(g_2,g_3).
$$
On the one hand,
$$
(k_1,\epsilon_1)(g,\epsilon)
=
(k_1kh(p^m),\epsilon_1\epsilon c(k_1,g)),
$$
so, by the definition of $f_w$,
$$
\begin{aligned}
f_w\bigl((k_1,\epsilon_1)(g,\epsilon)\bigr)
&=
\epsilon\epsilon_1c(k_1,g)c(k_1k,h(p^m))
\widetilde\rho((k_1k,1))w.
\end{aligned}
$$
On the other hand,
$$
\begin{aligned}
\widetilde\rho((k_1,\epsilon_1))f_w((g,\epsilon))
&=
\epsilon\epsilon_1c(k,h(p^m))
\widetilde\rho((k_1,1))\widetilde\rho((k,1))w\\
&=
\epsilon\epsilon_1c(k,h(p^m))c(k_1,k)
\widetilde\rho((k_1k,1))w.
\end{aligned}
$$
Thus it remains to prove
$$
c(k_1,g)c(k_1k,h(p^m))=c(k,h(p^m))c(k_1,k).
$$
Since $g=kh(p^m)$, this is
$$
c(k_1,kh(p^m))c(k_1k,h(p^m))=c(k,h(p^m))c(k_1,k).
$$
Taking $g_1=k_1$, $g_2=k$, $g_3=h(p^m)$ in the cocycle identity gives exactly
$$
c(k_1,k)c(k_1k,h(p^m))
=
c(k_1,kh(p^m))c(k,h(p^m)).
$$
Because all cocycle values are in $\{\pm1\}$, this is equivalent to the
required equality.  Hence $f_w$ satisfies the compact-induction
transformation law and therefore belongs to $\widetilde\pi$.
\end{proof}

In the following proposition, the linear representations under consideration are obtained from an unramified or ramified very cuspidal datum $\tau$, as described above.
In the ramified cases, $\phi_w$ is a vector
in the indicated induced representation of $K^0$ or $K^1$.
Case~(iii) uses the passage from an Iwahori constituent to a
strongly cuspidal $K^1$-type described in
\cite[Lemma~5.5(2)]{Ishimoto}.
Case~(iv) is a separate construction:
we first induce from $I$ to $K^0$ and then conjugate by
$\alpha$ to obtain a representation of $K^1$.

\begin{prop}\label{prop:supercuspidal-newvectors-four-cases}
Let
$$
\widetilde\pi
=
\operatorname{c-Ind}_{\widetilde{\mathcal K}}^{\widetilde G}\widetilde\rho,
$$
where $\mathcal K=K^0$ in case~(i) and $\mathcal K=K^1$ in cases~(ii)--(iv).
The vectors obtained from a strongly cuspidal type are described as
follows.

\begin{enumerate}[label=(\roman*)]
\item \textbf{Unramified, $\ell$ even.}
Write $m=\ell/2$.  Let $(\rho,W)$ be an irreducible strongly cuspidal
representation of $K^0$, arising from an unramified datum of level
$\ell$, so that $c(\rho)=\ell$ and $d(\rho)=0$.  Let
$w\in W^U_{\eta_\ell}$, and let $\widetilde\rho$ be the genuine lift of
$\rho$ with respect to $s^0$.  Then
$$
f_w\in\widetilde\pi_\eta^{K_{2\ell}}.
$$

\item \textbf{Unramified, $\ell$ odd.}
Write $m=(\ell-1)/2$.  Starting with an irreducible strongly cuspidal
representation $(\rho,W)$ of $K^0$ arising from an unramified datum of
level $\ell$, so that $c(\rho)=\ell$ and $d(\rho)=0$, put
$$
\rho^\alpha(k)=\rho(\alpha k\alpha^{-1}),
\qquad
\alpha=\begin{pmatrix}p&0\\0&1\end{pmatrix},
$$
view $\rho^\alpha$ as a representation of $K^1$, and let
$\widetilde\rho$ be its genuine lift with respect to $s^1$.  The
conjugated type again has conductor $\ell$ and defect $0$.  For the
corresponding $w\in W^U_{\eta_\ell}$ one has
$$
f_w\in\widetilde\pi_\eta^{K_{2\ell}}.
$$

\item \textbf{Ramified, $\ell$ even.}
Write $m=\ell/2$.  Let $(\sigma,W)$ be the Iwahori constituent arising
from a ramified datum of level $\ell$ and put
$$
\rho=\operatorname{Ind}_{I}^{K^1}\sigma.
$$
Then $\rho$ is irreducible and strongly cuspidal, with
$c(\rho)=\ell+1$ and $d(\rho)=1$, by
\cite[Lemma~5.5(2)]{Ishimoto}.
For $w\in W^U_{\eta_\ell}$ define $\phi_w\in\rho$ by
$$
\phi_w(k)=
\begin{cases}
\sigma(k)w,&k\in I,\\
0,&k\notin I.
\end{cases}
$$
Then $\phi_w$ is fixed by $N_\ell$ and $\overline N_{\ell+1}$ and $\rho(h(u))\phi_w=\eta_{\ell}(u)\phi_w$.  Let $\widetilde\rho$ be the genuine lift of
$\rho$ with respect to $s^1$.  Then
$$
f_{\phi_w}\in\widetilde\pi_\eta^{K_{2\ell+1}}.
$$

\item \textbf{Ramified, $\ell$ odd.}
Write $m=(\ell-1)/2$.  Let $(\sigma,W)$ be the corresponding conjugated
Iwahori constituent arising from a ramified datum of level $\ell$ and put
$$
\rho=\operatorname{Ind}_{I}^{K^0}\sigma.
$$
Take $w\in W^U_{\eta_\ell}$, define $\phi_w$ as in (iii), and pass to the
$\alpha$-conjugate representation $\rho^\alpha$ on $K^1$.
By Lemma~\ref{lem:odd-ramified-defect}, $\rho^\alpha$ is irreducible and
strongly cuspidal, with
$$
c(\rho^\alpha)=\ell+1,
\qquad
d(\rho^\alpha)=1.
$$
Let $\widetilde\rho$ be the genuine lift of $\rho^\alpha$ with respect to
$s^1$.  Then
$$
f_{\phi_w}\in\widetilde\pi_\eta^{K_{2\ell+1}}.
$$
\end{enumerate}
Corollary~\ref{cor:linear-supercuspidal-hecke} applies to every
$w\in W^U_{\eta_\ell}$ in the unramified cases.  In the even ramified
case, the displayed $\overline N_{\ell+1}$-fixedness of $\phi_w$,
together with part~(ii) of that corollary, supplies the filtration
hypotheses needed below.  In the odd ramified case the assertion
$c(\rho^\alpha)=\ell+1$, $d(\rho^\alpha)=1$ establishes the required
maximal-compact type.  The vector-level averaging needed for the later
linear Hecke calculation follows directly from its $I$-supported induced
model and is verified in
Proposition~\ref{lem:metaplectic-linear-bridge}.
Since $c(\eta)\leq1$, Theorem~5.11 of \cite{Ishimoto} gives
$c_\eta(\widetilde\pi)=2\ell$ in (i)--(ii) and
$c_\eta(\widetilde\pi)=2\ell+1$ in (iii)--(iv); Corollary~5.12
loc.~cit. identifies the nonzero vectors constructed above as newvectors.
\end{prop}

\begin{proof}
We treat the four cases separately.
The right action on the compact induction is
$$
{}[\widetilde\pi(\widetilde g_0)f](\widetilde g)
=f(\widetilde g\widetilde g_0).
$$
For $K_r=K_0(p^r)$ it is enough to examine the action of the upper
unipotents $x(a)$ with $a\in\mathbb Z_p$, of the diagonal elements $h(u)$
with $u\in\mathbb Z_p^\times$, and of the lower unipotents $y(t)$ with
$t\in p^r\mathbb Z_p$.

\medskip
\noindent\textbf{I. The unramified cases.}
Let $g=kh(p^m)$ with $k$ in the relevant maximal compact subgroup.

\smallskip
\noindent\emph{(1) Upper unipotents.}
For $a\in\mathbb Z_p$ we have
$$
h(p^m)x(a)=x(p^{2m}a)h(p^m),
$$
and therefore
$$
\begin{aligned}
{}[\widetilde\pi((x(a),1))f_w]((g,\epsilon))
&=f_w\bigl((g,\epsilon)(x(a),1)\bigr)\\
&=f_w\bigl((k x(p^{2m}a)h(p^m),
\epsilon c(g,x(a)))\bigr)\\
&=c(kx(p^{2m}a),h(p^m))\,\epsilon c(g,x(a))
\widetilde\rho((kx(p^{2m}a),1))w.
\end{aligned}
$$
We now distinguish the parity of $\ell$.

If $\ell$ is even, then $2m=\ell$ and
$x(p^{2m}a)=x(p^\ell a)\in N_\ell$.  Since $w$ is $N_\ell$-fixed,
$$
\widetilde\rho((kx(p^{2m}a),1))w
=
\widetilde\rho((k,1))w.
$$
For the chosen splitting over $K^0$, the cocycle factor
$c(k,x(p^{2m}a))$ is $1$, so this may also be written in the uniform form
$$
\widetilde\rho((kx(p^{2m}a),1))w
=
c(k,x(p^{2m}a))\widetilde\rho((k,1))w.
$$

If $\ell$ is odd, then $2m=\ell-1$.  Put
$x'=x(p^{\ell-1}a)$.  Since
$$
\alpha x'\alpha^{-1}=x(p^\ell a)\in N_\ell,
$$
we have $\rho^\alpha(x')w=w$.  Moreover $s^1(x')=1$, and the splitting
identity gives
$$
s^1(kx')=s^1(k)s^1(x')c(k,x')=s^1(k)c(k,x').
$$
Consequently
$$
\begin{aligned}
\widetilde\rho((kx',1))w
&=s^1(kx')\rho^\alpha(kx')w\\
&=s^1(k)c(k,x')\rho^\alpha(k)w\\
&=c(k,x')\widetilde\rho((k,1))w.
\end{aligned}
$$
Thus in both parity cases
$$
\widetilde\rho((kx(p^{2m}a),1))w
=
c(k,x(p^{2m}a))\widetilde\rho((k,1))w.
$$
It remains to prove
$$
c(kx(p^{2m}a),h(p^m))c(kh(p^m),x(a))
=
c(k,h(p^m))c(k,x(p^{2m}a)).
$$
Set $x'=x(p^{2m}a)$.  Since $x'h(p^m)=h(p^m)x(a)$, the cocycle identity gives
$$
c(k,x')c(kx',h(p^m))=c(k,x'h(p^m))c(x',h(p^m))
$$
and
$$
c(k,h(p^m))c(kh(p^m),x(a))=c(k,h(p^m)x(a))c(h(p^m),x(a)).
$$
The middle arguments are equal because $x'h(p^m)=h(p^m)x(a)$.  A direct evaluation
of the cocycle gives
$$
c(x',h(p^m))=c(h(p^m),x(a))=1.
$$
Hence the required equality follows, and therefore
$$
\widetilde\pi((x(a),1))f_w=f_w.
$$

\smallskip
\noindent\emph{(2) Diagonal units.}
Let $u\in\mathbb Z_p^\times$.  Since $h(p^m)$ commutes with $h(u)$,
$$
\begin{aligned}
{}[\widetilde\pi((h(u),1))f_w]((g,\epsilon))
&=f_w\bigl((kh(u)h(p^m),\epsilon c(g,h(u)))\bigr)\\
&=c(kh(u),h(p^m))\,\epsilon c(g,h(u))
\widetilde\rho((kh(u),1))w.
\end{aligned}
$$
For $\ell$ even, the $U$-weight condition gives directly
$$
\widetilde\rho((kh(u),1))w
=
\eta(u)c(k,h(u))\widetilde\rho((k,1))w.
$$
For $\ell$ odd we use the $K^1$-splitting and
$w\in W^U_{\eta(\frac{\cdot}{p})}$:
$$
\begin{aligned}
\widetilde\rho((kh(u),1))w
&=s^1(kh(u))\rho^\alpha(kh(u))w\\
&=s^1(k)\left(\frac{u}{p}\right)c(k,h(u))
\rho^\alpha(k)\eta(u)\left(\frac{u}{p}\right)w\\
&=\eta(u)c(k,h(u))\widetilde\rho((k,1))w.
\end{aligned}
$$
Here we used $s^1(h(u))=(\frac{u}{p})$ and
$(\frac{u}{p})^2=1$.

We now compare the cocycles.  From the cocycle identity,
$$
c(k,h(u))c(kh(u),h(p^m))=c(k,h(u)h(p^m))c(h(u),h(p^m))
$$
and
$$
c(k,h(p^m))c(kh(p^m),h(u))=c(k,h(p^m)h(u))c(h(p^m),h(u)).
$$
Since $h(u)h(p^m)=h(p^m)h(u)$, the first factors on the right are the same, and a
direct calculation gives
$$
c(h(u),h(p^m))=c(h(p^m),h(u)).
$$
It follows that
$$
c(kh(u),h(p^m))c(kh(p^m),h(u))=c(k,h(u))c(k,h(p^m)).
$$
Substituting this in the preceding expression and using
$c(k,h(u))^2=1$ yields
$$
\widetilde\pi((h(u),1))f_w=\eta(u)f_w.
$$

\smallskip
\noindent\emph{(3) Lower unipotents.}
Let $t\in p^{2\ell}\mathbb Z_p$.  Since
$$
h(p^m)y(t)=y(p^{-2m}t)h(p^m),
$$
we obtain
$$
\begin{aligned}
{}[\widetilde\pi((y(t),1))f_w]((g,\epsilon))
&=f_w\bigl((k y(p^{-2m}t)h(p^m),\epsilon c(g,y(t)))\bigr)\\
&=c(k y(p^{-2m}t),h(p^m))\,\epsilon c(g,y(t))
\widetilde\rho((k y(p^{-2m}t),1))w.
\end{aligned}
$$
If $\ell$ is even, then $2m=\ell$ and
$$
y(p^{-2m}t)\in\overline N_\ell,
$$
so it fixes $w$.  If $\ell$ is odd, then $2m=\ell-1$ and
$$
y(p^{-2m}t)\in\overline N_{\ell+1}.
$$
Moreover
$$
\alpha y(p^{-2m}t)\alpha^{-1}=y(p^{-2m-1}t)\in\overline N_\ell,
$$
so it fixes $w$ in the conjugated representation, and
$s^1(y(p^{-2m}t))=1$.  Therefore in both cases
$$
\widetilde\rho((k y(p^{-2m}t),1))w
=
c(k,y(p^{-2m}t))\widetilde\rho((k,1))w.
$$
Exactly as above, the cocycle identity reduces the desired equality to
$$
c(y(p^{-2m}t),h(p^m))c(h(p^m),y(t))=1.
$$
We check this explicitly. If $t=0$, both cocycle values are $1$.
For $a\in\mathbb Q_p^\times$ and $x\in\mathbb Q_p^\times$,
we have
$$
c\!\left(\begin{pmatrix}a&0\\0&a^{-1}\end{pmatrix},y(x)\right)
=
c\!\left(y(x),\begin{pmatrix}a&0\\0&a^{-1}\end{pmatrix}\right)
=
\begin{cases}
(a,a^{-1})_p,&2\nmid\operatorname{ord}_p(ax),\\
(a,x)_p,&2\mid\operatorname{ord}_p(ax).
\end{cases}
$$
For $t\ne0$, apply this with $a=p^m$ and with $x=p^{-2m}t$ and
$x=t$, respectively. The two values of $ax$ are $p^{-m}t$ and $p^mt$,
whose valuations differ by $2m$ and hence have the same parity. If
$2\nmid\operatorname{ord}_p(p^{-m}t)$, both cocycles equal
$(p^m,p^{-m})_p$, so their product is $1$. If
$2\mid\operatorname{ord}_p(p^{-m}t)$, their product is
$$
(p^m,p^{-2m}t)_p(p^m,t)_p
=(p^m,t)_p^2=1,
$$
since $p^{-2m}$ is a square.
Thus
$$
c(y(p^{-2m}t),h(p^m))c(h(p^m),y(t))=1,
$$
and consequently
$$
\widetilde\pi((y(t),1))f_w=f_w.
$$
This proves (i) and (ii).

\medskip
\noindent\textbf{II. The ramified cases.}
We first verify the invariance properties of $\phi_w$ used in the
calculation.  If $n=x(a)\in N_\ell$, then, for $k\in I$,
$$
{}[\rho(n)\phi_w](k)=\phi_w(kn)=\sigma(k)\sigma(n)w=\sigma(k)w=\phi_w(k),
$$
and if $k\notin I$, then also $kn\notin I$, so both sides are zero.  Hence
$\phi_w$ is $N_\ell$-fixed.  The same argument with lower unipotents shows
that $\phi_w$ is $\overline N_{\ell+1}$-fixed.  The diagonal action is
$$
\rho(h(u))\phi_w=\eta_\ell(u)\phi_w,
$$
where in the ramified normalization
$$
\eta_\ell(u)=\eta(u)\left(\frac{u}{p}\right).
$$

Again let $g=kh(p^m)$.

\smallskip
\noindent\emph{(1) Upper unipotents.}
For $a\in\mathbb Z_p$ the matrix identity is the same:
$$
h(p^m)x(a)=x(p^{2m}a)h(p^m).
$$
Thus
$$
\begin{aligned}
{}[\widetilde\pi((x(a),1))f_{\phi_w}]((g,\epsilon))
&=c(kx(p^{2m}a),h(p^m))\,\epsilon c(g,x(a))\\
&\qquad{}\cdot\widetilde\rho((kx(p^{2m}a),1))\phi_w.
\end{aligned}
$$
If $\ell$ is even, then $2m=\ell$ and
$x(p^{2m}a)\in N_\ell$, so $\rho(x(p^{2m}a))\phi_w=\phi_w$.  Since
$s^1(x(p^{2m}a))=1$,
$$
\begin{aligned}
\widetilde\rho((kx(p^{2m}a),1))\phi_w
&=s^1(kx(p^{2m}a))\rho(kx(p^{2m}a))\phi_w\\
&=s^1(k)c(k,x(p^{2m}a))\rho(k)\phi_w\\
&=c(k,x(p^{2m}a))\widetilde\rho((k,1))\phi_w.
\end{aligned}
$$
If $\ell$ is odd, then $2m=\ell-1$ and we use the conjugated
representation.  Since
$$
\alpha x(p^{2m}a)\alpha^{-1}=x(p^\ell a)\in N_\ell
$$
and $\phi_w$ is $N_\ell$-fixed, the same calculation gives
$$
\widetilde\rho((kx(p^{2m}a),1))\phi_w
=c(k,x(p^{2m}a))\widetilde\rho((k,1))\phi_w.
$$
The remaining cocycle computation is exactly the one already written in
Part I(1), and therefore
$$
\widetilde\pi((x(a),1))f_{\phi_w}=f_{\phi_w}.
$$

\smallskip
\noindent\emph{(2) Diagonal units.}
For $u\in\mathbb Z_p^\times$,
$$
\begin{aligned}
\widetilde\rho((kh(u),1))\phi_w
&=s^1(kh(u))\rho(kh(u))\phi_w\\
&=s^1(k)s^1(h(u))c(k,h(u))\rho(k)\eta_\ell(u)\phi_w\\
&=\eta(u)c(k,h(u))\widetilde\rho((k,1))\phi_w,
\end{aligned}
$$
since $s^1(h(u))=(\frac{u}{p})$ and
$\eta_\ell(u)=\eta(u)(\frac{u}{p})$.  The cocycle calculation is the same
as Part I(2), so the required diagonal transformation by $\eta(u)$ follows.

\smallskip
\noindent\emph{(3) Lower unipotents.}
Now let $t\in p^{2\ell+1}\mathbb Z_p$.  Again
$$
h(p^m)y(t)=y(p^{-2m}t)h(p^m).
$$
If $\ell$ is even, then $2m=\ell$ and
$$
y(p^{-2m}t)\in\overline N_{\ell+1},
$$
which fixes $\phi_w$.  If $\ell$ is odd, then $2m=\ell-1$ and
$$
y(p^{-2m}t)\in\overline N_{\ell+2},
$$
while
$$
\alpha y(p^{-2m}t)\alpha^{-1}
=y(p^{-2m-1}t)\in\overline N_{\ell+1},
$$
which also fixes $\phi_w$.  Since the splitting is $1$ on these lower
unipotents, in both cases
$$
\widetilde\rho((ky(p^{-2m}t),1))\phi_w
=c(k,y(p^{-2m}t))\widetilde\rho((k,1))\phi_w.
$$
The rest of the calculation is identical to Part I(3): the same two uses of
the cocycle identity reduce the equality to
$$
c(y(p^{-2m}t),h(p^m))c(h(p^m),y(t))=1,
$$
and the same Hilbert-symbol calculation proves it.  Hence
$$
\widetilde\pi((y(t),1))f_{\phi_w}=f_{\phi_w}.
$$
This proves (iii) and (iv), and hence the proposition.
\end{proof}

We now transfer the action of the highest nontrivial Hecke operators at the
chosen newvector level to the preceding linear lemma.  Thus, in the
unramified case we work in the Hecke algebra at level $K_{2\ell}$ and use
$\mathcal V_{2\ell-1,\delta}$; in the ramified case we work at level
$K_{2\ell+1}$ and use $\mathcal V_{2\ell,\delta}$.

\begin{prop}[Metaplectic-to-linear Hecke reduction]\label{lem:metaplectic-linear-bridge}
Let $\delta\in\{1,\lambda\}$.  We use the notation
$\mathcal T_{r,\delta}$ from Lemma~\ref{lem:linear-supercuspidal-hecke},
with the linear $U$-character specified in each case below.

\begin{enumerate}[label=(\roman*)]
\item In the unramified cases of
Proposition~\ref{prop:supercuspidal-newvectors-four-cases},
$$
\mathcal V_{2\ell-1,\delta}f_w
=f_{\mathcal T_{\ell,\delta}w}.
$$
Consequently
$$
\mathcal{Y}_{2\ell-1}f_w=\bigl(I+\mathcal V_{2\ell-1,1}+\mathcal V_{2\ell-1,\lambda}\bigr)f_w=0.
$$

\item In both ramified cases,
$$
\mathcal V_{2\ell,\delta}f_{\phi_w}
=f_{\mathcal T_{\ell+1,\delta}\phi_w}.
$$
Consequently
$$
\mathcal{Y}_{2\ell}f_{\phi_w}=\bigl(I+\mathcal V_{2\ell,1}+\mathcal V_{2\ell,\lambda}\bigr)f_{\phi_w}=0.
$$
\end{enumerate}
\end{prop}

\begin{proof}
We compute the Hecke action at the distinguished point $g=h(p^m)$ and then use
$\widetilde{\mathcal K}$-equivariance.  The right-coset decomposition used in the Hecke algebra gives
$$
{}[\mathcal V_{r,\delta}f]((g,1))
=
\sum_{u\in\mathbb Z_p^\times/\sqrt{1+p\mathbb Z_p}}\eta(u)
 f\bigl((g,1)(h(u),1)(y(\delta p^r),1)\bigr),
$$
because the left biequivariance factor is
$\overline{\eta((h(u),1))}=\overline{\eta(u^{-1})}=\eta(u)$.
Here we use that $\eta$ is unitary (and, in the present paper, quadratic).
Here $r=2\ell-1$ in the unramified case and $r=2\ell$ in the ramified
case.

\medskip
\noindent\textbf{I. Unramified case, $\ell$ even.}
Here $m=\ell/2$ and $r=2\ell-1$.  Since
$$
h(p^m)h(u)y(\delta p^{2\ell-1})
=h(u)y(\delta p^{\ell-1})h(p^m),
$$ with $h(u)y(\delta p^{\ell-1})\in K$,
the only point to check is that the central sign obtained by multiplying the
three lifts is exactly cancelled by the cocycle $c(h(u)y(\delta
p^{\ell-1}),h(p^m))$ appearing in the definition of $f_w$.  Equivalently we
must prove
$$
c\bigl(h(p^m),h(u)y(\delta p^{2\ell-1})\bigr)
 c\bigl(h(u)y(\delta p^{\ell-1}),h(p^m)\bigr)=1.
$$
Put $t=\delta p^{2\ell-1}$ and
$t'=p^{-2m}t=\delta p^{\ell-1}$.  Since $h(p^m)h(u)=h(u)h(p^m)$ and
$h(p^m)y(t)=y(t')h(p^m)$, the cocycle identity gives
$$
\begin{aligned}
c(h(p^m),h(u)y(t))
&=\frac{c(h(p^m),h(u))c(h(p^m)h(u),y(t))}{c(h(u),y(t))}\\
&=\frac{c(h(u),h(p^m)y(t))c(h(p^m),y(t))}{c(h(u),y(t))},\\
c(h(u)y(t'),h(p^m))
&=\frac{c(h(u),y(t')h(p^m))c(y(t'),h(p^m))}{c(h(u),y(t'))}\\
&=\frac{c(h(u),h(p^m)y(t))c(y(t'),h(p^m))}{c(h(u),y(t'))}.
\end{aligned}
$$
Here we used $c(h(p^m),h(u))=c(h(u),h(p^m))$.  The diagonal--lower-unipotent formula is
$$
c\!\left(\begin{pmatrix}a&0\\0&a^{-1}\end{pmatrix},y(x)\right)
=
c\!\left(y(x),\begin{pmatrix}a&0\\0&a^{-1}\end{pmatrix}\right)
=
\begin{cases}
(a,a^{-1})_p,&2\nmid\operatorname{ord}_p(ax),\\
(a,x)_p,&2\mid\operatorname{ord}_p(ax).
\end{cases}
$$
Since $t/t'=p^{2m}$ is a square, it gives
$c(h(u),y(t))=c(h(u),y(t'))$ and
$c(h(p^m),y(t))=c(y(t'),h(p^m))$.  Therefore
$$
c(h(p^m),h(u)y(t))c(h(u)y(t'),h(p^m))=1.
$$
Thus the metaplectic factor is $1$, and therefore
$$
\begin{aligned}
{}[\mathcal V_{2\ell-1,\delta}f_w]((h(p^m),1))
&=
\sum_{u\in\mathbb Z_p^\times/\sqrt{1+p\mathbb Z_p}}\eta(u)
\widetilde\rho((h(u)y(\delta p^{\ell-1}),1))w\\
&=
\sum_{u\in\mathbb Z_p^\times/\sqrt{1+p\mathbb Z_p}}\eta(u)
\rho(h(u)y(\delta p^{\ell-1}))w\\
&=\mathcal T_{\ell,\delta}w.
\end{aligned}
$$

\medskip
\noindent\textbf{II. Unramified case, $\ell$ odd.}
Now $m=(\ell-1)/2$ and again $r=2\ell-1$.  We have
$$
h(p^m)h(u)y(\delta p^{2\ell-1})
=h(u)y(\delta p^\ell)h(p^m).
$$
The same cocycle calculation as above, with
$t=\delta p^{2\ell-1}$ and $t'=\delta p^\ell$, shows that the total
central factor is $1$.  Since
$$
\alpha h(u)y(\delta p^\ell)\alpha^{-1}
=h(u)y(\delta p^{\ell-1}),
$$
we get
$$
\begin{aligned}
{}[\mathcal V_{2\ell-1,\delta}f_w]((h(p^m),1))
&=
\sum_{u\in\mathbb Z_p^\times/\sqrt{1+p\mathbb Z_p}}
\eta(u)s^1(h(u)y(\delta p^\ell))\\
&\qquad{}\cdot \rho^\alpha(h(u)y(\delta p^\ell))w\\
&=
\sum_{u\in\mathbb Z_p^\times/\sqrt{1+p\mathbb Z_p}}
\eta(u)\left(\frac{u}{p}\right)\\
&\qquad{}\cdot \rho(h(u)y(\delta p^{\ell-1}))w\\
&=\mathcal T_{\ell,\delta}w.
\end{aligned}
$$
Here the second equality uses exactly the splitting relation
$$
s^1(h(u)y(t))=s^1(h(u))s^1(y(t))c(h(u),y(t)),
$$
with $s^1(y(t))=1$ and $s^1(h(u))=(\frac{u}{p})$, together with
$w\in W^U_{\eta_\ell}$.

\medskip
\noindent\textbf{III. Ramified case, $\ell$ even.}
Here $m=\ell/2$ and the highest operator in the level-$2\ell+1$ Hecke
algebra is $\mathcal V_{2\ell,\delta}$.  We have
$$
h(p^m)h(u)y(\delta p^{2\ell})
=h(u)y(\delta p^\ell)h(p^m).
$$
The metaplectic sign is
$$
c\bigl(h(p^m),h(u)y(\delta p^{2\ell})\bigr)
 c\bigl(h(u)y(\delta p^\ell),h(p^m)\bigr),
$$
and the same cocycle calculation as in Case I, now with
$t=\delta p^{2\ell}$ and $t'=\delta p^\ell$, shows that this product is $1$.
Therefore
$$
\begin{aligned}
{}[\mathcal V_{2\ell,\delta}f_{\phi_w}]((h(p^m),1))
&=
\sum_{u\in\mathbb Z_p^\times/\sqrt{1+p\mathbb Z_p}}\eta(u)
\widetilde\rho((h(u)y(\delta p^\ell),1))\phi_w\\
&=
\sum_{u\in\mathbb Z_p^\times/\sqrt{1+p\mathbb Z_p}}\eta(u)\left(\frac{u}{p}\right)
\rho(h(u)y(\delta p^\ell))\phi_w\\
&=
\sum_{u\in\mathbb Z_p^\times/\sqrt{1+p\mathbb Z_p}}\eta_\ell(u)^{-1}
\rho(h(u)y(\delta p^\ell))\phi_w\\
&=\mathcal T_{\ell+1,\delta}\phi_w.
\end{aligned}
$$
This is precisely the linear calculation of
Lemma~\ref{lem:linear-supercuspidal-hecke} with $\ell$ replaced by
$\ell+1$.

\medskip
\noindent\textbf{IV. Ramified case, $\ell$ odd.}
Now $m=(\ell-1)/2$ and again $r=2\ell$.  Thus
$$
h(p^m)h(u)y(\delta p^{2\ell})
=h(u)y(\delta p^{\ell+1})h(p^m).
$$
The $K^1$-type is the $\alpha$-conjugate one, and
$$
\alpha h(u)y(\delta p^{\ell+1})\alpha^{-1}
=h(u)y(\delta p^\ell).
$$
The same cocycle calculation applies with
$t=\delta p^{2\ell}$ and $t'=\delta p^{\ell+1}$; the extra $s^1$-factor
is cancelled by the splitting relation, exactly as in the odd unramified
calculation.  Hence
$$
\begin{aligned}
{}[\mathcal V_{2\ell,\delta}f_{\phi_w}]((h(p^m),1))
&=
\sum_{u\in\mathbb Z_p^\times/\sqrt{1+p\mathbb Z_p}}\eta(u)\left(\frac{u}{p}\right)
\rho(h(u)y(\delta p^\ell))\phi_w\\
&=
\sum_{u\in\mathbb Z_p^\times/\sqrt{1+p\mathbb Z_p}}\eta_\ell(u)^{-1}
\rho(h(u)y(\delta p^\ell))\phi_w\\
&=\mathcal T_{\ell+1,\delta}\phi_w.
\end{aligned}
$$

For an arbitrary point $kh(p^m)$ in the support, both sides of each asserted
identity are obtained from their value at $h(p^m)$ by the same
$\widetilde{\mathcal K}$-equivariance, and both vanish off the corresponding support.
Thus the equalities hold as functions.

Finally, Lemma~\ref{lem:linear-supercuspidal-hecke} gives in the unramified
case
$$
\begin{aligned}
\bigl(I+\mathcal V_{2\ell-1,1}+\mathcal V_{2\ell-1,\lambda}\bigr)f_w
&=f_{(I+\mathcal T_{\ell,1}+\mathcal T_{\ell,\lambda})w}=0.
\end{aligned}
$$
It remains to justify the vanishing in the odd ramified case.  Here the
entire induced $K^0$-representation need not have zero
$\overline N_\ell$-fixed space, so such vanishing does not hold in general.
Nevertheless, the averaging projection kills the particular vector
$\phi_w$.  Indeed, $\phi_w$ is again $\overline N_{\ell+1}$-fixed, because
$\overline N_{\ell+1}\subset I_\ell$ and $\sigma$ is trivial on $I_\ell$.
Moreover, $\phi_w$ is supported on $I$, and for $k\in I$,
$$
\int_{\overline N_\ell}[\rho(n)\phi_w](k)\,dn
=
\sigma(k)\int_{\overline N_\ell}\sigma(n)w\,dn=0,
$$
where the last equality follows from
$\Hom_{\overline N_\ell}(1,\sigma)=0$, as proved above using
\cite[Remark~3.3.2]{L-R}; outside $I$ the integral is also zero.  Thus the
same averaging calculation gives
$$
\begin{aligned}
\bigl(I+\mathcal V_{2\ell,1}+\mathcal V_{2\ell,\lambda}\bigr)f_{\phi_w}
&=f_{(I+\mathcal T_{\ell+1,1}+\mathcal T_{\ell+1,\lambda})\phi_w}=0.
\end{aligned}
$$
This proves the remaining odd ramified case.
\end{proof}

The packet descriptions used in the next proposition are those of
\cite[\S5.1]{Ishimoto}; the relevant minimal fixed-space dimensions are
given in \cite[Propositions 3.3.4, 3.3.6, 3.3.8]{L-R}.

Let
$$
a_\lambda=\begin{pmatrix}\lambda&0\\0&1\end{pmatrix}
\in\GL_2(\mathbb Z_p).
$$
Conjugation by $a_\lambda$ normalizes the relevant determinant-one compact
subgroups, fixes every $h(u)$, and sends
$$
y(t)\longmapsto y(\lambda^{-1}t).
$$
Hence it preserves each $U$-weight space $W^U_{\eta_\ell}$ and interchanges
the two square classes in the linear operators of
Lemma~\ref{lem:linear-supercuspidal-hecke}.

Consider the operator \(A_\lambda=\tau(a_\lambda)\), where \(\tau\) is the very cuspidal representation introduced above.
Then on
the relevant $U$-weight spaces
$$
A_\lambda\mathcal T_{\ell,1}A_\lambda^{-1}
=\mathcal T_{\ell,\lambda},
\qquad
A_\lambda\mathcal T_{\ell,\lambda}A_\lambda^{-1}
=\mathcal T_{\ell,1}.
$$
In the cases where the restriction of the very cuspidal type splits into
two constituents, the same nonsquare-determinant conjugation exchanges
those two constituents; see \cite[\S3.3]{L-R}.
In the ramified cases, put $\mathcal K=K^0$ or $K^1$, as appropriate.
If $A_\lambda W_i=W_j$, the corresponding map between the
induced representations is defined by
$$
\widehat A_\lambda:\operatorname{Ind}_I^{\mathcal K}\sigma_i
\longrightarrow\operatorname{Ind}_I^{\mathcal K}\sigma_j,
\qquad
(\widehat A_\lambda\Phi)(k)
=A_\lambda\Phi(a_\lambda^{-1}ka_\lambda).
$$
This is well defined because $a_\lambda$ normalizes $I$ and $\mathcal K$,
and
$A_\lambda\sigma_i(a_\lambda^{-1}ia_\lambda)
=\sigma_j(i)A_\lambda$ for $i\in I$.
The definition of $\phi_w$ therefore gives
$$
\widehat A_\lambda\phi_w=\phi_{A_\lambda w}.
$$
On the induced spaces, the conjugation identities above are
understood with $\widehat A_\lambda$ in place of $A_\lambda$.
This is also compatible with the subsequent $\alpha$-conjugation
in the odd ramified case, since $a_\lambda$ commutes with $\alpha$.

No lift of $a_\lambda$ to
the metaplectic group is needed below: Proposition~\ref{lem:metaplectic-linear-bridge}
transfers the resulting linear identities to the metaplectic newvectors.

\begin{prop}\label{lem:supercuspidal-terminal-signs}
Let $\eta_\ell$ be as above.
\begin{enumerate}[label=(\roman*)]
\item Suppose that the supercuspidal $L$-packet is unramified of cardinality
two.  Then $\dim W^U_{\eta_\ell}=2$.  Choose common Hecke eigenvectors
$w_1,w_2$ spanning this space.  The operator $\mathcal W_{2\ell-1}$ has
distinct eigenvalues on $f_{w_1}$ and $f_{w_2}$, namely
$$
\sqrt{p^*}\quad\text{and}\quad-\sqrt{p^*}
$$
in some order.

\item Suppose that the supercuspidal $L$-packet is the exceptional unramified
packet of cardinality four.  Then $\ell=1$ and the relevant restriction
splits as $\sigma_1\oplus\sigma_2$.  If $0\ne w_i\in (W_i)^U_{\eta_\ell}$
and $f_{w_i}$ denotes the corresponding newvector for the packet member
attached to $\sigma_i$, then $\mathcal W_1$ acts on these two newvectors by
opposite eigenvalues $\sqrt{p^*}$ and $-\sqrt{p^*}$.

\item Suppose that the supercuspidal $L$-packet is ramified.  Write the
restriction of the very cuspidal type to the Iwahori subgroup of $\SL_2$ as
$\sigma_1\oplus\sigma_2$.  If $0\ne w_i\in (W_i)^U_{\eta_\ell}$ and
$f_{\phi_{w_i}}$ is the corresponding newvector for the packet member
attached to $\sigma_i$, then $\mathcal W_{2\ell}$ acts on
$f_{\phi_{w_1}}$ and $f_{\phi_{w_2}}$ by opposite eigenvalues
$\sqrt{p^*}$ and $-\sqrt{p^*}$.
\end{enumerate}
\end{prop}

\begin{proof}
Put
$$
D_\ell=\mathcal T_{\ell,1}-\mathcal T_{\ell,\lambda}.
$$
In (i), Proposition~\ref{lem:metaplectic-linear-bridge} gives
$$
\mathcal W_{2\ell-1}f_w=f_{D_\ell w}
$$
for every $w\in W^U_{\eta_\ell}$.  Suppose that
$\mathcal W_{2\ell-1}$ had the same eigenvalue $c$ on the two newvectors
$f_{w_1},f_{w_2}$.  Since $w_1,w_2$ form a basis of the two-dimensional
space $W^U_{\eta_\ell}$ and the map $w\mapsto f_w$ is injective,
Proposition~\ref{lem:metaplectic-linear-bridge} implies
$$
D_\ell=cI
$$
on $W^U_{\eta_\ell}$.  On the other hand, the linear conjugation above
gives
$$
A_\lambda D_\ell A_\lambda^{-1}=-D_\ell.
$$
Because $A_\lambda$ preserves $W^U_{\eta_\ell}$, conjugating the scalar
identity $D_\ell=cI$ yields $cI=-cI$, and hence $c=0$.  Thus $D_\ell=0$
on $W^U_{\eta_\ell}$.  Applying
Proposition~\ref{lem:metaplectic-linear-bridge} once more gives
$$
\mathcal W_{2\ell-1}=0
$$
on the metaplectic newspace, contradicting the terminal relation
$$
\mathcal W_{2\ell-1}^2=p^*I.
$$
Therefore the two eigenvalues are distinct, and the same terminal relation
shows that they are $\sqrt{p^*}$ and $-\sqrt{p^*}$.

For (ii), the two one-dimensional $U$-weight spaces lie in the two
constituents exchanged by $a_\lambda$.  The linear conjugation interchanges
$\mathcal T_{1,1}$ and $\mathcal T_{1,\lambda}$, and therefore sends
$$
D_1=\mathcal T_{1,1}-\mathcal T_{1,\lambda}
$$
to $-D_1$.  Hence the $D_1$-eigenvalues on the two $U$-weight lines are
opposite.  Proposition~\ref{lem:metaplectic-linear-bridge} transfers these
to opposite $\mathcal W_1$-eigenvalues on the corresponding metaplectic
newvectors.  Since $\mathcal W_1^2=p^*I$, those eigenvalues are
$\sqrt{p^*}$ and $-\sqrt{p^*}$.

The ramified case (iii) is identical, with
$$
D_{\ell+1}
=\mathcal T_{\ell+1,1}-\mathcal T_{\ell+1,\lambda}.
$$
The nonsquare-determinant conjugation exchanges the two linear constituents
and sends $D_{\ell+1}$ to $-D_{\ell+1}$.  By
Proposition~\ref{lem:metaplectic-linear-bridge}, their eigenvalues transfer
to opposite eigenvalues of $\mathcal W_{2\ell}$ on
$f_{\phi_{w_1}}$ and $f_{\phi_{w_2}}$.  Finally
$$
\mathcal W_{2\ell}^2=p^*I
$$
gives the two values $\sqrt{p^*}$ and $-\sqrt{p^*}$.
\end{proof}

\section{Relation with Ueda's Twisting Operator}\label{sec:Ueda-Ishimoto-Rp}

We first recall Ishimoto's local operator and its relation with Ueda's
classical twisting operator. We then illustrate the distinction between spaces of fixed
vectors at the minimal level and Kohnen--Ueda new quotients
through the $U_1$-eigenvalues in
Proposition~\ref{prop:eigenvalues_steinberg}, Case~4.
Finally, at levels $m\geq2$, we compare Ishimoto's operator with
$\mathcal W_{m-1}$ by a lifted $\GL_2$-conjugation. These comparisons are
not used in the preceding sections.

\subsection{Ishimoto's local operator and Ueda's twisting operator}
Let $\eta_0$ be either the trivial or the nontrivial quadratic character
of $\mathbb Z_p^\times$, and write $\eta$ for its genuine extensions
as in Section~\ref{sec:hecke-subalgebra}. %
For an irreducible genuine smooth representation $(\pi_{\sigma_0},V)$
of $\widetilde G_{\sigma_0}$, put
$\pi=\pi_{\sigma_0}\circ\iota$, where $\iota$ is the isomorphism of
cocycle models defined in Section~2. %
Ishimoto's operator defined below in his cocycle model has the same expression in our cocycle model.

With $\operatorname{vol}(\mathbb Z_p)=1$, Ishimoto's operator
$R_\varpi$, specialized to $F=\mathbb Q_p$ and $\varpi=p$, is
$$
R_p^\pi v=
\int_{\mathbb Z_p^\times}\left(\frac{u}{p}\right)
\pi((x(p^{-1}u),1))v\,du.
$$
It acts on $\bigcup_{r\geq c(\eta_0)}V_\eta^{K_r}(\pi)$, but need not
preserve each level. More precisely,
\cite[Proposition~3.7(b)]{IshimotoKU} gives
$$
R_p^\pi\bigl(V_\eta^{K_r}(\pi)\bigr)
\subseteq V_\eta^{K_{r'}}(\pi),
\qquad r'=\max\{2,c(\eta_0)+1,r\}.
$$
For $v\in V_\eta^{K_r}(\pi)$,
it is easy to see that
\begin{equation}\label{eq:R_pIshimoto}
R_p^\pi v=\frac1p\sum_{a\in\mathbb F_p^*}
\left(\frac{a}{p}\right)\pi((x(a/p),1))v.
\end{equation}

For $f(z)=\sum_{n\geq1}a(n)e^{2\pi inz}$, Ueda's quadratic
twisting operator is
\[
f|R_p(z)=\sum_{n\geq1}\left(\frac np\right)a(n)e^{2\pi inz};
\]
see \cite[Section~0(c)]{Ueda1993} and \cite{Ueda1998}.
Let $f\in S_{k+1/2}(4N,\chi)$, where $N$ is odd and $\chi$
is an even quadratic or trivial character. Write $\Phi_f$
for its adelization and $\rho_p$ for the local action at $p$.
With our Haar normalization,
\[
R_p^{\rho_p}\Phi_f=p^{-1}\tau_p^-\,\Phi_{f|R_p},
\qquad
\tau_p^-=\sum_{a\in\mathbb F_p^*}
\left(\frac ap\right)e^{-2\pi ia/p}.
\]
This identity is proved in \cite[Proposition~3.12]{IshimotoKU}
for the plus space. The same argument applies to the full
space: it uses only adelization and unipotent translations,
and does not use the plus-space condition.

\subsection{Minimal vectors and Kohnen--Ueda new quotients}

We recall the minimal-level case of Ishimoto's definition
of Kohnen--Ueda newforms \cite[Definition~3.1]{IshimotoKU}.
Let $c=c_\eta(\pi)<\infty$. Since the lower-level terms
vanish, the Kohnen--Ueda new space at level $c$, denoted by
$V_\eta^{K_c,\mathrm{KUnew}}(\pi)$, is the quotient
\[
V_\eta^{K_c,\mathrm{KUnew}}(\pi)
=
V_\eta^{K_c}(\pi)\big/
\ker\!\left(R_p^\pi\big|_{V_\eta^{K_c}(\pi)}\right).
\]
Thus the space $V_\eta^{K_c}(\pi)$ of fixed vectors at the
minimal level coincides with its KU-new quotient
$V_\eta^{K_c,\mathrm{KUnew}}(\pi)$ whenever this restriction
of $R_p^\pi$ is injective. This coincidence holds in many of the principal-series and
supercuspidal cases; see \cite[Theorem~3.11(2)]{IshimotoKU}.
However, it fails for certain even Weil representations,
whose space of fixed vectors at the minimal level is nonzero
but whose KU-new quotient is zero; see
\cite[Proposition~4.8]{IshimotoKU}.

In Proposition~\ref{prop:eigenvalues_steinberg}, Case~4,
our operator $U_1$ distinguishes such an even Weil constituent
from the corresponding Steinberg constituent.
With $a\in\{p,p\beta\}$ and $\eta_0=(\frac{\cdot}{p})$,
the representations $\omega_{\psi_a}^{+}$ and
$St_{\psi,\mu_a}$ both have $\eta$-conductor $1$, with one-dimensional spaces of $(\ov{K_1},\eta)$-fixed vectors.

Ishimoto shows that $R_p$ kills the even Weil line but is
nonzero on the Steinberg line; see
\cite[Proposition~4.8 and the proof of Proposition~4.9,
case~(2)]{IshimotoKU}.
Consequently, the even Weil KU-new quotient is zero,
whereas the Steinberg minimal line survives unchanged.
Our computation gives the $U_1$-eigenvalues $p\epsilon_p$
on the even Weil quotient and $-\epsilon_p$ on the
Steinberg subrepresentation. Thus the
$-\epsilon_p$-eigenspace of $U_1$ is the
Steinberg newvector line, which also coincides with the
Kohnen--Ueda new space of the Steinberg representation.

\subsection{The lifted $\GL_2$-conjugation}
For the remainder of this section, assume $m\geq2$ and retain the
quadratic character $\eta_0$ and its genuine extension $\eta$.
Since $c(\eta_0)\leq1$, the level statement above shows that $R_p^\pi$
preserves $V_\eta^{K_m}(\pi)$. We now compare its restriction to this
space with the action of $\mathcal W_{m-1}$, whose Hecke function is
supported in $\ov K$. The comparison uses a lifted $\GL_2$-conjugation
and the conjugated character $\eta^{J_m}$ defined below.

Put
$$
J_m=\begin{pmatrix}0&1\\-p^m&0\end{pmatrix}\in\GL_2(\Q_p).
$$
A direct calculation gives
$$
J_mK_mJ_m^{-1}=K_m,
\qquad
J_mx(t)J_m^{-1}=y(-p^mt),
$$
and
$$
J_my(t)J_m^{-1}=x(-p^{-m}t),
\qquad
J_mh(u)J_m^{-1}=h(u^{-1}).
$$
We now explain the lift of this conjugation to the double cover used in this paper.

Using the notation $\tau(A)$ from Section~2 for $A\in\GL_2(\Q_p)$, and following Kubota and Kazhdan--Patterson \cite{Kubota,KP}, consider the metaplectic double cover
$$
\widehat{\GL}_2(\Q_p)=\GL_2(\Q_p)\times\mu_2
$$
defined by the $2$-cocycle
$$
\sigma(A,B)
=
\left(
\frac{\tau(AB)}{\tau(A)},
\frac{\tau(AB)}{\tau(B)}\det A
\right)_p.
$$
Its restriction to $G=\SL_2(\Q_p)$ is the cocycle $\sigma_0$ of Section~2.
The inverse image of $G$ in this cover is therefore
$\widetilde G_{\sigma_0}$, which we identify with our $c$-model by
the isomorphism $\iota$ defined there.

Let $\widetilde J_m=(J_m,1)\in\widehat{\GL}_2(\Q_p)$ and put
$$
\alpha_m(g)=J_mgJ_m^{-1}.
$$
Since $\SL_2(\Q_p)$ is normal in $\GL_2(\Q_p)$, conjugation by $\widetilde J_m$ preserves the inverse image of $\SL_2(\Q_p)$.  We therefore obtain an automorphism of our cover by
$$
\Theta_m
=
\iota^{-1}\circ\operatorname{Ad}_{\widetilde J_m}\circ\iota:
\widetilde G_c\longrightarrow\widetilde G_c.
$$

We record explicitly the action of $\Theta_m$, for use in
the proof of the following lemma.
In the $\sigma$-model,
$$
\operatorname{Ad}_{\widetilde J_m}(g,\epsilon)
=
(\alpha_m(g),\epsilon r_m(g)),
$$
where
$$
r_m(g)
=\sigma(J_m,g)\sigma(J_mg,J_m^{-1})\sigma(J_m,J_m^{-1}).
$$
Since $\tau(J_m)=-p^m$ and $\tau(J_m^{-1})=1$,
one has
$$
\sigma(J_m,J_m^{-1})=(-p^{-m},p^m)_p=1.
$$
Applying $\iota^{-1}$ yields
$$
\Theta_m((g,\epsilon))
=
(\alpha_m(g),\epsilon\kappa_m(g)),
$$
with
$$
\kappa_m(g)
= s_p(g)s_p(\alpha_m(g))r_m(g)
=
s_p(g)s_p(\alpha_m(g))
\sigma(J_m,g)\sigma(J_mg,J_m^{-1}).
$$

\begin{lem}\label{lem:section5-Jm-lift}
For $t\in\Q_p$, $u\in\Q_p^\times$, and $\epsilon\in\{\pm1\}$,
$$
\Theta_m((x(t),\epsilon))=(y(-p^mt),\epsilon),
$$
$$
\Theta_m((y(t),\epsilon))=(x(-p^{-m}t),\epsilon),
$$
and
$$
\Theta_m((h(u),\epsilon))
=
\bigl(h(u^{-1}),\epsilon(u,p^m)_p\bigr).
$$
In particular, for $u\in\Z_p^\times$,
$$
(u,p^m)_p
=(u,p)_p^m
=\left(\frac{u}{p}\right)^m.
$$
\end{lem}

\begin{proof}
For the upper unipotent subgroup, direct substitution in the $\GL_2$ cocycle gives
$$
\sigma(J_m,x(t))=1.
$$
For $t\ne0$,
$$
\sigma(J_mx(t),J_m^{-1})
=(t,-p^{2m}t)_p
=(t,-t)_p
=1,
$$
and for $t=0$ the same value is $\sigma(J_m,J_m^{-1})=1$.  Moreover
$$
s_p(x(t))=s_p(y(-p^mt))=1.
$$
Hence $\kappa_m(x(t))=1$.  The calculation for the lower unipotent subgroup is analogous and gives
$$
\sigma(J_m,y(t))=\sigma(J_my(t),J_m^{-1})=1,
$$
with trivial $s_p$-factors.

For the torus,
$$
\sigma(J_m,h(u))=(u,-1)_p,
$$
whereas
$$
\sigma(J_mh(u),J_m^{-1})=(-p^{-m},p^mu)_p.
$$
Using $(-p^{-m},p^m)_p=1$, and bilinearity and symmetry, %
$$
\begin{aligned}
(u,-1)_p(-p^{-m},p^mu)_p
&=(u,-1)_p(-p^{-m},u)_p\\
&=(u,-1)_p(-1,u)_p(p^m,u)_p=(u,p^m)_p.
\end{aligned}
$$
Since $s_p(h(u))=s_p(h(u^{-1}))=1$, the torus formula follows.  For a unit $u$ and odd $p$ one has $(u,p)_p=(\frac{u}{p})$, proving the last assertion.
\end{proof}

Recall
$$
\eta((h(u),\epsilon))=\epsilon\eta_0(u^{-1}),
\qquad u\in\Z_p^\times,
$$
where $\eta_0$ is the underlying quadratic character associated with $\eta$, as fixed at the beginning of this section. Define the conjugated character
$$
\eta^{J_m}(k)=\eta(\Theta_m^{-1}(k)),
\qquad k\in\ov{K_m}.
$$
By Lemma~\ref{lem:section5-Jm-lift}, a direct calculation gives
$$
\begin{aligned}
\eta^{J_m}((h(u),\epsilon))
&=\eta\bigl(\Theta_m^{-1}((h(u),\epsilon))\bigr)\\
&=\eta((h(u^{-1}),\epsilon(u,p^m)_p))\\
&=\epsilon (u,p^m)_p\eta_0(u).
\end{aligned}
$$
We therefore define $\eta_0^{J_m}$ by
$$
\eta_0^{J_m}(u^{-1})
=(u,p^m)_p\eta_0(u),
$$
so that
$$
\eta^{J_m}((h(u),\epsilon))
=\epsilon\eta_0^{J_m}(u^{-1}).
$$
Equivalently, replacing $u$ by $u^{-1}$ and using the fact that $\eta_0$ is quadratic,
$$
\eta_0^{J_m}(u)
=\eta_0(u)(u,p^m)_p
=\eta_0(u)\left(\frac{u}{p}\right)^m.
$$

Thus
$$
\eta_0^{J_m}
=
\begin{cases}
\eta_0,&m\text{ even},\\
\eta_0(\frac{\cdot}{p}),&m\text{ odd};
\end{cases}
$$
the parity of $m$ enters the formula for the conjugated character.

Let
$$
\mathcal H_m^{\mathrm{full}}(\eta)
=
H(\widetilde G//\ov{K_m},\eta)
$$
denote the full genuine Hecke algebra.  Its subalgebra supported in $\ov K$ is $H_m(\eta)$, in the notation of Section~\ref{sec:hecke-subalgebra}.  Since $J_mK_mJ_m^{-1}=K_m$, the automorphism $\Theta_m$ preserves $\ov{K_m}$ and induces
$$
\Theta_{m,*}:\mathcal H_m^{\mathrm{full}}(\eta)
\xrightarrow{\sim}
\mathcal H_m^{\mathrm{full}}(\eta^{J_m}),
\qquad
(\Theta_{m,*}f)(g)=f(\Theta_m^{-1}(g)).
$$

For $\delta\in\{1,\lambda\}$, Lemma~\ref{lem:section5-Jm-lift} gives
$$
\Theta_m((x(\delta p^{-1}),1))
=(y(-\delta p^{m-1}),1).
$$
We next verify the support condition needed to define the corresponding Hecke function.

\begin{lem}[Support condition]\label{lem:section5-affine-admissibility}
For $m\geq2$ and $\delta\in\Z_p^\times$, the double coset
$$
\ov{K_m}(x(\delta p^{-1}),1)\ov{K_m},
$$
supports a unique element of $\mathcal H_m^{\mathrm{full}}(\eta)$ whose value at $(x(\delta p^{-1}),1)$ is $1$.
\end{lem}

\begin{proof}
Put $g_\delta=x(\delta p^{-1})$.  If
$$
 k=\begin{pmatrix}a&b\\c&d\end{pmatrix}
 \in K_m\cap g_\delta K_mg_\delta^{-1}
$$
and $k'=g_\delta^{-1}kg_\delta$, then
$$
k'=
\begin{pmatrix}
a-\delta p^{-1}c&
b+\delta p^{-1}(a-d)-\delta^2p^{-2}c\\
c&d+\delta p^{-1}c
\end{pmatrix}.
$$
In particular,
$$
d(k')d(k)^{-1}
=1+\delta p^{-1}cd^{-1}\in1+p\Z_p.
$$
Since $c(\eta_0)\leq1$, this gives $\eta_0(d(k'))=\eta_0(d(k))$.

It remains to check that no central sign is introduced by the chosen lifts.  In the $\sigma_0$-model, direct substitution gives
$$
\sigma_0(k,g_\delta)=\sigma_0(g_\delta,k')=1.
$$
For $c\ne0$, the compact splitting cochain satisfies
$$
\frac{s_p(k')}{s_p(k)}
=\left(c,\frac{d(k')}{d(k)}\right)_p
=\left(c,1+\delta p^{-1}cd^{-1}\right)_p
=1,
$$
because $1+p\Z_p\subset\Q_p^{\times2}$ for odd $p$; if $c=0$, both splitting factors are $1$.  Using
$c(g,h)=\sigma_0(g,h)s_p(g)s_p(h)s_p(gh)$ and $kg_\delta=g_\delta k'$, we conclude that
$$
(k,1)(g_\delta,1)=(g_\delta,1)(k',1)
$$
in the cocycle realization used in this paper.  Thus the two genuine type characters agree on the intersection.  This is precisely the support condition for the character $\eta$, and it proves existence and uniqueness of the normalized function.
\end{proof}

We denote this normalized function by
$$
\mathcal X_{-1,\delta}^{\eta}
\in\mathcal H_m^{\mathrm{full}}(\eta)
$$
and put
$$
\mathcal R_m^{\eta}
=\mathcal X_{-1,1}^{\eta}-\mathcal X_{-1,\lambda}^{\eta}.
$$
The superscript records the character $\eta$ and will be retained in this section.

Then
$$
\Theta_{m,*}(\mathcal X_{-1,\delta}^{\eta})
=\mathcal V_{m-1,-\delta}^{\eta^{J_m}},
$$
where the superscript indicates the conjugated character $\eta^{J_m}$.  In particular, the conjugated Hecke element belongs to $H_m(\eta^{J_m})$.

\subsection{The Hecke element underlying $R_p$}
We first record the right coset decomposition that will be used in the comparison.

\begin{lem}[Right coset decomposition]\label{lem:section5-affine-cosets}
Let $m\ge2$ and $\delta\in\Z_p^\times$. Then
$$
\begin{aligned}
K_mx(\delta p^{-1})K_m
&=
\bigsqcup_{u\in\Z_p^\times/\sqrt{1+p\Z_p}}
 h(u)x(\delta p^{-1})K_m
\\
&=
\bigsqcup_{u\in\Z_p^\times/\sqrt{1+p\Z_p}}
 x(\delta u^2p^{-1})K_m.
\end{aligned}
$$
Here, as in Lemma~\ref{lem:sqrt},
$\sqrt{1+p\Z_p}=\{u\in\Z_p^\times:u^2\in1+p\Z_p\}=\{\pm1\}(1+p\Z_p)$; hence the quotient has $(p-1)/2$ elements.

\end{lem}

\begin{proof}
Put $g_\delta=x(\delta p^{-1})$ and write
$$
k=\begin{pmatrix}a&b\\c&d\end{pmatrix}\in K_m.
$$
Then
$$
g_\delta^{-1}kg_\delta
=
\begin{pmatrix}
a-\delta p^{-1}c &
 b+\delta p^{-1}(a-d)-\delta^2p^{-2}c\\
c&d+\delta p^{-1}c
\end{pmatrix}.
$$
Because $m\ge2$ and $c\in p^m\Z_p$, this belongs to $K_m$ if and only if
$a\equiv d\pmod p$.
Since $ad-bc=1$ and $c\equiv0\pmod p$, this is equivalent to
$a\equiv d\equiv\pm1\pmod p$.
Hence
$$
[K_m:K_m\cap g_\delta K_mg_\delta^{-1}]
=\frac{p-1}{2}.
$$
The reduction map
$$
K_m\longrightarrow\mathbb F_p^*,
\qquad
\begin{pmatrix}a&b\\c&d\end{pmatrix}\longmapsto a\pmod p,
$$
is surjective, and the intersection above is precisely the inverse image of $\{\pm1\}$.  Thus the diagonal elements $h(u)$, with $u\in\Z_p^\times/\sqrt{1+p\Z_p}$,
give representatives for the right cosets.  Finally,
$$
h(u)x(t)=x(u^2t)h(u),
$$
which gives the second decomposition.  On the corresponding canonical lifts, the cocycle on these products of unit diagonal and upper unipotent elements is trivial, so the same decomposition holds in the metaplectic cover.
\end{proof}

Let $(\pi,V)$ be a smooth genuine representation and let
$v\in V_\eta^{K_m}(\pi)$. On the canonical lifts we have
$$
(x(\delta p^{-1}u^2),1)
=(h(u),1)(x(\delta p^{-1}),1)(h(u)^{-1},1),
$$
so biequivariance and the normalization
$\mathcal X_{-1,\delta}^{\eta}((x(\delta p^{-1}),1))=1$ give
$
\mathcal X_{-1,\delta}^{\eta}((x(\delta p^{-1}u^2),1))
=1$.
Thus, with $\operatorname{vol}(\ov{K_m})=1$, the Hecke action defined in
Section~\ref{sec:hecke-subalgebra} and
Lemma~\ref{lem:section5-affine-cosets} give
$$
\begin{aligned}
\mathcal X_{-1,\delta}^{\eta}v
=
\sum_{u\in\Z_p^\times/\sqrt{1+p\Z_p}}
\pi((x(\delta p^{-1}u^2),1))v.
\end{aligned}
$$
Since $(x(b),1)$ fixes $v$ for $b\in\Z_p$, each summand depends only
on $\delta u^2$ modulo $p$. The values $u^2$ run through the nonzero
squares modulo $p$, while $\lambda u^2$ runs through the nonsquares.
Subtracting the two sums therefore yields
$$
\mathcal R_m^{\eta}v
=
\sum_{a\in\mathbb F_p^*}
\left(\frac{a}{p}\right)
\pi((x(a/p),1))v.
$$
With the additive normalization $\operatorname{vol}(\Z_p)=1$ used in
\eqref{eq:R_pIshimoto}, comparison with the
preceding formula gives
$$
\mathcal R_m^{\eta}v=pR_p^\pi v,
\qquad v\in V_\eta^{K_m}(\pi).
$$

\subsection{From $R_p$ to $\mathcal W_{m-1}$}
Applying $\Theta_{m,*}$ gives
$$
\Theta_{m,*}(\mathcal R_m^{\eta})
=
\mathcal V_{m-1,-1}^{\eta^{J_m}}
-
\mathcal V_{m-1,-\lambda}^{\eta^{J_m}}.
$$
By the square-class dependence noted in Section~2, if
$\left(\frac{-1}{p}\right)=1$, then
$$
\mathcal V_{m-1,-1}^{\eta^{J_m}}
=\mathcal V_{m-1,1}^{\eta^{J_m}},
\qquad
\mathcal V_{m-1,-\lambda}^{\eta^{J_m}}
=\mathcal V_{m-1,\lambda}^{\eta^{J_m}},
$$
whereas if $\left(\frac{-1}{p}\right)=-1$, the two operators are
interchanged.  Consequently
$$
\Theta_{m,*}(\mathcal R_m^{\eta})
=
\left(\frac{-1}{p}\right)
\mathcal W_{m-1}^{\eta^{J_m}}.
$$
Combining this identity with the preceding comparison gives the following form of the relation with Ishimoto's operator.

\begin{prop}[Hecke algebra comparison with Ishimoto]\label{prop:section5-Rp-W}
Let $m\ge2$, and let $\eta$ be the genuine extension to $\ov{K_m}$ of either the trivial or the nontrivial quadratic character modulo $p$.  Put
$$
\mathcal R_m^{\eta}
=\mathcal X_{-1,1}^{\eta}-\mathcal X_{-1,\lambda}^{\eta}
\in\mathcal H_m^{\mathrm{full}}(\eta).
$$
Then
$$
\Theta_{m,*}(\mathcal R_m^{\eta})
=
\left(\frac{-1}{p}\right)
\mathcal W_{m-1}^{\eta^{J_m}}.
$$
Let $(\pi,V)$ be a genuine smooth representation of $\widetilde G_c$ and define
$\pi^{J_m}=\pi\circ\Theta_m^{-1}$.  The identity map on the underlying vector space $V$ identifies
$$
V_\eta^{K_m}(\pi)
=V_{\eta^{J_m}}^{K_m}(\pi^{J_m}).
$$
Under this identification,
$$
pR_p^\pi v
=
\left(\frac{-1}{p}\right)
\pi^{J_m}\!\left(\mathcal W_{m-1}^{\eta^{J_m}}\right)v,
\qquad v\in V_\eta^{K_m}(\pi).
$$
\end{prop}

\begin{proof}
The formula for $\Theta_{m,*}(\mathcal R_m^\eta)$ follows
from the preceding calculation, and the equality of the
corresponding spaces of fixed vectors follows from the
definitions. Since $\Theta_m$ preserves
$\ov{K_m}$, it preserves the Haar measure normalized by
$\operatorname{vol}(\ov{K_m})=1$. A change of variables in the Hecke action
therefore gives
$$
\pi^{J_m}(\Theta_{m,*}F)=\pi(F),
\qquad F\in\mathcal H_m^{\mathrm{full}}(\eta).
$$
Apply this to $F=\mathcal R_m^\eta$ and use
$\mathcal R_m^\eta v=pR_p^\pi v$ to obtain the asserted formula.
\end{proof}

For odd $m$, this compares the actions in the conjugated representation
and conjugated type; no identification with the original $(\pi,\eta)$-space
is asserted.

Assume that $m$ is even and set
$$
g_m:=p^{-m/2}J_m
=\begin{pmatrix}0&p^{-m/2}\\-p^{m/2}&0\end{pmatrix}
\in\SL_2(\Q_p),
$$
since $\det(J_m)=p^m$. Put $z=p^{m/2}$, so that $J_m=zg_m$.
In the $\sigma$-model, direct computation %
gives
$$
\widetilde z(g,\epsilon)=(g,\epsilon)\widetilde z,
\qquad \widetilde z=(zI,1),\quad
g\in\SL_2(\Q_p),\quad\epsilon\in\mu_2.
$$
Thus $\widetilde z$ centralizes $\widetilde G_{\sigma_0}$.
Since $\widetilde J_m$ and $\widetilde z(g_m,1)$ differ only by an
element of the central kernel, they induce the same conjugation, and hence
$$
\operatorname{Ad}_{\widetilde J_m}
=\operatorname{Ad}_{(g_m,1)}
\quad\text{on }\widetilde G_{\sigma_0}.
$$
Now put $\widetilde g_m=(g_m,1)\in\widetilde G_c$. The lower-right
entry of $g_m$ is zero, so $s_p(g_m)=1$ and
$\iota(\widetilde g_m)=(g_m,1)$. Consequently
$$
\Theta_m=\operatorname{Ad}_{\widetilde g_m}.
$$
In particular, $\pi^{J_m}\simeq\pi$: with
$A_m=\pi(\widetilde g_m)$, we have
$$
A_m\pi^{J_m}(\widetilde g)A_m^{-1}=\pi(\widetilde g),
\qquad \widetilde g\in\widetilde G_c.
$$
As $\eta^{J_m}=\eta$, the operator $A_m$ preserves $V_\eta^{K_m}(\pi)$.
Proposition~\ref{prop:section5-Rp-W} therefore gives
$$
pA_mR_p^\pi A_m^{-1}
=\left(\frac{-1}{p}\right)\pi(\mathcal W_{m-1}^{\eta})
\quad\text{on }V_\eta^{K_m}(\pi).
$$
The factor $p^{-m/2}$ merely rescales the conjugating matrix; it is not a
normalization of either $R_p$ or $\mathcal W_{m-1}$.

\end{document}